\documentclass[reqno,11pt]{amsart}
\usepackage{amsmath, latexsym, amsfonts, amssymb, amsthm, amscd}
\usepackage{mathrsfs}
\usepackage{graphics,epsf,psfrag}
\usepackage[lofdepth,lotdepth]{subfig}
\usepackage{graphicx}

\numberwithin{equation}{section}

\newtheorem{theorem}{Theorem}[section]
\newtheorem{lemma}[theorem]{Lemma}
\newtheorem{proposition}[theorem]{Proposition}
\newtheorem{cor}[theorem]{Corollary}
\newtheorem{rem}[theorem]{Remark}

\newcommand{\ind}{\mathbf{1}}

\renewcommand{\tilde}{\widetilde}

\newcommand{\cF}{{\ensuremath{\mathcal F}} }
\newcommand{\cG}{{\ensuremath{\mathcal G}} }

\newcommand{\cE}{{\ensuremath{\mathcal E}} }
\newcommand{\cH}{{\ensuremath{\mathcal H}} }
\newcommand{\cK}{{\ensuremath{\mathcal K}} }

\newcommand{\cM}{{\ensuremath{\mathcal M}} }
\newcommand{\cL}{{\ensuremath{\mathcal L}} }

\newcommand{\cU}{{\ensuremath{\mathcal U}} }
\newcommand{\cV}{{\ensuremath{\mathcal V}} }
\newcommand{\cR}{{\ensuremath{\mathcal R}} }

\DeclareMathSymbol{\leqslant}{\mathalpha}{AMSa}{"36} 
\DeclareMathSymbol{\geqslant}{\mathalpha}{AMSa}{"3E} 
\DeclareMathSymbol{\eset}{\mathalpha}{AMSb}{"3F}     
\renewcommand{\leq}{\;\leqslant\;}                   
\renewcommand{\geq}{\;\geqslant\;}                   
\newcommand{\dd}{\,\text{\rm d}}             

\newcommand{\bbE}{{\ensuremath{\mathbb E}} }

\newcommand{\bbL}{{\ensuremath{\mathbb L}} }

\newcommand{\bbN}{{\ensuremath{\mathbb N}} }

\newcommand{\bbP}{{\ensuremath{\mathbb P}} }

\newcommand{\bbR}{{\ensuremath{\mathbb R}} }
\newcommand{\bbS}{{\ensuremath{\mathbb S}} }

\newcommand{\bbZ}{{\ensuremath{\mathbb Z}} }

\newcommand{\ga}{\alpha}

\newcommand{\gd}{\delta}
\newcommand{\gep}{\varepsilon}       
\newcommand{\gp}{\varphi}

\newcommand{\gz}{\zeta}
\newcommand{\gG}{\Gamma}

\newcommand{\gD}{\Delta}

\newcommand{\gO}{\Omega}
\newcommand{\gl}{\lambda}

\newcommand{\gs}{\sigma}

\makeatletter
\def\captionfont@{\footnotesize}
\def\captionheadfont@{\scshape}

\long\def\@makecaption#1#2{%
  \vspace{2mm}
  \setbox\@tempboxa\vbox{\color@setgroup
    \advance\hsize-6pc\noindent
    \captionfont@\captionheadfont@#1\@xp\@ifnotempty\@xp
        {\@cdr#2\@nil}{.\captionfont@\upshape\enspace#2}%
    \unskip\kern-6pc\par
    \global\setbox\@ne\lastbox\color@endgroup}%
  \ifhbox\@ne 
    \setbox\@ne\hbox{\unhbox\@ne\unskip\unskip\unpenalty\unkern}%
  \fi
  \ifdim\wd\@tempboxa=\z@ 
    \setbox\@ne\hbox to\columnwidth{\hss\kern-6pc\box\@ne\hss}%
  \else 
    \setbox\@ne\vbox{\unvbox\@tempboxa\parskip\z@skip
        \noindent\unhbox\@ne\advance\hsize-6pc\par}%
\fi
  \ifnum\@tempcnta<64 
    \addvspace\abovecaptionskip
    \moveright 3pc\box\@ne
  \else 
    \moveright 3pc\box\@ne
    \nobreak
    \vskip\belowcaptionskip
  \fi
\relax
}
\makeatother
\def\writefig#1 #2 #3 {\rlap{\kern #1 truecm
\raise #2 truecm \hbox{#3}}}

\newcommand{\proj}{\ensuremath{\mathtt{p}}}

\begin{document}

\title[random long time dynamics for mean
field plane rotators]{Synchronization and  random long time dynamics  \\ for mean-field plane rotators}

\author{Lorenzo Bertini}
\address{Dipartimento di Matematica, Universit\`a di Roma La Sapienza
P.le A. Moro 2, 00185 Roma, Italy }

\author{Giambattista Giacomin}
\address{
  Universit\'e Paris Diderot, Sorbonne Paris Cit\'e,  Laboratoire de Probabilit{\'e}s et Mod\`eles Al\'eatoires, UMR 7599,
            F- 75205 Paris, France
}

\author{Christophe Poquet}
\address{
  Universit\'e Paris Diderot, Sorbonne Paris Cit\'e,  Laboratoire de Probabilit{\'e}s et Mod\`eles Al\'eatoires, UMR 7599,
            F- 75205 Paris, France
}

\date{\today}

\begin{abstract}
We consider the natural Langevin dynamics which is reversible with respect to the mean-field plane rotator (or classical
spin XY) measure. It is well known that this model exhibits a phase transition at a critical value of the interaction strength parameter $K$, in the limit of the number
$N$
of rotators going to infinity. A Fokker-Planck PDE   captures 
the evolution of the empirical measure of the system as $N \to \infty$,  at least for finite times and when the empirical measure of the system
at time zero satisfies a law of large numbers. The phase transition is reflected in the fact that the PDE for $K$ above the critical value
has several stationary solutions, notably a stable manifold -- in fact, a circle -- of stationary solutions that are equivalent up to  rotations. These stationary solutions are actually unimodal  densities parametrized by the position of their maximum (the synchronization phase or center). We  characterize the dynamics
on times of order $N$ and we show substantial deviations from the behavior of  the solutions of the  PDE. 
In fact, if the empirical measure at time zero
converges as $N \to \infty$ to a probability measure (which is away from a thin set that we characterize) and if time is speeded up by $N$,
the empirical measure   reaches almost instantaneously a small neighborhood of the stable manifold, to which it then sticks 
and on which  a non-trivial random dynamics takes place. In fact  the synchronization center performs a Brownian motion with 
a  diffusion coefficient that we compute. 
Our approach therefore provides, for one of the basic statistical mechanics systems with continuum symmetry,  a detailed characterization  of the macroscopic deviations from the large scale limit -- or law of large numbers -- due to finite size effects. But the interest for this  model
goes beyond statistical mechanics, since it plays a central role
 in a variety of scientific domains in which one aims at understanding synchronization phenomena.  
  \\[10pt]
  \\[10pt]
\end{abstract}

\maketitle

\section{Introduction}
\subsection{Overview}
In a variety of instances partial differential equations are a faithful approximation -- in fact, a law of large numbers --
 for particle systems in suitable limits. This is notably the case for stochastic interacting particle systems, for which the mathematical theory has gone very far \cite{cf:KL}. The closeness between the particle system and PDE  is typically proven in the limit of systems with a large number $N$ of particles or for infinite systems under a space rescaling involving a large parameter $N$ -- for example a spin or particle system on $\bbZ^d$ and the lattice spacing scaled down to $\frac 1N$ --
 and up to a time horizon which may depend on $N$. 
Of course the question of capturing the finite $N$ corrections has been taken up too, and the related
 central limit theorems as well as large deviations principles 
 have been established (see \cite{cf:KL} and references therein).
  {\sl Sizable} deviations  from
the law of large numbers, not just small fluctuations or rare events, can be observed beyond the time horizon 
for which the PDE behavior  has been established and these phenomena can be very relevant. 

The first examples that come to mind are the ones in which
the PDE has multiple isolated stable stationary points: metastability phenomena happens on exponentially long time scales \cite{cf:OV}.
Deviations on substantially shorter time scales can also take place and this is the case for example of the 
noise induced
 escape from stationary unstable solutions, which is particularly relevant in plenty of situations: for example for the model 
 in \cite[Ch.~5]{cf:Errico} 
  phase segregation originates from homogeneous initial data via this mechanism, on times proportional  to the logarithm
  of the size of the system.
 The logarithmic factor is directly tied to the exponential instability of the stationary solution
 (see   
 \cite{cf:Errico} for more literature on this phenomenon). Of course, the type of phenomena happen also in finite dimensional
 random dynamical systems, in the limit of small noise, but we restrict this quick discussion to  infinite dimensional models and PDEs.

In the case on which we focus
the  deviations also happen on time scales substantially shorter than the exponential ones, but the mechanism of the phenomenon
does not involve exponential instabilities.
 In the system we consider there  are multiple stationary solutions, but they are not (or, at least, not all) isolated, and hence they are not stable in the standard sense. Deviations from the PDE behavior happen  as a  direct 
result of the cumulative effect of the fluctuations. More precisely, this phenomenon is due to the presence of whole stable manifold of stationary solutions: the deterministic limit dynamics has no dumping effect along the tangential direction to the manifold so, for the finite size system, the weak noise does have a macroscopic  effect on a suitable time scale that depends on how large the system is.
We review the mathematical literature on this type of phenomena in \S~\ref{sec:review}, after stating our results.

Apart for the general interest on deviations from the PDE behavior, the model  we consider -- mean-field plane rotators -- 
is a fundamental one in mathematical physics and, more generally, it is the basic model for synchronization phenomena.
Our results provide  a sharp description of the long time dynamics of this model for general initial data.   



\subsection{The model} 
Consider the set of
ordinary stochastic differential equations
\begin{equation}
\label{eq:evol}
\dd \gp_t^{j,N}\, =\, \frac 1N \sum_{i=1}^N J \left( \gp_t^{j,N}-\gp_t^{i,N}\right) \dd t + \dd W^j_t\, .
\end{equation}
with
$j=1, 2, \ldots , N$,
$\{W_j\}_{j=1,2 ,\ldots}$ is an IID collection of standard Brownian motions and 
 $J(\cdot)= -K \sin (\cdot)$.
With abuse of notation, when writing $\gp_t^{j,N}$
we will actually mean $\gp_t^{j,N}\text{mod}(2\pi)$ and 
for us \eqref{eq:evol}, supplemented with an (arbitrary) initial condition,
will give origin to a diffusion process on $\bbS^N$, where
$\bbS$ is the circle  $\bbR /(2\pi \bbZ)$.

The choice of the interaction potential $J(\cdot)$ is such that
the (unique) invariant probability of the system is 
\begin{equation}
\label{eq:Gibbs}
\pi_{N, K}(\dd \gp)
\propto \exp\left( \frac{K}{N} \sum_{i,j=1}^N \cos( \gp_i-\gp_j)\right) \gl_N(\dd \gp)\, ,
\end{equation}
where $\gl_N$ is the uniform probability measure on $\bbS^N$. Moreover, the evolution
is reversible with respect to $\pi_{N, K}$, which is the well known Gibbs measure
associated to mean-field plane rotators (or classical $XY$ model).

We are therefore considering the simplest Langevin dynamics of mean-field plane rotators
and it is well known that such a model exhibits a phase transition, for $K>K_c:=1$, that breaks the continuum symmetry of the model (for a detailed mathematical physics literature we refer
to \cite{cf:BGP}). The continuum symmetry of the model is evident both
in the dynamics \eqref{eq:evol} and in the equilibrium measure \eqref{eq:Gibbs}:
 if $\{\gp_t^{j,N}\}_{t\ge 0, j=1,\ldots, N}$ solves
 \eqref{eq:evol}, so does $\{\gp_t^{j,N}+c\}_{t\ge 0, j=1,\ldots, N}$, $c$ an arbitrary constant,
 and $\pi_{N,K} \Theta_c^{-1}= \pi_{N,K}$, where $\Theta_c$ is the rotation by an angle $c$, that is  $(\Theta_c \gp)_j= \gp_j+c$ for every $j$.

\subsection{The $N \to \infty$ dynamics and the stationary states}
The phase transition can be  understood also taking a dynamical standpoint.
Given the mean-field set up it turns out to be particularly convenient to consider
the empirical  measure
\begin{equation}
\label{eq:empm}
\mu_{N,t}( \dd \theta)\, :=\, \frac 1N \sum_{j=1}^N \gd_{\gp_t^{j,N}}(\dd \theta)\, ,
\end{equation}
which is a probability on (the Borel subsets of) $\bbS$.
It is well known, see \cite{cf:BGP} (for detailed treatment and  original references), that if  $\mu_{N,0}$ converges weakly for $N \to \infty$,
then so does $\mu_{N,t}$ for every $t>0$. Actually, the process itself
$t \mapsto \{\mu_{N, t}\}$, seen as an element of $C^0 ([0,T], \cM_1)$, where $T>0$  and $\cM_1$ is the space of probability measures on $\bbS$ equipped with the weak topology, converges to a 
non-random limit which is the process that concentrates on  
the unique solution of the non-local PDE ($*$ denotes the convolution)
\begin{equation}
\label{eq:K}
\partial_t p_t(\theta) \, =\, \frac 12 \partial_\theta ^2 p_t( \theta) - \partial_\theta \big( (J*p_t)(\theta) p_t(\theta)
\big), 
\end{equation}
with initial condition prescribed by the limit of $\{\mu_{N, 0}\}_{N=1,2, \ldots}$. If such a 
limit probability does not have a ($C^2$) density (with respect to the uniform measure),
one has to interpret \eqref{eq:K} in a weak sense, but actually, even if the initial datum 
is just in $\cM_1$, that is if it does 
not admit a density or if such a density is not smooth, the probability measure that solves \eqref{eq:K} has a density $p_t(\cdot)\in C^\infty$ for every
$t>0$, see 
 \cite{cf:GPP}.
We insist on the fact that  $p_t(\cdot)$ is a probability density:
$\int_\bbS p_t(\theta) \dd \theta=1$. We will often commit the abuse of notation of writing
$p(\theta)$ when $p\in \cM_1$ and $p$ has a density. Much in the same way,  if $p(\cdot)$ is a probability density,
$p$, or $p(\dd \theta)$, is the probability measure. 

It is worthwhile to point out  that $ (J*p)(\theta)= -\Re (\hat p_1) K \sin (\theta) + \Im (\hat p_1) K \cos(\theta)$
with $\hat p_1:= \int _\bbS p(\theta)\exp(i \theta)Ê\dd\theta$. 
This is to say that the nonlinearity enters only through the first Fourier coefficient
of the solution, a peculiarity that allows to go rather far in the analysis 
of the model. Notably, starting from this observation one can easily (once again
details and references are given in \cite{cf:BGP}) see that all the stationary solutions
to \eqref{eq:K}, in the class of probability densities, can be written, up to a rotation, as
\begin{equation}
\label{eq:q}
q(\theta)\,:=\, \frac{\exp(2Kr \cos (\theta)}{2\pi I_0(2Kr)}\, ,
\end{equation}
where $2\pi I_0(2Kr)$ is the normalization constant written in terms of the modified 
Bessel function of order zero ($I_j(x)= (2\pi)^{-1}\int_\bbS (\cos \theta)^j \exp(x \cos (\theta)) \dd \theta$, for $j=0,1$) and $r$ is a non-negative solution of
the fixed point equation $r= \Psi (2Kr)$, with $\Psi(x)=I_1(x)/I_0(x)$. 
Since $\Psi(\cdot):[0, \infty) \to [0,1)$ is increasing, concave, $\Psi(0)=0$ and $\Psi'(0)=1/2$
we readily see that if (and only if) $K>1$ there exists a non-trivial (i.e. non-constant)
solution to \eqref{eq:K}. Let us not forget however that $\Psi(0)=0$ implies that $r=0$ is a solution and therefore the constant density $\frac 1{2\pi}$ is a solution no matter what the value of $K$ is.  From now on we set $K>1$ and choose $r=r(K)$, the unique 
positive solution of the fixed point equation, so that the probability 
density $q(\cdot)$ in \eqref{eq:q} is non trivial and it achieves the unique maximum at $0$
and the minimum at $\pi$. Note that the rotation invariance of the system
immediately yields that there is a whole family of stationary solution:
\begin{equation}
\label{eq:M}
M\, =\, \{ q_\psi (\cdot):\, q_\psi (\cdot):= q(\cdot- \psi)
\text{ and } \psi \in \bbS\}\,,
\end{equation}
and, when $x \in \bbR$, $q_x(\cdot)$  of course means $q_{x\text{mod}(2\pi)}(\cdot)$.
$M$, which is more practically viewed as a manifold (in a suitable function space, see  \S~\ref{sec:Manif} below), is invariant and stable for the evolution. The proper notion of stability is
given in the context of {\sl normally hyperbolic manifolds} (see \cite{cf:SellYou} and references therein), but the full power of such a
concept is not needed for the remainder. Nevertheless  let us stress that in \cite{cf:GPP} one can find
a complete analysis of the global dynamic phase diagram, notably the fact that
unless $p_0(\cdot)$ belongs to the stable manifold $U$ of the unstable solution $\frac1{2\pi}$ --
the solution corresponding to $r=0$ in \eqref{eq:q} --
 $p_t(\cdot)$ converges (also in strong norms, controlling all the derivatives) to one of the points in $M$, see Figure~\ref{fig:1}. 
There is actually an explicit characterization of $U$:
\begin{equation}
\label{eq:Umanif}
U\, =\, \left \{ p\in \cM_1: \,  \int_\bbS  \exp(i \theta) p(\dd \theta)\,  =0\right\}\, .
\end{equation}
As a matter of fact, it is easy to realize that if $p_0(\cdot) \in U$ then \eqref{eq:K} reduces to the heat equation
$\partial_t p_t(\theta)= \frac 12 \partial_\theta^2 p_t (\theta)$ which of course relaxes to $\frac1{2\pi}$.

\begin{figure}
\begin{center}
\leavevmode
\epsfxsize =14.5 cm
\psfragscanon 
\psfrag{0}[l][l]{\small $0$}
\psfrag{p0}[l][l]{\small $p_0$}
\psfrag{pin}[l][l]{\small $p_\infty$}
\psfrag{U}[l][l]{\small $U$}
\psfrag{M0}[l][l]{\small $M$}
\psfrag{2pi}[l][l]{\small $\frac 1{2\pi}$}
\psfrag{th}[l][l]{\small $\theta$}
\psfrag{pi2}[l][l]{\small $2\pi$}
\psfrag{t=0}[l][l]{\small $t=0$}
\psfrag{tlarge}[l][l]{\small $t\to \infty$}
\psfrag{psi}[l][l]{\small $\psi$}
\epsfbox{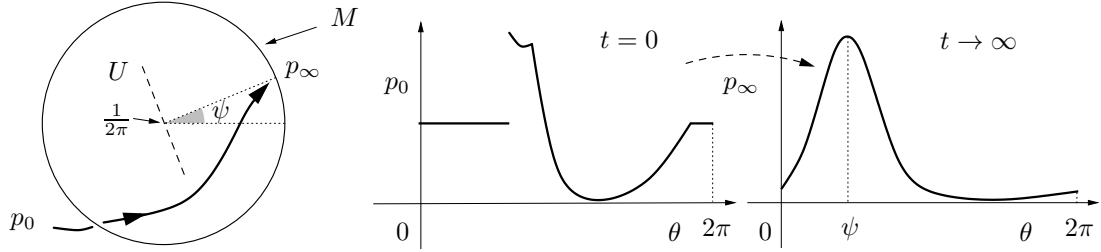}
\end{center}
\caption{\label{fig:1}
The evolution limit evolution \eqref{eq:K} instantaneously smoothens an arbitrary initial probability and, unless
the Fourier decomposition such an initial condition has zero coefficients corresponding to the first harmonics (the hyperplane $U$), it drives it
to a point $p_\infty$ -- a synchronized profile -- on the invariant manifold $M$ and of course it stays there for all times. This has been proven in \cite{cf:GPP}, here we are interested in what happens for the finite size -- $N$ -- system and we show that the PDE approximation is faithful up to times
much shorter than $N$: on times proportional to $N$ synchronization is kept and the center of synchronization $\psi$  performs a Brownian motion on $\bbS$.
}
\end{figure}


\subsection{Random dynamics on $M$: the main result}
In spite of the stability of $M$, $q_\psi(\cdot)$ itself  is not stable, simply because if we start nearby, say from
$q_{\psi'}$, the solution of $\eqref{eq:K}$ does not converge to $q_\psi(\cdot)$. 
The important point here is that the linearized evolution operator around $q(\cdot) \in M$ ($q$ is an arbitrary element of $M$,
not necessarily the one in \eqref{eq:q}: the phase $\psi$ of $q_\psi$ is explicit only when its absence may be misleading)
\begin{equation}
\label{eq:Lq}
L_q u (\theta)\,:=\, \frac 12  u'' -[ u J*q+ q J*u]'\, ,
\end{equation}
with domain $\{u\in C^2(\bbS, \bbR):\, \int_\bbS u =0\}$ is symmetric in 
$H_{-1,1/q}$ -- a weighted $H_{-1}$ Hilbert space that we introduce in detail
in Section~\ref{sec:linear} -- and it has compact resolvent.
Moreover the  
spectrum of $L_q$, which is of course discrete, lies  in $(-\infty, 0]$ and the eigenvalue $0$
has a one dimensional eigenspace, generated by $q'$. So $q'$ is the only {\sl neutral direction}
and it corresponds precisely to the tangent space of $M$ at $q(\cdot)$: all other directions, in function
space, are contracted by the linear evolution and the nonlinear part of the evolution
does not alter substantially this fact \cite{cf:GPP,cf:Henry}.   

Let us now step back and recall that our main concern is  with the behavior of \eqref{eq:evol}, with $N$ large but finite, 
and not \eqref{eq:K}. In a sense  the finite size, i.e. finite $N$, system 
is close to a suitable stochastic perturbation of \eqref{eq:K}: the type of
stochastic PDE, with noise vanishing as $N \to \infty$, needs to be carefully {\sl guessed} \cite{cf:GLP}, keeping
in particular in mind that we are dealing with a system with one conservation law.
We will tackle directly \eqref{eq:evol}, but the heuristic picture that one obtains by
thinking of an SPDE with vanishing noise is of help. 
 In fact the considerations we have just made on $L_q$ suggest 
 that if one starts the SPDE on $M$, the solution keeps  very close to $M$,  since the deterministic part of the dynamics is contractive in the orthogonal directions to $M$, but a (slow, since the noise is small) random motion on $M$ arises because
 in the tangential direction the deterministic part of the dynamics is {\sl neutral}. This is indeed what happens for the model we consider for $N$ large.
The difficulty that arises in dealing with the interacting diffusion system  \eqref{eq:evol}
is that  one has to work with \eqref{eq:empm}, which is not a function.
Of course one can mollify it, but the evolution is naturally written and, to a certain extent, {\sl closed} 
in  terms
of the empirical measure, and we do not believe that any 
significative simplification arises in proving our
main statement for a mollified version. Working with the empirical measure imposes a clarification from now: as we explain in Section~\ref{sec:linear} and Appendix~\ref{sec:A_H}, if $\mu$ and $\nu\in \cM_1$, then 
$\mu-\nu$ can be seen as an element of $H_{-1}$ (or, as a matter of fact, also as an element of a weighted 
$H_{-1}$ space).
\medskip

Here is the main result that we prove (recall that $K>1$):
\medskip

\begin{theorem}
\label{th:main}
Choose a positive constant $\tau_f$ and a probability   $p_0\in \cM_1 \setminus   U$. If for every $\gep>0$
\begin{equation}
\label{eq:init-main}
\lim_{N \to \infty}
\bbP \left(
\left \Vert \mu_{N, 0} - p_0 \right \Vert_{-1}\le \gep \right) \, =\, 1\, ,
\end{equation}
then there exist a constant $\psi_0$ that depends only on $p_0(\cdot)$ and, for every $N$, a continuous process 
$\{W_{N,\tau}\}_{\tau \ge 0}$, adapted to  the natural filtration of $\{W^j_{N\cdot }\}_{j=1,2, \ldots,N}$, 
such that $W_{N, \cdot}\in C^0([0,\tau_f]; \bbR)$ converges weakly to a standard Brownian motion and 
 for every $\gep>0$
\begin{equation}
\label{eq:main}
\lim_{N \to \infty}
\bbP \left( 
\sup_{\tau \in [\gep_N, \tau_f]} 
\left\Vert \mu_{N, \tau N} - q_{\psi_0 + D_K W_{N,\tau} } \right \Vert_{-1}\le \gep \right) \, =\, 1\, ,
\end{equation}
where $\gep_N := C /N$, $C=C(K, p_0, \gep)>0$,  and  
\begin{equation}
\label{eq:DK}
D_K\, :=\,  \frac1{\sqrt{1- \left(I_0(2Kr)\right)^{-2}}}\, .
\end{equation}
\end{theorem}

\medskip

The result is saying that, unless one starts on the stable manifold of the unstable solution (see Remark~\ref{rem:fromU}
for what one expects if $p_0\in U$),
the empirical measure reaches very quickly a small neighborhood of the  manifold $M$: this happens
on a time scale of order one, as  a consequence of the properties of the deterministic 
evolution law \eqref{eq:K} (Figure~\ref{fig:1}), and, since we are looking at times of order $N$,
this happens almost instantaneously. Actually, in spite of the fact that the result just addresses
the limit of the empirical measure, the drift along $M$ is due to fluctuations: the noise 
pushes the empirical measure away from $M$ but the deterministic part of the dynamics 
{\sl projects back} the trajectory to $M$ and the net effect of the noise is a random 
shift -- in fact, a rotation -- along the manifold (this is taken up 
in more detail in the next section, where we give a
complete heuristic version  of the proof of Theorem~\ref{th:main}).



\begin{rem}
Without much effort, one can upgrade this result to much longer times: if we set  $\tau_f(N)=N^a$ with an arbitrary $a>1$,
there exists an adapted process $W^a_{N,\tau}$ converging to a standard Brownian motion such that
\begin{equation}
\lim_{N \to \infty}
\bbP \left( 
\sup_{\tau \in [\gep_N, \tau_f(N)]} 
\left\Vert \mu_{N, \tau N^a} - q_{\psi_0 + D_K N^{a-1}W^a_{N,\tau} } \right \Vert_{-1}\le \gep \right) \, =\, 1\, .
\end{equation}
This is due to the fact that our estimates ultimately rely on moment estimates, cf. Section~\ref{sec:apriori}. 
These estimates are obtained for arbitrary moments and we choose the moment sufficiently large
to get uniformity for times $O(N)$, but working for times $O(N^a)$ would just require choosing larger moments.
We have preferred to focus on the case $a=1$ 
this is the natural scale, 
that is the scale  in which the center of the probability density converges to a Brownian motion and not to an ``accelerated" Brownian motion
(this is really due to the fact that we work on $\bbS$ and marks a difference with \cite{cf:BBDMP,cf:BBB} where one
can rescale the space variable). 
\end{rem}

\subsection{The synchronization phenomena viewpoint}
\label{sec:sync}
The model \eqref{eq:evol} we consider is actually a particular case of the Kuramoto synchronization model
(the full Kuramoto model includes {\sl quenched disorder} in terms of random constant speeds for the rotators, see 
\cite{cf:acebron,cf:BGP} and references therein). The mathematical physics literature and the more  bio-physically oriented literature 
use somewhat different notations reflecting a slightly different viewpoint. In the synchronization literature one introduces
the synchronization degree $\boldsymbol r_{N,t}$ and the synchronization center $\boldsymbol\Psi_{N,t} $ via
\begin{equation}
\boldsymbol r_{N,t} \exp(i \boldsymbol \Psi_{N, t})\, :=\, \frac 1N \sum_{j=1}^N
\exp(i \gp^{j,N}_t) \, \left( \,= \,
\int_\bbS \exp(i \theta) \mu_{N, t} (\dd \theta)\right)\, ,
\end{equation}
which clearly correspond to the parameters $r$ and $\psi$ that appear in the definition of $M$, but 
$\boldsymbol r_{N,t}$ and  $\boldsymbol\Psi_{N,t} $ are defined for $N$ finite and also far from $M$.
Note that if \eqref{eq:init-main} holds, then both $\boldsymbol r_{N,t}$ and  $\boldsymbol\Psi_{N,t} $ converge in probability as $N \to \infty$
to the limits $r$ and $\psi$, with $r\exp(i\psi)= \int_\bbS \exp(i \theta) p_0(\dd \theta)$ and the assumption that
$p_0 \not\in U$ just means $r\not= 0$.
Here is a straightforward consequence of Theorem~\ref{th:main}:
\medskip

\begin{cor}
\label{th:cor}
Under the same hypotheses and definitions as in  Theorem~\ref{th:main} we have that
the stochastic process $ \boldsymbol \Psi_{N, N\cdot} \in C^0([\gep, \tau_f]; \bbS)$ 
converges weakly, for every $\gep\in (0, \tau_f]$,  to 
$(\psi_0 + D_K W_{\cdot})\mathrm{mod} (2\pi)$.
\end{cor}
\medskip

It is tempting to prove such a result by looking directly 
at the evolution of $\boldsymbol\Psi_{N,t}$:
\begin{multline}
\label{eq:altern}
\dd \boldsymbol\Psi_{N,t} \, =\, 
\left(- K + \frac 1{2N \boldsymbol r^2_{N,t}} \right) \frac 1N 
\sum_{j=1}^N \sin( 2( \gp^{j,N}_t- \boldsymbol \Psi_{N,t})) \dd t \\
+ \frac 1{ \boldsymbol  r_{N,t} N} \sum_{j=1}^N \cos(\gp^{j,N}_t- \boldsymbol\Psi_{N,t}) \dd W_j (t)\, . 
\end{multline}
But this clearly requires a control of the evolution of the empirical measure, so it does not
seem that \eqref{eq:altern} could provide an alternative way to many of the estimates that we develop, namely convergence to a neighborhood of $M$ and persistence of the proximity to $M$ (see Section~\ref{sec:apriori} and Section~\ref{sec:approachM}).
On the other hand, it seems plausible that one could use \eqref{eq:altern} to develop an alternative 
approach to the dynamics on $M$, that is an alternative to Section~\ref{sec:dyntan}. While this can be
interesting in its own right, since the notion of synchronization center   that we use in the proof and $\boldsymbol\Psi_{N, t}$ 
are almost identical (where they are both defined, that is close to $M$) we do not expect substantial simplifications. 

\subsection{A look at the literature and perspectives}
\label{sec:review}
Results related to our work have been obtained in the context of SPDE models with
vanishing noise. In \cite{cf:BDMP,cf:Funaki} one dimensional stochastic reaction diffusion equations with bistable potential (also called
{\sl stochastic Cahn-Allen} or {\sl model A}) are analyzed for initial data that are close to profiles that connect the two phases.
It is shown
that the location of the phase boundary performs a Brownian motion. These  results have been improved in a number of ways,
notably to include {\sl small asymmetries}   that result in a drift for the arising diffusion process \cite{cf:BBjsp} and to deal with macroscopically finite volumes \cite{cf:BBB}
(which introduce a repulsive effect approaching the boundary). Also the case of stochastic phase field equations has been considered 
\cite{cf:BBBP}.

For interacting particle  systems results  
 have been obtained  for the zero temperature limit of $d$-dimensional  Brownian particles interacting via local pair  potentials
 in \cite{cf:Funaki2}: in this case
the {\sl frozen clusters}  perform a Brownian motion and, in one dimension, also the merging of clusters is analyzed \cite{cf:Funaki3}. In this
case the very small temperature is the small noise from which cluster diffusion originates. With respect to
 \cite{cf:Funaki2,cf:Funaki3}, our results hold for any super-critical interaction, but of course our system is of mean field type.
 It is also interesting to observe that  for the model
in \cite{cf:Funaki2,cf:Funaki3} establishing the stability of the frozen clusters is the crucial issue,
because the motion of the center of mass is a martingale, i.e. there is no drift. A substantial part
of our work is in controlling that the drift of the center of synchronization  vanishes (and controlling the drift
is a substantial part also of \cite{cf:BDMP,cf:Funaki,cf:BBjsp,cf:BBB,cf:BBBP}). This is directly related to the content
of \S~\ref{sec:sync}. 

As a matter of fact,
in spite of the fact that our work deals directly with an interacting system,
 and not with an SPDE model, our approach is closer to the one in the SPDE literature. However,
 as we have already pointed out, a non negligible point is that
we are forced to perform an analysis in distribution spaces, in fact Sobolev spaces with negative exponent, in contrast to
the approach in the space of continuous functions in    \cite{cf:BDMP,cf:Funaki,cf:BBjsp,cf:BBB,cf:BBBP}.
 We point out that 
approaches to dynamical mean field type systems via Hilbert spaces of distribution has been already taken up in \cite{cf:FM}
but in our case the specific use of weighted Sobolev spaces is not only a technical tool, but it
is intimately related to the geometry of the contractive invariant manifold $M$.
In this sense and because of the iterative procedure we apply -- originally introduced in \cite{cf:BDMP} --
 our work is a natural development  of \cite{cf:BDMP,cf:BBB}.

\medskip

An important issue about our model that we have not stressed at all  is that propagation of chaos holds (see e.g.~\cite{cf:Gartner}), in the sense that
if the initial condition is given by a product measure, then this property is approximately preserved, at least for finite times.
Recently much work has been done toward establishing   quantitative estimates of chaos propagation
(see for example the references in \cite{cf:CDW}). On the other hand,  like for the model  in 
\cite{cf:CDW}, we know that, for our model,  chaos propagation eventually breaks down: this is just because 
one can  show by Large Deviations arguments
that the empirical measure at equilibrium converges in law as $N \to \infty$  to the random probability density $q_{X}(\cdot)$, with $X$ a uniform random variable on $\bbS$. 
But using Theorem~\ref{th:main} one can go much farther and 
show  that  chaos propagation breaks down at times proportional to $N$. From  Theorem~\ref{th:main}  one can actually extract  also  an accurate
description of how the correlations  build up due to the random motion on $M$.

\medskip

It is natural to ask whether the type of results we have proven extend to 
the case in which random natural frequencies are present, that is to the disordered version
of the model we consider that goes under the name of Kuramoto model. The question is natural because for the limit PDE
\cite{cf:dPdH,cf:eric}
 there is a 
contractive manifold similar to $M$ \cite{cf:GLPdis}. However the results in
  \cite{cf:eric2} suggest that a nontrivial dynamics on the contractive manifold is observed rather
  on times proportional to $\sqrt{N}$ and one expects a dynamics with a nontrivial random drift.
 The role of disorder in this type of models is not fully elucidated  (see however
  \cite{cf:CdP} on the critical case) and the global long time dynamics  represents  a challenging issue.

\medskip

The paper is organized as follows: we start off (Section~\ref{sec:more}) by introducing the precise mathematical set-up and a number
of technical results. This will allow us to present  quantitative heuristic arguments and sketch of poofs.
In Section~\ref{sec:apriori} we prove that if the system is close to $M$, it stays so for a long time.
We then move on to analyzing the dynamics on $M$ (Section
\ref{sec:dyntan}) and it is here that we show that the drift is negligible.
Section~\ref{sec:approachM} provides the estimates that guarantee that we do approach $M$
and in Section~\ref{sec:proofmain} 
we collect all these estimates and complete the proof our main result (Theorem~\ref{th:main}).

\section{More on the mathematical set-up and sketch of proofs}
\label{sec:more}

\subsection{On the linearized evolution}
\label{sec:linear}
We introduce the Hilbert space $H_{-1,1/q}$ or, more generally, the space $H_{-1,w}$ for
a general weight $w\in C^1(\bbS; (0, \infty))$ by using the rigged Hilbert space structure \cite{cf:Brezis}
with pivot space  $\bbL^2_0:= \{ u\in \bbL^2: \, \int_\bbS u =0\}$.
In this way given an Hilbert space $V\subset \bbL^2_0$, $V$ dense in $\bbL^2_0$, 
for which the canonical injection of $V$ into $\bbL^2_0$  is continuous, one automatically 
obtains a representation of $V'$ -- the dual space -- in terms of a third Hilbert space into which $\bbL^2_0$ is canonically and densely injected. If  $V$ is the closure of $\{u \in C^1(\bbS; \bbR):\, \int u =0\}$
under the squared norm
$\int_\bbS (u')^2 / w$, that is $H_{1, 1/w}$, the third Hilbert space is precisely
$H_{-1,w}$. The duality between $H_{1, 1/w}$ and $H_{-1,w}$ is denoted in principle by
$\langle\, \cdot\, , \, \cdot\,\rangle_{H_{1, 1/w},H_{-1,w}}$, but less cumbersome notations
will be introduced when the duality is needed (for example, below we drop the subscripts).  

It is not difficult to see that for $u, v \in H_{-1,w}$
\begin{equation}
\left( u, v \right)_{-1,w}\, =\, \int_\bbS  w\,  \cU \cV\, ,
\end{equation}
where $\cU$, respectively $\cV$, is the primitive of $u$ (resp. $v$) such that $\int_\bbS w \cU =0$
(resp. $\int_\bbS w \cV =0$), see \cite[\S~2.2]{cf:BGP}. More precisely,  $u\in H_{-1,w}$ if there
exists $\cU \in \bbL^2(\bbS; \bbR)$ such that $\int_\bbS \cU w =0$ and 
  $\langle u, h\rangle=-\int_\bbS \cU h'$ for every $h \in H_{1,1/w}$. 
One sees directly also that by changing $w$ one produces equivalent $H_{1, w}$ norms
 \cite[\S2.1]{cf:GPPP} so, when the geometry of the Hilbert space is not crucial, one can simply replace the weight by $1$, and in this case we simply write $H_{-1}$. Occasionally
 we will need also $H_{-2}$ which is introduced in an absolutely analogous way. 
 \medskip

 \begin{rem}\label{rem:norecenter}
 One observation that is of help in estimating weighted $H_{-1}$ norms is that
 computing the norm of $u$ requires access to $\cU$: 
 in practice if one identifies
 a primitive $\tilde \cU$ of $u$, then $\Vert u\Vert _{-1, w}^2 \le \int_\bbS \tilde \cU ^2 w$.
 This is just because $\tilde \cU = \cU +c$ for some $c \in \bbR$ 
 and $\int_\bbS \tilde \cU ^2 w= \int_\bbS  \cU ^2 w + c^2 \int_\bbS w$.
 \end{rem}
  \medskip
  
The reason for introducing weighted $H_{-1}$ spaces is because, as one can readily verify,
$L_q$, given in \eqref{eq:Lq}, is symmetric in $H_{-1,1/q}$. A deeper
analysis (cf. \cite{cf:BGP}) shows that $L_q$ is essentially self-adjoint, with
compact resolvent. The spectrum of $-L_q$ lies in $[0, \infty)$, there is
an eigenvalue $\gl_0=0$ with one dimensional eigenspace generated by $q'$.
We therefore denote the set of eigenvalues of $-L_q$
as $\{\gl_0, \gl_1, \ldots \}$, with $\gl_1>0$ and $\gl_{j+1} \ge \gl_j$ for $j=1,2, \ldots$.
The set of eigenfunctions is denoted by $\{e_j\}_{j=0,1, \ldots}$ and let us
point out that it is straightforward to see that $e_j \in C^\infty (\bbS; \bbR)$.
Moreover, if $u \in C^2(\bbS; \bbR)$ 
is even (respectively, odd),  then 
 $L_q u$ is even   (respectively, odd): the notion of parity is of course the one obtained by observing that 
 $u \in C^2(\bbS; \bbR)$ can be  extended to a periodic function in $C^2(\bbR; \bbR)$. This implies that
 one can choose $\{e_j\}_{j=0,1, \ldots}$ with $e_j$ that is  either even or odd, and we will do so. 
\medskip

\begin{rem}
\label{rem:Lpsi}
By rotation symmetry  the eigenvalues do not depend on the choice of
$q(\cdot)\in M$, but the eigenfunctions do depend on it, even if in a rather
trivial way: the eigenfunction of
$L_{q_\psi}$ and $L_{q_{\psi'}}$ just differ by a rotation
of $\psi'-\psi$. We will often need to be precise about the choice of $q(\cdot)$
and for this it is worthwhile to introduce the notations
\begin{equation}
L_\psi \, :=\, L_{q_{\psi}} \ \text{ and } \  -L_\psi e_{\psi, j}\, =\, \gl_j  e_{\psi, j}\, .
\end{equation}
The eigenfunctions are normalized in $H_{-1, 1/q_{\psi}}$. 
\end{rem}

\medskip

\begin{rem}
\label{rem:compute}
Some expressions involving weighted $H_{-1}$ norms can be worked out explicitly.
For example a recurrent expression in what follows is $(u, q')_{1, 1/q}$, for $u \in H_{-1}$
and $q \in M$. If $\cU$ is the primitive of $u$ such that $\int_\bbS \cU /q =0$, then
we have $(u, q')_{1, 1/q}= \int_\bbS \cU (q-c)/q= \int_\bbS \cU$, where $c$ is uniquely defined
by $\int_\bbS (q-c)/q=0$, but of course the explicit value of $c$ is not used in the final expression. 
In practice however it may be more straightforward to use an arbitrary primitive $\tilde \cU$ of $u$ (i.e. $\int_\bbS \tilde \cU /q$ is not necessarily zero) for which we have
\begin{equation}
(u, q')_{1, 1/q}\, =\, 
\int_\bbS \tilde \cU \left(1 - \frac c q\right)\, .
\end{equation}
Since now $c$ appears, let us make it explicit:
\begin{equation}
c\, = \, \frac {2\pi}{\int_\bbS 1/q}\, =\,
\frac 1{2\pi I_0^2(2Kr)}\,.
\end{equation}
\end{rem}


\subsection{About the manifold $M$}
\label{sec:Manif}

As we have anticipated, we look at the set of stationary solutions $M$, defined in
\eqref{sec:Manif}, as a manifold. For this we introduce 
\begin{equation}
\label{eq:Htilde}
\tilde H_{-1}\, :=\, \left\{ \mu:\, \mu - \frac 1{2\pi } \in H_{-1}\right\}\, ,
\end{equation}
which is a metric space equipped with the distance inherited from $H_{-1}$, that is
$\textrm{dist}(\mu_1,\mu_2)= \Vert \mu_1-\mu_2 \Vert_{-1}$.
 We have $M \subset \tilde H_{-1}$ and $M$ can be viewed as
a smooth one dimensional manifold in  $\tilde H_{-1}$. The tangent space at $q \in M$ is $q' \bbR$
and for every $u \in H_{-1}$ we define the projection  $P^o_q$ on this tangent space
as $P^o_q u= (u, q')_{-1,1/q}q'/(q',q')_{-1,1/q}$. 
The following result is proven in \cite[p. 501]{cf:SellYou}
(see also \cite[Lemma 5.1]{cf:GPPP}): 
\medskip

\begin{lemma}
\label{lem:def proj on M}
There exists $\gs>0$ such that for all $p\in N_\gs$ with
\begin{equation}
\label{eq:Ngs}
N_\gs\,:=\, \cup_{q \in M}\left\{ \mu \in  \tilde H_{-1} : \, \Vert \mu -q \Vert_{-1} < \gs \right\}\, ,
\end{equation}
there is one and only one $q=:v( \mu )\in M$ such that $(\mu-q,q')_{-1,1/q}=0$. Furthermore, 
the mapping $\mu\mapsto v( \mu )$ 
is in $C^\infty(\tilde H_{-1},\tilde H_{-1})$, and (with $D$ the Fr\'echet derivative) 
\begin{equation}
 Dv( \mu )\, =\, P^o_{v( \mu )}\, .
\end{equation}
\end{lemma}

\medskip

Note that the empirical (probability) measure $\mu_{N,t}$ that describes  our system at time $t$  is in $\tilde H_{-1}$ (see Appendix~\ref{sec:A_H}) and 
Lemma~\ref{lem:def proj on M} guarantees in particular that 
as soon as it
is sufficiently 
close to $M$  there is a well defined projection $v\left(\mu_{N,t}\right)$ on the manifold. Since the manifold is 
isomorphic to $\bbS$ it is practical to introduce, for $\mu \in \tilde H_{-1}$, also $\proj (\mu)\in \bbS$,
uniquely defined by $v(\mu)=q_{\proj(\mu)}$.
It is immediate to see that the projection $\proj$ is $C^\infty(\tilde H_{-1},\bbS)$.

\subsection{A quantitative heuristic analysis: the diffusion coefficient}
\label{sec:heur}

The proof of
Theorem~\ref{th:main} is naturally split into two parts: the approach to $M$ and the motion on
$M$. The approach to $M$ is based on the properties of the PDE \eqref{eq:K}:  in \cite{cf:GPP}
it is shown, using the gradient flow structure of \eqref{eq:K}, that if the initial condition
is not on the stable manifold $U$ (see \eqref{eq:Umanif}) of the unstable stationary solution $\frac 1{2\pi}$,
then the solution converges for time going to infinity to one of the probability densities $q=q_\psi\in M$ 
(of course $\psi$ is a function of the initial condition), so given a neighborhood of $q_\psi$ after
a finite time (how large it depends only on the initial condition), it gets to the chosen neighborhood: due to the regularizing
properties of the PDE, such a neighborhood can be even in a topology that controls all the derivatives \cite{cf:GPP}, but
here there is no point to use a strong topology, since at the level of interacting diffusions we deal with a measure (that
we inject into $H_{-1}$). And in fact  
we have to estimate the distance between the empirical measure and the solution to \eqref{eq:K} --
controlling thus the effect of the noise --
 but this type of estimates 
on finite time intervals is standard. However here there is a subtle point: the result we are after is a matter of fluctuations and
it will not come as a surprise that the empirical measure approaches $M$ but does not reach it
(of course: $M$ just contains smooth functions, and $\mu_{N, t}$ is not a  function),
but it will stay in a $N^{-1/2}$-neighborhood (measured in the $H_{-1}$ norm). How long will it take to reach such a neighborhood? The approach to $M$
is actually exponential and driven by the spectral gap ($\gl_1$) of the linearized evolution operator (at least close to $M$). Therefore
in order to enter such a $N^{-1/2}$-neighborhood a time proportional to $\log N$ appears to be needed, 
as the quick observation that $\exp( -\gl_1 t)=O(N^{-1/2})$ for $t\ge \log N/(2 \gl _1)$ suggests.
 The proofs on this stage of the evolution are   
in Section~\ref{sec:approachM}: here we just stress that 
\begin{enumerate}
\item controlling the effect of the noise on the  system 
on times
$O(\log N)$ is in any case sensibly easier than controlling it on times of order $N$, which is our final aim;
\item 
on times of order $N$ it is no longer a matter of showing that the empirical measure stays close to the solution
of the PDE: on such a time scale the noise takes over and the finite $N$ system, which has a non-trivial (random) dynamics, substantially deviates 
 from the behavior
of the solution to the PDE, which just converges to one of the stationary profiles. 
\end{enumerate}

\medskip

Let us therefore assume that the empirical measure is in a $N^{-1/2}$-neighborhood of a given $q=q_\psi$. 
It is reasonable to assume that the  dominating part of the dynamics close to $q$ is captured by 
the operator $L_q$ and we want to understand the action of the semigroup generated by $L_q$ 
on the noise that stirs the system, on long times. Note that we cannot choose arbitrarily long times, in particular
not times proportional to $N$ right away, because in view of the result we are after 
the stationary profile $q$ around which we linearize changes of an order one amount. We will actually choose 
some intermediate time scale $N^{1/10}$  as we will see in \S~\ref{sec:sheme} and 
Remark~\ref{rem:choice T zeta}{,  that guarantees  that
working with $L_q$ makes sense, i.e. that the projection of the empirical measure on $M$ is still sufficiently
close to $q$. The point is that the effect of the noise on intermediate times is very different in the tangential direction 
and the orthogonal directions to $M$, simply because in the orthogonal direction there is a damping, that is absent 
in the tangential direction. So on intermediate times the the leading term in the evolution of the empirical measure
turns out to be the projection of the evolution on the tangential direction, that is
$( q',\mu_{N,t}-q)_{-1,1/q}/  \Vert q'\Vert_{-1,1/q}$. One can now use Remark~\ref{rem:compute} to obtain
\begin{equation}
\label{eq:forI1}
\left( q',\mu_{N,t}-q\right)_{-1,1/q}
\, =\,  - \int_\bbS \cK  (\theta) \left( \mu_{N,t} (\dd \theta) - q(\theta) \dd \theta\right)\, ,
\end{equation} 
with $\cK$ a primitive of $1-c/q$ ($c$ given in   Remark~\ref{rem:compute}).
By applying It\^o's formula 
we see that the term in \eqref{eq:forI1} can be written as the sum of a drift term and of a martingale term.
It is not difficult to see that to leading order the drift term is zero (a more attentive analysis shows that one has to show that
the next order correction does not give a contribution, but we come back to this below).
The quadratic variation of the martingale term instead turns out to be equal to $t/N$ times
\begin{equation}
\label{eq:heur2}
 \int_{\bbS}(\cK' (\theta))^2 q(\theta) \dd \theta \, =\, 1- \frac{(2\pi)^2}{\int_\bbS 1/q}\, =\,  \Vert q'\Vert_{-1,1/q}^2\, .
 \end{equation}
 Since $q_{\psi+\gep}= q_\psi -  \gep q'_\psi + \cdots$ (note that $q'_\psi$ is not normalized),  
\eqref{eq:heur2} suggests that the diffusion coefficient $D_K$ in Therem~\ref{th:main} is
$\Vert q'\Vert^{-1}_{-1,1/q}$, 
which coincides  with \eqref{eq:DK}. 

To make this procedure work one has to carefully put together the analysis on the intermediate time scale, by setting up an 
adequate iterative scheme. Several delicate issues arise and one of the challenging points
is precisely to control  that the drift can be neglected. In fact the first order expansion
of the projection that we have used
\begin{equation}
 \proj\left(q_\psi+h\right)\, =\, \psi -\frac{( h, q')_{-1, 1/q}}{( q', q')_{-1,1/q}}+ O ( \Vert h \Vert_{-1}^2)\, ,
 \end{equation}
is not accurate enough and one has to go to the next order, see Lemma~\ref{lem:second order projection}. 
This is due to the fact that the random contribution, which in principle appears as first order,  fluctuates and generates 
a cancellation, so in the end the term is of second order. 
\medskip

\begin{rem}
\label{rem:fromU}
It is natural to expect that  Theorem~\ref{th:main} holds true also when $p_0\in U$ and this is just because 
the evolution is attracted to $\frac 1{2\pi}$ and then the noise will cause an escape from this unstable profile
after a time $\propto \log N$, since the exponential instability will make the fluctuations grow exponentially 
with a rate which is just given by the linearized dynamics (linearized around $\frac 1{2\pi}$ of course). 
Arguments in this spirit can be found for example in \cite[Ch.~5]{cf:Errico}, see \cite{cf:Bak} and references therein  for the finite dimensional counterpart. However 
\begin{enumerate}
\item this is not so straightforward because it requires a good control on
the dynamics on and around the heteroclinic orbits linking $\frac 1{2\pi}$ to $M_0$
\cite[Section~5]{cf:GPP};
\item the statement would require more details about the initial condition: the simple convergence 
to a point on $U$ is largely non sufficient (the fluctuations of the initial conditions now matter!); 
\item in general the initial phase $\psi_0$ on $M$ is certainly going to be random: if the initial condition
is rotation invariant (at least in law), like if $\{\gp_0^{j,N}\}_{j=1, \ldots, N}$ are IID variables uniformly 
distributed on $\bbS$ or if $\gp_0^{j,N}= 2\pi j/N$, one  expects $\psi_0$ to be uniformly distributed
on $\bbS$. Note however  that uniform distribution of $\psi_0$ is definitely  not expected in the general case
and asymmetries in the initial condition should affect the distribution of $\psi_0$.
\end{enumerate}
\end{rem}

\subsection{The iterative scheme}
\label{sec:sheme}
As we have explained in \S~\ref{sec:heur}, the analysis close to $M$ requires an iterative procedure,
which we introduce here. We assume that at $t=0$ the system is already close to $M$, while in practice this will happen 
after some time: in Section~\ref{sec:proofmain} we explain how to put together the results on the early stage 
of the evolution and the analysis close to $M$, that we start here. So, for $\mu_0=\mu_{N,0}=\frac{1}{N}\sum_{j=1}^N \gd _{\varphi^{j,N}_0}$ 
such that $\text{dist}(\mu_0,M)\le \gs$ (here and below 
$\text{dist}(\cdot, \cdot)$ is the distance built with the norm of $H_{-1}$), by Lemma~\ref{lem:def proj on M} we can define $\psi_0\, =\, \proj (\mu_0)$. Applying the It\^o formula to $\nu_t=\mu_t-q_{\psi_0}$, we see that
\begin{equation}
\label{eq:nu1}
\nu_{t}\, =\, e^{-t L_{\psi_0}} \nu_0-\int_0^{t} e^{-(t -s)L_{\psi_0}}\partial_\theta[\nu_s J*\nu_s]\dd s + Z_{t}\, ,
\end{equation}
where
\begin{equation}
\label{eq:Z1}
 Z_t\, =\, \frac1N \sum_{j=1}^N \int_0^t \partial_{\theta'}\cG^{\psi_0}_{t-s}\left(\theta,\varphi^{j,N}_s\right)\dd W^j_s\, ,
\end{equation}
and $\cG^{\psi_0}_s(\theta,\theta')$ is the kernel of $e^{-sL_{\psi_0}}$ in $\bbL^2$.
The evolution equation \eqref{eq:nu1} and the noise term \eqref{eq:Z1}
have a meaning in $H_{-1}$, as well as the recentered empirical measures $\nu_t$, and it is in this sense that we will use them: we detail this in
Appendix~\ref{sec:A_H}, where one finds also an explicit expression and some basic facts
about the kernel $\cG^{\psi_0}_s(\theta,\theta')$. We have started here an abuse of notation
that will be persistent through the text: $\partial_{\theta'}\cG^{\psi_0}_{t-s}\left(\theta,\varphi^{j,N}_s\right)$
stands for $\partial_{\theta'}\cG^{\psi_0}_{t-s}\left(\theta,\theta'\right)\vert _{\theta'=\varphi^{j,N}_s}$.

Equations \eqref{eq:nu1}--\eqref{eq:Z1}
are useful tools as long as we can properly define the phase 
associated to the empirical measure of the system and that this phase is close to $\psi_0$: 
in view of the result we want to prove, this is expected to be true for a long time,
but it is certainly expected to fail for times of the order of $N$, since on this timescale the phase 
does change of an amount that does not vanish as $N$ becomes large.

The idea is therefore  
 to divide the evolution of the particle system up to 
 a final time proportional to $N$ 
 into $n=n_N\stackrel{N \to \infty}\longrightarrow \infty$ time intervals
$[T_i, T_{i+1}]$, where  $T_i=iT$ and $T=T(N)$
is chosen close to a fractional power of $N$ (see Remark \ref{rem:choice T zeta}).
Moreover $i$ runs from $1$ up to $n=n_N$ so that $n_N T(N)= T_{n_N}$ and $\lim_N T_{n_N}/N$ is equal to a positive
constant
(the $\tau_f$ of Theorem~\ref{th:main}).
 If the empirical measure $\mu_t$ stays close to the manifold $M$, we can define the
 projections of $\mu_{T_k}$ and successively update the 
 reentering phase at all times $T_k$. The point then will be essentially to show that the process
 given by these phases, on the time scale $\propto N$,
   converges to a Brownian motion. 
   
   More formally, we construct the following iterative scheme: we choose
 \begin{equation}
 \label{eq:sigma}
 \gs=\gs_N := \lceil N^{2 \zeta} \sqrt{T/N}\rceil\stackrel{N\to \infty} \longrightarrow 0\, ,
 \end{equation}
 $\zeta>0$ (see Remark~\ref{rem:choice T zeta}), we set
$\tau^0_{\gs_N}=0$
and for $k=1,2, \ldots$ we define
\begin{equation}
 \psi_{k-1}\, :=\, \proj(\mu_{T_{k-1}})\, ,  
\end{equation}
if $\text{dist}(\mu_{T_{k-1}},M)\le \gs_N$ 
and 
\begin{equation}
 \tau^k_{\gs_N}=\tau_{\gs_N}^{k-1}\ind_{\{\tau_{\gs_N}^{k-1}<T_{k-1}\}}+\inf\{s\in[T_{k-1}, T_{k}],\, \Vert \mu_s-q_{\psi_{k-1}}\Vert_{-1}>{\gs_N}\}\ind_{\{\tau_{\gs_N}^{k-1}\geq T_{k-1}\}}\, .
\end{equation}
Then we set 
\begin{equation}
\nu_t^{k}:= \mu_t - q_{\psi_{k-1}}\, ,
\end{equation}
 for 
 $t\in [T_{k-1},T_{k}]$ and $t\le   \tau_{\gs_N}^{k}$, and otherwise
 $\nu_t^{k}:=\nu_{ \tau_{\gs_N}^{k}}^k$ for every $t\ge \tau_{\gs_N}^{k}$
 (of course $\tau_{\gs_N}^{k}$ can be smaller than $T_{k-1}$ and, in this case,
 the definition becomes redundant). Therefore the $\nu$ process we have just defined
 solves for $t\in [T_{k-1},T_{k}]$
\begin{multline}
\label{eq:ito nu^k}
 \nu^{k}_{t}\, =\, 
 \ind_{\{\tau^k_{\gs_N}<T_{k-1}\}}\nu^k_{T_{k-1}} +  \ind_{\{\tau^k_{\gs_N}\geq T_{k-1}\}} \times \\
\left(e^{-(t\wedge \tau^k_{\gs_N}-T_{k-1}) L_{\psi_{k-1}}} \nu^{k}_{T_{k-1}}-\int_{T_{k-1}}^{t\wedge \tau^k_{\gs_N}} e^{-(t\wedge \tau^k_{\gs_N} -s)L_{\psi_{k-1}}}\partial_\theta[\nu^{k}_s J*\nu^{k}_s]\dd s +Z^{k}_{t\wedge \tau^k_{\gs_N}}\right)\, ,
\end{multline}
where
\begin{equation}
\label{eq:Zkt0}
 Z^k_{t}\, =\, \frac1N \sum_{j=1}^N \int_{T_{k-1}}^t \partial_{\theta'}\cG^{\psi_{k-1}}_{t-s}\left(\theta,\varphi^{j,N}_s\right)\dd W^j_s\, .
\end{equation}
Once again, we refer to Appendix~\ref{sec:A_H} for the precise meaning of 
\eqref{eq:ito nu^k} and \eqref{eq:Zkt0}.

\begin{rem}
\label{rem:choice T zeta}
For the remainder of the paper we choose $T(N)\sim N^{1/10}$ and 
  $\zeta\le 1/100$. The two exponents do not have any particular meaning: a look at the
  argument shows that the exponent for $T(N)$ has in any case to be chosen smaller than $1/2$, but then
  a number of technical estimates enter the game and we have settled for a value $1/10$ without
  trying to get the optimal value that comes out of the method we use.
\end{rem}

\section{A priori estimates: persistence of proximity to $M$}
\label{sec:apriori}
The aim of this section is to prove that, if we are (say, at time zero) sufficiently close to $M$, 
we stay close to $M$  for times $O(N)$. The arguments in this section justify the
choice of the proximity parameter
$\gs_N$ that we have made in the iterative scheme.
We first prove some estimates on the size of the noise term and then we will
give the estimates on the empirical measure.

 \subsection{Noise estimates}
We define the event
\begin{multline}
\label{eq:BupN}
 B^N\, =\, \left\{ \sup_{1\leq k\leq n-1}\sup_{ t \in [T_k,T_{k+1}]} \left\Vert Z^k_t \right\Vert_{-1}\leq \sqrt{\frac{T}{N}}N^\gz\right\}\\
\bigcap
\left\{ \sup_{1\leq k\leq n-1}\sup_{ t \in [T_k,T_{k+1}]} \left\Vert Z^{k,\perp}_t \right\Vert_{-1}\leq \frac{1}{\sqrt{N}}N^\gz\right\}\, ,
\end{multline}
where $Z^{k,\perp}_t $ is defined precisely like $Z^{k}_t $, see  \eqref{eq:Zkt0}, 
except for the replacement of  $\cG_{t-s}^{\psi}(\cdot, \cdot)$ with $\cG_{t-s}^{\psi}
(\theta, \theta')-e_{\psi_{k-1},0}(\theta)
f_{\psi_{k-1},0}(\theta' )$.
\medskip

\begin{lemma}
\label{lem:bound Zk} 
$\lim_{N \to \infty} \bbP \left(B^N\right) =1$.
\end{lemma}
\medskip

\begin{proof}
In order to perform the estimates we introduce and work with approximated versions of
$Z^{k}_t $ and $Z^{k,\perp}_t $ (see Lemma~\ref{th:add2}).
Define for $T_{k-1}<t'<t$
\begin{equation}
\label{eq:M.1}
 Z^k_{t,t'}=\frac1N \sum_ {j=1}^N\int_{T_{k-1}}^{t'}\partial_{\theta'}\cG^{\psi_{k-1}}_{t-s}(\theta,\varphi^{j,N}_s)\dd W^j_s\, .
\end{equation}
The kernel $\cG^{\psi_{k-1}}_\cdot$ in this case is (cf. Appendix~\ref{sec:A_H})
\begin{equation}
 \cG^{\psi_{k-1}}_s(\theta,\theta')\, =\, \sum_{l=0}^\infty e^{-s\gl_l}e_{\psi_{k-1},l}(\theta)f_{\psi_{k-1},l}(\theta')\, ,
\end{equation}
where $\gl_l$ are the ordered eigenvalues of $-L_{\psi_{k-1}}$, $e_{\psi_{k-1},l}$ are the associated eigenfunctions of unit norm in $H_{-1,1/{q_{\psi_{k-1}}}}$,
cf. Remark~\ref{rem:Lpsi},  and $f_{\psi_{k-1},l}$ are the eigenfunctions
of  $L_{\psi_{k-1}}^*$, the adjoint in $\bbL ^2$ (see   Appendix~\ref{sec:A_H}). 

Very much in  in the same way we define
\begin{equation}
\label{eq:M-ort}
  Z_{t,t'}^{k,\perp}=\frac1N \sum_ {j=1}^N\int_{T_{k-1}}^{t'}\partial_{\theta'} \cG^{\psi_{k-1},\perp}_{t-s}(\theta,\varphi^{j,N}_s)\dd W^j_s\, ,
\end{equation}
with
\begin{equation}
   \cG^{\psi_{k-1},\perp}_s(\theta,\theta')\, =\, \sum_{l=1}^\infty e^{-s\gl_l}e_{\psi_{k-1},l}(\theta)f_{\psi_{k-1},l}(\theta')\, .
\end{equation}
We decompose for $T_{k-1}<s'<s<t$ and $s'<t'<t$
\begin{multline}
\label{eq:decompose}
 Z^k_{t,t'}-Z^k_{s,s'}\, =\, \frac1N \sum_ {j=1}^N\int_{T_{k-1}}^{s'}\left(\partial_{\theta'}\cG^{\psi_{k-1}}_{t-u}(\theta,\varphi^{j,N}_u)-\partial_{\theta'}\cG^{\psi_{k-1}}_{s-u}(\theta,\varphi^{j,N}_u)\right)\dd W^j_u \\
+\frac1N \sum_ {j=1}^N\int_{s'}^{t'}\partial_{\theta'}\cG^{\psi_{k-1}}_{t-u}(\theta,\varphi^{j,N}_u)\dd W^j_u\, ,
\end{multline}
and an absolutely analogous formula holds for  $Z^{k,\perp}$: in fact the
bounds for $Z^{k}$ and $Z^{k,\perp}$  are obtained with the same technique even if the results
are slightly different due to the presence of the zero eigenvalue in $Z^{k}$.
Moreover we apply
$
 \Vert a+b\Vert^2\,  \leq \, 2(\Vert a\Vert^2 + \Vert b\Vert^2)$, so that we can estimate the two terms in the right-hand side of \eqref{eq:decompose} separately.
 
And we start with
the second term of the right-hand side in \eqref{eq:decompose}: by the orthogonality properties
of the eigenvectors we obtain 
\begin{multline}
\label{eq:norm 2 second term}
 \left\Vert \frac1N \sum_ {j=1}^N\int_{s'}^{t'}\partial_{\theta'}\cG^{\psi_{k-1}}_{t-u}(\cdot,\varphi^{j,N}_u)\dd W^j_u\right\Vert_{-1,1/q}^2\\
 =\, \frac{1}{N^2} \sum_{l=0}^\infty \sum_{j,j'=1}^N \int_{s'}^{t'}\int_{s'}^{t'} e^{-(2t-u-u')\gl_l}f'_{\psi_{k-1},l}(\varphi^{j,N}_u)f'_{\psi_{k-1},l}(\varphi^{j',N}_{u'})\dd W^j_u\dd W^{j'}_{u'}\, ,
\end{multline}
and by taking the expectation
\begin{multline}
\label{eq:int-step4.1-0}
 \bbE\left[ \left\Vert \frac1N \sum_ {j=1}^N\int_{s'}^{t'}\partial_{\theta'}\cG^{\psi_{k-1}}_{t-u}(\cdot,\varphi^{j,N}_u)\dd W^j_u\right\Vert_{-1,1/q}^2\right]\, =\\
  \frac{1}{N^2}\sum_{l=0}^\infty \sum_{j=1}^N \int_{s'}^{t'}e^{-2(t-u)\gl_l}\bbE\left[(f_{\psi_{k-1},l}'(\varphi^{j,N}_u))^2\right]\dd u\, .
\end{multline}
By Corollary \ref{cor:fj}  
there exists a constant $C_1$ such that
\begin{equation}
\label{eq:int-step4.1}
 \bbE\left[ \left\Vert \frac1N \sum_ {j=1}^N\int_{s'}^{t'}\partial_{\theta'}\cG^{\psi_{k-1}}_{t-u}(\cdot,\varphi^{j,N}_u)\dd W^j_u\right\Vert_{-1,1/q}^2\right]\,
\leq \, \frac{C_1}{N}\sum_{l=0}^\infty \int_{s'}^{t'}e^{-2(t-u)\gl_l}\dd u\, .
\end{equation}
Proposition \ref{prop:eigenvalues and eigenfunction expansion}, Remark~\ref{rem:eigenvalues and eigenfunction expansion}, leads us to
\begin{equation}
\label{eq:peig2}
 \sum_{l=0}^\infty \int_{s'}^{t'}e^{-2(t-u)\gl_l}\dd u\, \leq\, \sum_{l=0}^\infty \int_{s'}^{t'}e^{-(t-u)\frac{l^2}{C}}\dd u
\, \leq\, C\sum_{l=0}^\infty \frac{1}{l^2}\left(1-e^{-(t'-s')\frac{l^2}{C}}\right)\, ,
\end{equation}
where the addend  with $l=0$ (times $C$) has to be read as $t'-s'$. 
The right-most term in \eqref{eq:peig2} for $t'-s'\ge 1$
can be bounded by $C(t'-s')+ C\sum_{l=1}^\infty 1/l^2 \le 3C(t'-s')$. Instead for $t'-s'<1$
 we decompose the same term and then estimate as follows:
\begin{multline}
C\sum_{l=0}^{\left\llcorner(t'-s')^{-1/2}\right\lrcorner} \frac{1}{l^2}\left(1-e^{-(t'-s')\frac{l^2}{C}}\right)+C\sum_{l=\left\llcorner(t'-s')^{-1/2}\right\lrcorner+1}^\infty \frac{1}{l^2}\left(1-e^{-(t'-s')\frac{l^2}{C}}\right)
\\
\leq\, \sum_{l=0}^{\left\llcorner(t'-s')^{-1/2}\right\lrcorner} (t'-s')+\sum_{l=\left\llcorner(t'-s')^{-1/2}\right\lrcorner+1}^\infty \frac{C}{l^2}\, \le \, (3+2C) \sqrt{t'-s'}\, ,
\end{multline}
where for the first term we have used  $(1-\exp(-a)) \le a$, for $a\ge 0$.
Therefore we have proven that there exists $C$ such that for every $k$
and every $s$, $s'$, $t$, $t'$ such that $T_{k-1} <s' <s<t$ and $s' <t'<t$ we have
\begin{equation}
\label{eq:res1}
 \bbE\left[ \left\Vert \frac1N \sum_ {j=1}^N\int_{s'}^{t'}\partial_{\theta'}\cG^{\psi_{k-1}}_{t-u}(\cdot,\varphi^{j,N}_u)\dd W^j_u\right\Vert_{-1,1/q}^2\right]\, \leq\,
\frac{C h_1(t'-s')}{N}\, .
\end{equation}
with 
\begin{equation}
h_1(u)\, :=\, u^{1/2}\ind_{[0,1)}(u) +u\ind_{[1,\infty)}\, .
\end{equation}
\medskip

We can do better in the case of $\cG^{\psi_{k-1},\perp}$, for which a direct inspection of the argument 
we have just presented shows that the linearly growing term in the estimate can be 
 avoided (since the term $l=0$ is no longer there) and the net result is
\begin{equation}
\label{eq:res2}
 \bbE\left[ \left\Vert \frac1N \sum_ {j=1}^N\int_{s'}^{t'}\partial_{\theta'}\cG^{\psi_{k-1},\perp}_{t-u}(\cdot,\varphi^{j,N}_u)\dd W^j_u\right\Vert_{-1,1/q}^2\right]\, \leq\,
\frac{C h_2(t'-s')}{N}\, .
\end{equation}
with $h_2$ defined as
\begin{equation}
h_2(u)\, =\,  u^{1/2}\ind_{[0,1)}(u) +\ind_{[1,\infty)}\, .
\end{equation}

\medskip

For what concerns the first term in the right-hand side of \eqref{eq:decompose}, we have
\begin{multline}
\label{eq:expr5}
 \bbE\left[ \left\Vert \frac1N \sum_ {j=1}^N\int_{T_{k-1}}^{s'}\left(\partial_{\theta'}\cG^{\psi_{k-1}}_{t-u}(\cdot,\varphi^{j,N}_u)-\partial_{\theta'}\cG^{\psi_{k-1}}_{s-u}(\cdot,\varphi^{j,N}_u)\right)\dd W^j_u\right\Vert_{-1,1/q}^2\right]\\
 =\, \frac{1}{N^2} \sum_{l=1}^\infty\sum_{j=1}^N \int_{T_{k-1}}^{s'}\left(e^{-(t-u)\gl_l}-e^{-(s-u)\gl_l}\right)^2\bbE \left[(f'_{\psi_{k-1},l}(\gp^{j,N}_u))^2\right] \dd u\, ,
\end{multline}
and, by proceeding like for \eqref{eq:int-step4.1}, we see that the expression 
in \eqref{eq:expr5} is bounded by
\begin{equation}
\label{ineq:diff moment order 2}
\frac{C_1}{N}\sum_{l=1}^\infty \frac{\left(1- e^{-\gl_l(t-s')}\right)^2}{\gl_l}
\, \le \, \frac{C_1 C}{N}\sum_{l=1}^\infty \frac{\left( e^{-(t-s')\frac{l^2}{C}}-1\right)^2}{l^2}\, .
\end{equation}
This last term is estimated once again by separating the two cases of 
$t-s'$ small and large. The net result is that there exists $C>0$ such that for every $k$,
every $s$, $s'$ and $t$ such that $T_k<s'< s<t$ we have
\begin{equation}
\label{eq:res3}
 \bbE\left[ \left\Vert \frac1N \sum_ {j=1}^N\int_{T_{k-1}}^{s'}\left(\partial_{\theta'}\cG^{\psi_{k-1}}_{t-u}(\cdot,\varphi^{j,N}_u)-\partial_{\theta'}\cG^{\psi_{k-1}}_{s-u}(\cdot,\varphi^{j,N}_u)\right)\dd W^j_u\right\Vert_{-1,1/q}^2\right] \,
\leq \, C h_2(t-s')\, .
\end{equation}

\medskip

In order to complete the proof of
Lemma~\ref{lem:bound Zk}
quadratic estimates do not suffice: we need to generalize \eqref{eq:res1}, \eqref{eq:res2} and\eqref{eq:res3} to larger exponents. We actually need 
estimates on moments of order $2m$, with $m$ finite, but sufficiently large, so to apply the standard Kolmogorov Lemma type estimates and get uniform bounds.   
We  are going to use
\begin{equation}
 \Vert a + b \Vert^m \, \leq \, m (\Vert a \Vert^m + \Vert b \Vert^m)\, ,
\end{equation}
but actually we will not track the $m$ dependence of the constants.
We aim at showing that the expectation of the moments of order $2m$ of the quantities we are interested in 
are bounded by the $m^{\mathrm{th}}$ power of the estimate we found in the quadratic case, times 
an $m$-dependent constant.

\medskip 

For $m=1,2, \ldots $, the $m^\mathrm{th}$--power of the expression in  \eqref{eq:norm 2 second term} gives
\begin{multline}
\label{eq:devel1}
 \left\Vert \frac1N \sum_ {j=1}^N\int_{s'}^{t'}\partial_{\theta'}\cG^{\psi_{k-1}}_{t-u}(\theta,\varphi^{j,N}_u)\dd W^j_u\right\Vert_{-1,1/q}^{2m}\\
 =\, \frac{1}{N^{2m}} \sum_{l_1,\dots ,l_m=0}^\infty \sum_{j_1,j_1',\dots,j_m,j'_m=1}^N F^{j_1}_{l_1}(t,s',t')F^{j'_1}_{l_1}(t,s',t')\cdots F^{j_m}_{l_m}(t,s',t')F^{j'_m}_{l_m}(t,s',t')\, ,
\end{multline}
in which we have introduced 
  the random variables
\begin{equation}
 F^j_l(t,s',t')\, =\, \int_{s'}^{t'} e^{-\gl_l(t-u)}f'_{\psi_{k-1},l}(\gp^{j,N}_u)\dd W^j_u\, .
\end{equation}
We now take the expectation of both terms in 
\eqref{eq:devel1} and  all the terms in the sum that do not include an even number of each Brownian motion vanish. The number of non-zero terms in the expectation can thus be bounded by $(2m)!N^m$. Applying the It\^o formula to each of these non-zero terms, we get at most $(2m)!/(2^mm!)$ terms (the number of possibilities classifying $2m$ elements in couples) of the type $I_1\cdots I_m$, where 
\begin{equation}
 I_k\, =I_k(l_1,l_2)\, =\, \int_{s'}^{t'} e^{-(\gl_{l_1}+\gl_{l_2})(t-u)}\bbE \left[f'_{\psi_{k-1},l_1}(\gp^{j,N}_u)f'_{\psi_{k-1},l_2}(\gp^{j,N}_u) \right]\dd u\, .
\end{equation}
We now observe that 
\begin{equation}
\vert I_k(l_1,l_2)\vert \, \le \, \sqrt{I_k(l_1,l_1)I_k(l_2,l_2)}\, , 
\end{equation}
and for each index index $l_i$ in the first sum in the right-hand side of \eqref{eq:devel1} gives rise either
directly to a term $I_k (l_i)$ (for this it is needed that the two terms share the Brownian motion),
or in the arising products the terms  $I_k(l_i,l_{i'})$ are associated with a term of the type $I_k(l_i,l_{i''})$.
Therefore the expression obtained after applying It\^o formula can be bounded by a sum of terms 
of the type $\hat I_1\cdots \hat I_m$, with
\begin{equation}
 \hat I_k\, =\, I_k(\,l,l)\, =\,  \int_{s'}^{t'} e^{-2\gl_l(t-u)}\bbE \left[(f'_{\psi_{k-1},l}(\gp^{j,N}_u))^2\right] \dd u\, .
\end{equation}
Therefore we are facing the same estimates that we have encountered in the
quadratic case, see \eqref{eq:int-step4.1-0} and \eqref{eq:int-step4.1}, except of course for 
combinatorial contribution.
 In the end we obtain that there exists $C=C_m$ such that
\begin{equation}
\label{bound:first term noise non perp}
\bbE \left[
 \left\Vert \frac1N \sum_ {j=1}^N\int_{s'}^{t'}\partial_{\theta'}\cG^{\psi_{k-1}}_{t-u}(\cdot,\varphi^{j,N}_u)\dd W^j_u\right\Vert_{-1,1/q}^{2m}\right] \, \leq\,  C\frac{h^m_1(t'-s')}{N^m}\, ,
\end{equation}
\begin{equation}
\label{bound:second term noise perp}
\bbE\left[
 \left\Vert \frac1N \sum_ {j=1}^N\int_{s'}^{t'}\partial_{\theta'}\cG^{\psi_{k-1},\perp}_{t-u}(\cdot,\varphi^{j,N}_u)\dd W^j_u\right\Vert_{-1,1/q}^{2m}\right] \, \leq\,  C\frac{h_2^m(t'-s')}{N^m}\, .
\end{equation}
In a similar way
\begin{multline}
\left\Vert \frac1N \sum_ {j=1}^N\int_{T_{k-1}}^{s'}\left(\partial_{\theta'}\cG^{\psi_{k-1}}_{t-u}(\cdot,\varphi^{j,N}_u)-\partial_{\theta'}\cG^{\psi_{k-1}}_{s-u}(\cdot,\varphi^{j,N}_u)\right)\dd W^j_u\right\Vert_{-1,1/q}^{2m}\\
=\, \frac{1}{N^{2m}} \sum_{l_1,\dots ,l_m=0}^\infty \sum_{j_1,j_1',\dots,j_m,j'_m=1}^N G^{j_1}_{l_1}(s,t,s')G^{j'_1}_{l_1}(s,t,s')\cdots G^{j_m}_{l_m}(s,t,s')G^{j'_m}_{l_m}(s,t,s')\,
\end{multline}
with
\begin{equation}
 G^{j}_{l}(s,t,s')\, =\, \int_{T_{k-1}}^{s'} \left(e^{-\gl_l(t-u)}-e^{-\gl_l(s-u)}\right)f'_{\psi_{k-1},l}(\gp^{j,N}_u)\dd W^j_u\, .
\end{equation}
We reduce the problem as above to the study of products of integral terms $J_1\cdots J_k$ with
\begin{equation}
 J_k\, =\, \int_{T_{k-1}}^{s'} \left(e^{-\gl_{l_1}(t-u)}-e^{-\gl_{l_1}(s-u)}\right)\left(e^{-\gl_{l_2}(t-u)}-e^{-\gl_{l_2}(s-u)}\right)\bbE \left[
 f'_{\psi_{k-1},l_1}(\gp^{j,N}_u)f'_{\psi_{k-1},l_2}(\gp^{j,N}_u)\right] \dd u\, ,
\end{equation}
and then, like before, in terms of products of {\sl diagonal} terms of the type
\begin{equation}
 \hat J_k\, =\, \int_{T_{k-1}}^{s'} \left(e^{-\gl_l(t-u)}-e^{-\gl_l(s-u)}\right)^2\bbE\left[(f'_{\psi_{k-1},l}(\gp^{j,N}_u))^2\right]
  \dd u \, .
\end{equation}
Again we are reduced to the estimating terms that have already appeared in the quadratic case, see \eqref{eq:expr5}, so we obtain that there exists $C=C_m$ such that
\begin{multline}
\label{bound:second term noise}
 \bbE\left[ \left\Vert \frac1N \sum_ {j=1}^N\int_{T_{k-1}}^{s'}\left(\partial_{\theta'}\cG^{\psi_{k-1}}_{t-u}(\cdot,\varphi^{j,N}_u)-\partial_{\theta'}\cG^{\psi_{k-1}}_{s-u}(\cdot ,\varphi^{j,N}_u)\right)\dd W^j_u\right\Vert_{-1,1/q}^{2m}\right]\\
\leq \, C\frac{h_2^m(t-s)}{N^m}\, .
\end{multline}

\medskip

We now let $t' \nearrow t$ and $s'\nearrow s$ and by  applying Fatou's Lemma and Lemma~\ref{th:add2}, from\eqref{eq:decompose}, \eqref{bound:first term noise non perp}, \eqref{bound:second term noise perp} and \eqref{bound:second term noise} we get
\begin{equation}
\label{eq:estcp1}
 \bbE\left[ \left\Vert Z^k_t-Z^k_s \right\Vert_{-1}^{2m} \right]\, \leq \, C\frac{h_1^m(t-s)}{N^m}\, ,
\end{equation}
and 
\begin{equation}
\label{eq:estcp2}
 \bbE\left[ \left\Vert Z_t^{k,\perp}-Z_s^{k,\perp} \right\Vert_{-1}^{2m} \right]\, \leq \, C\frac{h_2^m(t-s)}{N^m}\, .
\end{equation}
The fact that  we are allowed to drop the weight in the $H_{-1}$ norm is of course due to the norm equivalence.
\medskip

We are now in good condition to apply the Garsia-Rodemich-Rumsey Lemma \cite{cf:SV}:
\medskip

\begin{lemma}
\label{th:GRR}
Let $p$ and $\Psi$ be continuous, strictly increasing functions on $(0,\infty)$ such that $p(0)=\Psi(0)=0$ and $\lim_{t\nearrow\infty} \Psi(t)=\infty$.
Given $T>0$ and $\phi$ continuous on $(0,T)$ and taking its values in a Banach space $(E,\Vert.\Vert)$, if
\begin{equation}
 \int_0^T\int_0^T \Psi\left(\frac{\Vert \phi(t)-\phi(s)\Vert}{p(|t-s|)}\right)\dd s\dd t \, \leq \, B\, < \, \infty\, ,
\end{equation}
then for $0\leq s\leq t\leq T$:
\begin{equation}
 \Vert\phi(t)-\phi(s)\Vert\, \leq \, 8\int_0^{t-s} \Psi^{-1}\left(\frac{4B}{u^2}\right)p(\dd u)\, .
\end{equation}
\end{lemma}

\medskip 

We apply Lemma~\ref{th:GRR} with 
\begin{equation}
\phi(t)\, =\, Z^k_{t-T_{k-1}}\, , \ \ \ p(u)\, =\, u^{\frac{2+\gz}{2m}}\ \text{  and } \ \Psi(u)\, =\, u^{2m}\, ,
\end{equation}
and $\zeta=1/100$ (Remark~\ref{rem:choice T zeta}).
With these choices we can find an explicit constant $C=C(m, \zeta)$ such that
\begin{equation}
\Vert Z^k_t-Z^k_s \Vert_{-1}^{2m}\, \le \, C ( t-s ) ^\zeta B \, ,
\end{equation}
for every $s$ and $t$ such that $T_{k-1}\leq s< t\leq T_k$ and $B$ is a positive random variable
such that
\begin{equation}
\bbE [B]\, \le \, \frac C{N^m} \int_0^T \int_0^T \frac{h_1^m(\vert t-s \vert)}{\vert t-s\vert^{2+\zeta}} \dd s \dd t\, ,
\end{equation}
where $C$ is the constant in \eqref{eq:estcp1}.
For $m>4$ the function $t \mapsto h_1^m(t) / t^{2+\zeta}$, defined for $t>0$,  is increasing (and  it tends to zero for $t \searrow 0$). So $\bbE [B]$ is bounded by $CN^{-m} h_1^m(T)/T^\zeta$ and therefore
\begin{equation}
 \bbE\left[\sup_{T_{k-1}\leq s< t\leq T_k}\frac{\Vert Z^k_t-Z^k_s \Vert_{-1}^{2m}}{|t-s|^\gz}\right]\, \leq \, C \frac{T^{m-\gz}}{N^m}\, ,
\end{equation}
which leads to
\begin{equation}
 \bbP\left[ \sup_{T_{k-1}\leq t\leq T_k} \Vert Z^k_t \Vert_{-1} \geq \sqrt{\frac{T}{N}}N^\gz\right]\, \leq \, C\frac{1}{N^{m\gz}}\, .
\end{equation}
Then, (recall $n=n_N=\frac{N}{T}$) we deduce
\begin{equation}
  \bbP\left[ \sup_{1\leq k\leq n} \sup_{T_{k-1}\leq t\leq T_k} \Vert Z^k_t \Vert_{-1} \geq \sqrt{\frac{T}{N}}N^\gz\right]\, \leq \, C\frac{1}{TN^{m\gz-1}}\, ,
\end{equation}
where the right hand side tends to $0$ when $m$ is chosen sufficiently large.
A similar argument gives for $Z_t^{k,\perp}$
\begin{equation}
  \bbP\left[\sup_{1\leq k\leq n} \sup_{T_{k-1}\leq t\leq T_k} \Vert Z_t^{k,\perp} \Vert_{-1} \geq \frac{1}{\sqrt{N}}N^\gz\right]\, \leq \, C\frac{T^{\gz-1}}{N^{m\gz-1}}\, .
\end{equation}

\end{proof}

\medskip

We now give the main result of the section:
\medskip

\begin{proposition}
\label{prop:closeness traj manifold}
If $\Vert \nu^1_0 \Vert_{-1} \leq \frac{N^{2\gz}}{\sqrt{N}}$ and if the event $B^N$ defined in \eqref{eq:BupN} is realized (then, with probability approaching $1$
as $N \to \infty$) we have
 \begin{equation}
\label{ineq:bound sup nuk}
   \sup_{1\leq k\leq n}\sup_{ t \in [T_{k-1},T_{k}]} \left\Vert \nu^k_t \right\Vert_{-1}\leq \sqrt{\frac{T}{N}}N^{2\gz}\, ,
 \end{equation}
 and 
\begin{equation}
\label{ineq:bound sup nuk aux}
  \max_{1\leq k\leq n} \left\Vert \nu^{k}_{T_{k-1}} \right\Vert_{-1}\leq \frac{N^{2\gz}}{\sqrt{N}}\, .
\end{equation}
\end{proposition}

\medskip

\begin{proof}
In view of Lemma~\ref{lem:bound Zk} we can and will assume that $B^N$ is verified.
From \eqref{eq:ito nu^k} and Lemma~\ref{th:add1}, we get,for all $k=1\dots n$ and $t\in [T_{k-1},T_k]$
\begin{equation}
\Vert \nu^k_t \Vert_{-1}\, \leq\, Ce^{-\gl_1 (t-T_{k-1})} \Vert \nu^k_{T_{k-1}}\Vert_{-1} + C \int_{T_{k-1}}^t \left( 1+\frac{1}{\sqrt{t-s}}\right) \Vert \nu^k_s \Vert_{-1}^2 \dd s
+ \Vert Z^k_t \Vert_{-1}\, .
\end{equation}
The constant $C$ in front of the first term of the right hand side above would be equal to $1$ if we were using the $\Vert .\Vert_{-1,1/q_{\psi_{k-1}}}$ norm.
Let us assume that $\Vert \nu^k_{T_{k-1}} \Vert_{-1}\le N^{2\zeta}/\sqrt{N}$, using 
Lemma \ref{lem:bound Zk} we obtain
\begin{equation}
\label{ineq:bound nuk}
\Vert \nu^k_t\Vert_{-1} \, \leq \, Ce^{-\gl_1 (t-T_{k-1})} \frac{N^{2\zeta} }{\sqrt{N}}+ C(T+\sqrt{T})\sup_{T_{k-1}\leq s\leq t}\Vert \nu^k_s \Vert_{-1}^2 + 
\frac{\sqrt{T}}{\sqrt{N}}N^{\zeta}\, .
\end{equation}
Therefore we readily see that if we  define
\begin{equation}
t^*\, =\, \sup\left\{ t\in [T_{k-1},T_k]: \, \Vert \nu^k_t \Vert_{-1} > \frac{\sqrt{T}}{\sqrt{N}}N^{2\zeta}\right\}\, ,
\end{equation}
we have that for $t\le t^*$
\begin{equation}
\label{ineq:bound nuk-2}
\Vert \nu^k_t\Vert_{-1} \, \le \, CN^{2\gz - \frac 12}+2CT^2 N^{4\zeta -1}+ \sqrt{T} N^{\zeta -\frac 12}\, .
\end{equation}
Therefore since $\lim_N T^3 N^{-1+4\zeta}=0$ (see Remark \ref{rem:choice T zeta}),
for $N$ large enough, we have  $t^*=T_k$ and \eqref{ineq:bound sup nuk}
is reduced to proving n $\Vert \nu^k_{T_{k-1}} \Vert_{-1}\le N^{2\zeta}/\sqrt{N}$ for $k=1,2, \ldots, n$.
This holds for $k=1$: we are now going to show 
by induction \eqref{ineq:bound sup nuk aux} and therefore that the assumption propagates from
$k$ to $k+1$.

To prove the bound on $\nu^{k+1}_{T_k}$, assuming 
the bound on $\nu^{k}_{T_{k-1}}$, we use the smoothness of the manifold $M$. Since we are working in $B^N$, $\tau^k_\gs=T_k$ and  we have
\begin{equation}
\label{eq:decomp nu0}
\begin{split}
 \nu^{k+1}_{T_{k}}\, &=\, q_{\psi_{k-1}}+\nu^k_{T_k}-q_{\psi_k} \\
               &=\, P^\perp_{\psi_k}\left[q_{\psi_{k-1}}+\nu^k_{T_k}-q_{\psi_k}\right] \\
               &=\, \left(P^\perp_{\psi_k}-P^\perp_{\psi_{k-1}}\right)\left[ q_{\psi_{k-1}}+\nu^k_{T_k}-q_{\psi_k}\right] +P^\perp_{\psi_{k-1}}\left[q_{\psi_{k-1}}-q_{\psi_k}\right]
               +P^\perp_{\psi_{k-1}}\nu^k_{T_k}\, ,
\end{split}
\end{equation}
Since the mapping $\psi\mapsto P^\perp_\psi$ is smooth on the compact $M$, we have (cf. \S~\ref{sec:Manif})
\begin{equation}
 \left\Vert P^\perp_{\psi_k}-P^\perp_{\psi_{k-1}} \right\Vert_{\cL(H_{-1},H_{-1})}\, \leq \, C\left\vert \psi_k-\psi_{k-1} \right\vert \, ,
\end{equation}
and the identities
\begin{equation}
\label{eq:link psi nu}
 \psi_k-\psi_{k-1}\, =\, \proj (\mu_{T_k})-\proj(\mu_{T_{k-1}})\, ,
\end{equation}
and
\begin{equation}
\label{eq:link mu nu tk}
 \mu_{T_k}-\mu_{T_{k-1}}\, =\, \nu^k_{T_k}-\nu^k_{T_{k-1}}\, ,
\end{equation}
combined with the smoothness of $\proj$,  lead to (using \eqref{ineq:bound sup nuk})
\begin{equation}
  \left\Vert P^\perp_{\psi_k}-P^\perp_{\psi_{k-1}} \right\Vert_{\cL(H_{-1},H_{-1})}\, \leq \, C\frac{\sqrt{T}}{\sqrt{N}}N^{2\zeta}\, .
\end{equation}
On the other hand, the smoothness of $q_\psi$ with respect to $\psi$, \eqref{eq:link psi nu} and \eqref{eq:link mu nu tk} imply
\begin{equation}
 \left\Vert q_{\psi_{k-1}}+\nu^k_{T_k}-q_{\psi_k} \right\Vert_{-1}\, \leq\, C \left(
 \Vert \nu^k_{T_{k-1}}\Vert_{-1} +\Vert \nu^k_{T_k}\Vert_{-1}\right)\, ,
\end{equation}
so the first term in the last line of \eqref{eq:decomp nu0} is of order $\frac{T}{N}N^{4\zeta}$,
which is much smaller than $ \frac{N^{2\gz}}{\sqrt{N}}$ nor $N \to \infty$,
since $\lim_N T N^{2\zeta -\frac 12}=0$ (see Remark \ref{rem:choice T zeta}).
Moreover, Lemma \ref{lem:def proj on M} implies
\begin{equation}
 \left\Vert P^\perp_{\psi_{k-1}}\left[q_{\psi_{k-1}}-q_{\psi_k}\right]\right\Vert_{-1}\, =\,
\left\Vert  P^\perp_{\psi_{k-1}}  \left[ v(\mu_{T_{k-1}})-v(\mu_{T_k}) \right] \right\Vert_{-1}\, \leq\, 
C\left\Vert \mu_{T_{k-1}}-\mu_{T_k}\right\Vert_{-1}\, ,
\end{equation}
so the second term in the last line of \eqref{eq:decomp nu0} is also of order $\frac{T}{N}N^{4\zeta}$. 
Finally, projecting \eqref{eq:ito nu^k} on
$\text{Range}\left(L_{q_{\psi_{k-1}}}\right)$ and by using again Lemma~\ref{th:add1}, we get
\begin{equation}
 \left\Vert P^\perp_{{\psi_{k-1}}}\nu^k_t\right\Vert_{-1} \,  \leq \,C e^{-\gl_1 (t-T_{k-1})} \Vert \nu^k_{T_{k-1}}\Vert_{-1} + C \int_{T_{k-1}}^t \left( 1+\frac{1}{\sqrt{t-s}}\right) \Vert \nu^k_s \Vert_{-1}^2 \dd s
+ \Vert Z^{k,\perp}_t \Vert_{-1}\, ,
\end{equation}
which, since $\lim_N T^4 N^{1-5\gz}=0$ (see Remark \ref{rem:choice T zeta}), leads  for $N$ large enough to
\begin{equation}
 \left\Vert P^\perp_{{\psi_{k-1}}}\nu^k_{T_k}\right\Vert_{-1} \,  \leq \, \frac{N^{3\zeta/2}}{\sqrt{N}}\, .
\end{equation}
This takes care of the third term in the last line of \eqref{eq:decomp nu0}
and by collecting the three estimates we obtain \eqref{ineq:bound sup nuk aux} and the proof is complete.
\end{proof}

\section{The effective dynamics on the tangent space}
\label{sec:dyntan}

\begin{proposition}
\label{th:dyntan}
 We have the first order approximation in probability: for every $\gep>0$
\begin{equation}
\bbP\left(\left\vert
 \sum_{k=1}^n (\psi_k-\psi_{k-1})\, -\, \sum_{k=1}^n \frac{( Z^{k}_{T_k}, q'_{\psi_{k-1}})_{-1,1/q_{\psi_{k-1}}}}{( q',q')_{-1,1/q}}\right \vert \le \gep \right)\, =\, 1\, .
\end{equation}
\end{proposition}
\medskip

\begin{proof}
 Lemma \ref{lem:second order projection} and Proposition \ref{prop:closeness traj manifold} give
 (assuming that $B^N$ is realized: we will do this through all the proof)
\begin{multline}
 \psi_k-\psi_{k-1}\, =\, -\frac{( \nu^k_{T_k}, q'_{\psi_{k-1}})_{-1,1/q_{\psi_{k-1}}}}{( q', q')_{-1,1/q}}\\
 -\frac{1}{2\pi I_0^2(2Kr)}\frac{( \nu^k_{T_k}, (\log q_{\psi_{k-1}})'')_{-1,1/q_{\psi_{k-1}}}}{( q',q')_{-1,1/q}}\frac{( \nu^k_{T_k}, q'_{\psi_{k-1}})_{-1,1/q_{\psi_{k-1}}}}{( q', q')_{-1,1/q}}+o\left(\frac{T}{N}\right)\, .
\end{multline}
Since $\log(q_{\psi_{k-1}})''$ is in $R(L_{q_{\psi_{k-1}}})$, we have 
\begin{equation}
( \nu^k_{T_k},
 (\log q_{\psi_{k-1}})'')_{-1,1/q_{\psi_{k-1}}}=
 ( (\nu^k_{T_k})^\perp, (\log q_{\psi_{k-1}})'')_{-1,1/q_{\psi_{k-1}}},
 \end{equation}
  and thus using again 
 Proposition \ref{prop:closeness traj manifold} 
 we get for the second term of the right-hand side
\begin{equation}
\left\Vert \frac{1}{2\pi I_0^2(2Kr)}\frac{( \nu^k_{T_k}, (\log q_{\psi_{k-1}})'')_{-1,1/q_{\psi_{k-1}}}}{( q',q')_{-1,1/q}}\frac{( \nu^k_{T_k}, q'_{\psi_{k-1}})_{-1,1/q_{\psi_{k-1}}}}{( q', q')_{-1,1/q}} \right\Vert_{-1} \, \leq \, C\frac{1}{\sqrt{N}}N^{2\zeta} \frac{\sqrt{T}}{\sqrt{N}}N^{2\zeta}\, , 
\end{equation}
and hence it is  $o(T/N)$, since $\lim_N N^{4\gz}/\sqrt{T}=0$ (see Remark \ref{rem:choice T zeta}).
So only the component on the tangent space of $M$ at the point $\psi_{k-1}$ is of order 
$T/N$:
\begin{equation}
 \psi_k-\psi_{k-1}\, =\, -\frac{( \nu^k_{T_k}, q'_{\psi_{k-1}})_{-1,1/q_{\psi_{k-1}}}}{( q', q')_{-1,1/q}}+o\left(\frac{T}{N}\right)\, .
\end{equation}
We now decompose this tangent term.
 Our goal is to show that the projection of the noise $Z_{T_k}$ is the only term that gives a non negligible contribution when $N$ goes to infinity. However, a  direct domination of the remainder -- the nonlinear part of the evolution equation \eqref{eq:ito nu^k} -- 
 using the a priori bound $\Vert\nu^k_t\Vert_{-1}\leq\frac{\sqrt{T}}{\sqrt{N}}N^{2\zeta}$ is not sufficient. In fact 
\begin{equation}
 \left|\left( \int_{T_{k-1}}^{T_k} e^{-(T_k-s)L_{q_{\psi_{k-1}}}}\partial_\theta [\nu^k_sJ*\nu^k_s]\dd s, q'_{\psi_{k-1}}\right)_{-1,1/q_{\psi_{k-1}}}\right|\, \leq\, \frac{T^2}{N}N^{4\zeta}\, .
\end{equation}
In order to improve this estimate the strategy is to
 we re-inject \eqref{eq:ito nu^k} into the projection $( I_k(T_k), q_{\psi_{k-1}})_{-1,1/q_{\psi_{k-1}}}$, where
\begin{equation}
 I_k(t)\, =\, \ind_{\{\tau^k_{\gs_N}\geq T_{k-1}\}}\int_{T_{k-1}}^{t\wedge \tau^k_{\gs_N} } e^{-(T_k-s)L_{q_{\psi_{k-1}}}}\partial_\theta [\nu^k_sJ*\nu^k_s]\dd s\, ,
\end{equation}
and this leads to a rather long expression
\begin{equation}
\label{eq:long}
 \sum_{k=1}^n (\psi_k-\psi_{k-1})\, =\, \sum_{k=1}^n \frac{( Z^{k}_{T_k}, q'_{\psi_{k-1}})_{-1,1/q_{\psi_{k-1}}}}{( q',q')_{-1,1/q}}+\sum_{k=1}^n\sum_{i=1}^9A_{k,i}+o(1)\, ,
\end{equation}
with
\begin{equation}
\label{eq:def Aki}
\begin{split}
& A_{k,1}\, =\, \ind_E \times\\
 &\left(\int_\star e^{-(T_k-s)L_{q_{\psi_{k-1}}}}\partial_\theta \left[e^{-(s-T_{k-1})L_{q_{\psi_{k-1}}}}\nu^k_{T_{k-1}}J*\left(e^{-(s-T_{k-1})L_{q_{\psi_{k-1}}}}\nu^k_{T_{k-1}}\right)\right]\dd s, q'_{\psi_{k-1}}\right)_\star\\
& A_{k,2}\, =\,  \ind_E\left(\int_\star e^{-(T_k-s)L_{q_{\psi_{k-1}}}}\partial_\theta \left[I_k(s)J*I_k(s)\right]\dd s, q'_{\psi_{k-1}}\right)_\star\\
& A_{k,3}\, =\,  \ind_E\left(\int_\star e^{-(T_k-s)L_{q_{\psi_{k-1}}}}\partial_\theta \left[Z^k_sJ*Z^k_s\right]\dd s, q'_{\psi_{k-1}}\right)_\star\\
 & A_{k,4}\, =\, \ind_E\left(\int_\star e^{-(T_k-s)L_{q_{\psi_{k-1}}}}\partial_\theta \left[e^{-(s-T_{k-1})L_{q_{\psi_{k-1}}}}\nu^k_{T_{k-1}}J*I_k(s)\right]\dd s, q'_{\psi_{k-1}}\right)_\star\\
& A_{k,5}\, =\, \ind_E\left(\int_\star e^{-(T_k-s)L_{q_{\psi_{k-1}}}}\partial_\theta \left[I_k(s)J*\left(e^{-(s-T_{k-1})L_{q_{\psi_{k-1}}}}\nu^k_{T_{k-1}}\right)\right]\dd s, q'_{\psi_{k-1}}\right)_\star\\
& A_{k,6}\, =\,  \ind_E\left(\int_\star e^{-(T_k-s)L_{q_{\psi_{k-1}}}}\partial_\theta \left[e^{-(s-T_{k-1})L_{q_{\psi_{k-1}}}}\nu^k_{T_{k-1}}J*Z^k_s\right]\dd s, q'_{\psi_{k-1}}\right)_\star\\
& A_{k,7}\, =\,  \ind_E\left(\int_\star e^{-(T_k-s)L_{q_{\psi_{k-1}}}}\partial_\theta \left[Z^k_sJ*\left(e^{-(s-T_{k-1})L_{q_{\psi_{k-1}}}}\nu^k_{T_{k-1}}\right)\right]\dd s, q'_{\psi_{k-1}}\right)_\star\\
& A_{k,8}\, =\,  \ind_E\left(\int_\star e^{-(T_k-s)L_{q_{\psi_{k-1}}}}\partial_\theta \left[I_k(s)J*Z^k_s\right]\dd s, q'_{\psi_{k-1}}\right)_\star\\
& A_{k,9}\, =\,  \ind_E\left(\int_\star e^{-(T_k-s)L_{q_{\psi_{k-1}}}}\partial_\theta \left[Z^k_sJ*I_k(s)\right]\dd s, q'_{\psi_{k-1}}\right)_\star\, ,
\end{split}
\end{equation}
where  we have used the shortcuts $E=\{\tau^k_{\gs_N}\geq T_{\psi_{k-1}}\}$,
$\int_\star$ stands for $\int_{T_{k-1}}^{T\wedge \tau^k_{\gs_N}}$ and 
and $(\cdot, \cdot)_\star$ is
$(\cdot, \cdot)_{-1,1/q_{\psi_{k-1}}}$.

The following bound (a direct consequence of Lemma \ref{th:add1} and \ref{lem:H-1 H-2}) is now going to be of help:
\begin{multline}
\label{eq:bound-add}
 \left\Vert \int_{T_{k-1}}^{T_k} e^{-(T_k-s)L_{q_{\psi_{k-1}}}}\partial_\theta [h_1(s)J*h_2(s)]\dd s\right\Vert_{-1}\\ \leq\, 
C\int_{T_{k-1}}^{T_k} \left(1+\frac{1}{\sqrt{T_k-s}}\right)\Vert h_1(s)\Vert_{-1} \Vert h_2(s)\Vert_{-1}\dd s\, .
\end{multline}
In fact it is not difficult to see that  by using Lemma \ref{lem:bound Zk}, Proposition \ref{prop:closeness traj manifold} and \eqref{eq:bound-add}
we can efficiently bound  all the $A_{k,j}$'s, except $A_{k,3}$: 
\begin{multline}
| A_{k,1} |\, \leq \, \frac{1}{N}N^{5\zeta}, \ \ 
| A_{k,2} |\, \leq \, \frac{T^5}{N^2}N^{9\zeta}, \ \
| A_{k,4} |\, \leq \, \frac{T^2}{N^{3/2}}N^{7\zeta}, \ \
| A_{k,5} |\, \leq \, \frac{T^2}{N^{3/2}}N^{7\zeta}, \\
| A_{k,6} |\, \leq \, \frac{T^{1/2}}{N^{3/2}}N^{4\zeta}, \ \
| A_{k,7} |\, \leq \, \frac{T^{1/2}}{N^{3/2}}N^{4\zeta}, \ \
| A_{k,8} |\, \leq \, \frac{T^{7/2}}{N^{3/2}}N^{6\zeta} \ \ and \ \ 
| A_{k,9} |\, \leq \, \frac{T^{7/2}}{N^{3/2}}N^{6\zeta} \, .
\end{multline}
Since $T^4N^{9\zeta-1}\rightarrow 0$ and $N^{5\zeta}/T\rightarrow 0$ (see Remark \ref{rem:choice T zeta}), we get (recall that $n=n_N=\frac{N}{T}$)
\begin{equation}
\label{eq:lesslong}
 \sum_{k=1}^n (\psi_k-\psi_{k-1})\, =\, \sum_{k=1}^n \frac{( Z^{k}_{T_k}, q'_{\psi_{k-1}})
 _{-1,1/q_{\psi_{k-1}}}
 }{( q',q'\rangle}+\sum_{k=1}^n A_{k,3}+o(1)\, .
\end{equation}

For the $A_{k,3}$ terms we need to use something more sophisticate. 
To deal with these terms in fact we rely on an averaging
phenomena. This method has been used in \cite{cf:BBB} for the same kind of problem. We write the Doob decomposition
\begin{equation}
\label{eq:doob dec Fk}
 \sum_{k=1}^{m}  A_{k,3} \, =\, M_m +\sum_{k=1}^m \gamma_k\, ,
\end{equation}
where
\begin{equation}
 \gamma_k\, =\, \bbE\left[ A_{k,3}| \cF_{T_{k-1}}\right]\, ,
\end{equation}
and $M_m$ is a $ \cF_{T_m}$-martingale with brackets
\begin{equation}
 \langle M\rangle_m\, =\, \sum_{k=1}^m \left( \bbE\left[ A_{k,3}^2|\cF_{T_{k-1}}\right]-\gamma_k^2\right)\, .
\end{equation}
We have
\begin{multline}
 \gamma_k\, =\,
\bbE\Bigg[\frac{1}{N^2}\sum_{i,j=1}^N\int_{T_{k-1}}^{T_k\wedge \tau^k_{\gs_N}}\dd W^i_s\int_{T_{k-1}}^{T_k\wedge \tau^k_{\gs_N}}\dd W^j_{s'}\int_\bbS \dd \theta
 \bigg(1
-\frac{1}{2\pi I^2_0(2Kr)q_{\psi_{k-1}}(\theta)}\bigg) \\  \int_\bbS\dd  \theta''
 \partial\cG^{\psi_{k-1}}_{T_k-s}(\theta,\varphi^{i,N}_s)J(\theta-\theta'')\partial\cG^{\psi_{k-1}}_{T_k-{s'}}(\theta'',\varphi^{j,N}_{s'})\bigg| \cF_{T_{k-1}} \Bigg]\ind_{\tau^{k}_{\gs_N}\geq T_{k-1}} \, ,
\end{multline}
where $\partial \cG^{\psi}_{t}(\theta,\theta'):=\partial_{\theta' } \cG^{\psi}_{t}(\theta,\theta')$, and from this we obtain
\begin{multline}
\label{eq:gammak5}
 \gamma_k\, =\, 
 \bbE\Bigg[\frac{1}{N}\int_{0}^{T\wedge \tilde \tau_{\gs_N}} \dd s \int_\bbS \tilde\mu_s(\dd \theta') 
 \int_\bbS\dd \theta 
 \bigg(1-\frac{1}{2\pi I^2_0(2Kr)q_{\psi_{k-1}}(\theta)}\bigg) 
 \\
 \int_\bbS\dd  \theta''
 \partial_{\theta'}\cG^{\psi_{k-1}}_{T-s}(\theta,\theta')J(\theta-\theta'')\partial_{\theta'}\cG^{\psi_{k-1}}_{T-s}(\theta'',\theta') 
 \Bigg]
 \ind_{\tau^{k}_{\gs_N}\geq T_{k-1}} \, ,
\end{multline}
with
\begin{equation}
 \tilde \mu_s\, :=\,\frac 1N \sum_{j=1}^N \gd_{\tilde\varphi^{j,N}_s}\, ,
\end{equation}
where  $\{\tilde \gp^{j,N}_s\}_{s \ge 0}$ is a solution of \eqref{eq:evol} 
depending on  $\cF_{T_{k-1}}$ only through
the initial condition
\begin{equation}
 \tilde \varphi^{j,N}_0\, =\, \varphi^{j,N}_{T_{k-1}}\, .
\end{equation}
The stopping time $\tilde \tau_{\gs_N}$ is defined as follows:
\begin{equation}
\tilde \tau_{\gs_N}\, :=\,  \inf\{s>0,\, \Vert \tilde\mu_s-q_{\psi_{T_{k-1}}}\Vert_{-1} >\gs_N\}\, .
\end{equation}
We now 
write  $ \tilde\mu_s(\dd \theta')=q_{\psi_{k-1}}(\theta')\dd \theta' + \tilde\nu^k_s(\dd \theta')$ 
and split the 
for right hand-side of \eqref{eq:gammak5} into the corresponding two terms. 

The term coming $q_{\psi_{k-1}}(\theta')\dd \theta'$ is zero as one can see
by using the symmetry:
\begin{equation}
\cG^{\psi_{{k-1}}}_s(\psi_{k-1}+\theta,\psi_{k-1}+\theta')\, =\, \cG^{\psi_{{k-1}}}_s(\psi_{k-1}-\theta,\psi_{k-1}-\theta')\, 
\end{equation}
which follows from the same statement with $\psi_{{k-1}}=0$, which in turn is a consequence
of the representation \eqref{sec:A_H} and of the fact that if $e_j$ is even (respectively, odd)
then $f_j$ is even (respectively, odd) too  (see Section~\ref{sec:linear} and Section~\ref{sec:A_H}).

For the  term containing  $\tilde\nu^k_s(\dd \theta')$ instead we get the bound
\begin{multline}
\label{bound integral term gammak}
 \Bigg| \bbE\Bigg[\frac{1}{N}\int_{T_{k-1}}^{T_k\wedge \tilde\tau_{\gs_N}} \dd s \int_\bbS \dd \tilde\nu^k_s(\theta') \int_\bbS\dd \theta \int_\bbS\dd  \theta''
 \bigg(1-\frac{1}{2\pi I^2_0(2Kr)q_{\psi_{k-1}}(\theta)}\bigg) \\
 \partial_{\theta'}\cG^{\psi_{k-1}}_{T_k-s}(\theta,\theta')J(\theta-\theta'')\partial_{\theta'}\cG^{\psi_{k-1}}_{T_k-s}(\theta'',\theta') \Bigg]
\Bigg| \\
\leq\, \bbE\Bigg[ \frac{1}{N}\int_{T_{k-1}}^{T_k\wedge \tau^k_{\gs_N}} \dd s \Vert \tilde\nu^k_s\Vert_{-1} \Vert H^k_s\Vert_{H_1} \Bigg]
\end{multline}
where
\begin{equation}
 H^k_s(\theta')\, =\,   \int_\bbS\dd \theta \int_\bbS\dd  \theta''
 \bigg(1-\frac{1}{2\pi I^2_0(2Kr)q_{\psi_{k-1}}(\theta)}\bigg) 
 \partial_{\theta'}\cG^{\psi_{k-1}}_{T_k-s}(\theta,\theta')J(\theta-\theta'')\partial_{\theta'}\cG^{\psi_{k-1}}_{T_k-s}(\theta'',\theta')\, .
\end{equation}
We  now plug in the explicit representation for the kernels:
\begin{multline}
 H^k_s(\theta')\, =\, \sum_{l_1,l_2=0}^\infty e^{-(\gl_{l_1}+\gl_{l_2})(T_k-s)} \int_0^{2\pi}\dd \theta \int_0^{2\pi}\dd  \theta''
 \bigg(1-\frac{1}{2\pi I^2_0(2Kr)q_{\psi_{k-1}}(\theta)}\bigg)\\
 e_{\psi_{k-1},l_1}(\theta)J(\theta-\theta'')e_{\psi_{k-1},l_2}(\theta'')f'_{\psi_{k-1},l_1}(\theta')f'_{\psi_{k-1},l_2}(\theta')\, .
\end{multline}
We obtain
\begin{multline}
\label{bound vert Hk}
 \Vert H^k_s \Vert_1 \, \leq \, \sum_{l,m=0}^\infty e^{-(\gl_{l_1}+\gl_{l_2})(T_k-s)} \Bigg|  \int_0^{2\pi}\dd \theta \int_0^{2\pi}\dd  \theta''
 \bigg(1-\frac{1}{2\pi I^2_0(2Kr)q_{\psi_{k-1}}(\theta)}\bigg)\\
 e_{\psi_{k-1},l_1}(\theta)J(\theta-\theta'')e_{\psi_{k-1},l_2}(\theta'') \Bigg| \left(\Vert f''_{\psi_{k-1},l_1} \Vert_2 \Vert f'_{\psi_{k-1},l_2}\Vert_\infty  +\Vert f'_{\psi_{k-1},l_1}\Vert_\infty \Vert f''_{\psi_{k-1},l_2}\Vert_2\right)\, .
\end{multline}
We aim at proving the convergence of this sum. For the integral term, thanks to the rotation symetry, we can limit the study to $\psi_{k-1}=0$. Since $J(\theta-\theta'')=-K\sin(\theta-\theta'')=-K\sin(\theta)\cos(\theta'')+K\cos(\theta)\sin(\theta'')$, we can split these double integrals into products of two simple ones. Corollary \ref{cor: unitary eigenfunction expansion} implies that there exists $l_0$ in $\bbN$ such that $e_{0,l_0+2p}$ and $e_{0,l_0+2p+1}$ can be writen as
\begin{equation}
 e_{0,l_0+2p}\, =\, p q_0^{1/2} (c_{1,l_0+2p}v_{1,l_0+p}+c_{2,l_0+2p}v_{2,l_0+p})+O\left(\frac{1}{p}\right)\, ,
\end{equation}
\begin{equation}
 e_{0,l_0+2p+1}\, =\, p q_0^{1/2} (c_{1,l_0+2p+1}v_{1,l_0+p}+c_{2,l_0+2p+1}v_{2,l_0+p})+O\left(\frac{1}{p}\right)\, ,
\end{equation}
where
\begin{equation}
 \sup_{l\geq l_0}\{ |c_{1,l}|,|c_{2,l}| \}\, <\, \infty
\end{equation}
and the functions $v_{i,l}$ are defined in Proposition \ref{prop:eigenvalues and eigenfunction expansion}. The $v_{i,l}$ are sums and products of sines and cosines, and there exists $h\in\bbN$ such that the only non-zero Fourier coefficients of $v_{i,l_0+p}$ are of index included between $h+p-2$ and $h+p+2$ and are bounded with respect to $p$. We deduce that the simple integral terms containing $e_{0,l_0+2p}$, which are of the form
\begin{equation}
 C\int_0^{2\pi} q_0^{\pm 1/2}(\theta) e_{0,l_0+2p}(\theta)g(\theta)\dd\theta\, ,
\end{equation}
where $g$ is sine or cosine and $C$ is a constant independent of $p$, are up to a correction of order $1/p$ a bounded linear combination of the Fourier coefficients of $q_0^{1/2}$ or $q_0^{-1/2}$ of index taken between $h+p-3$ and $h+p+3$. The same argument applies for $e_{0,l_0+2p+1}$. Since these Fourier terms decrease faster than exponentially 
(this can be seen by observing that  $\int_\bbS \exp(a \cos \theta) \dd\theta = 2\pi I_n(a)$ and  that 
$\exp(a \cos(\cdot))$ is an entire function), 
these simple integral terms are of order $1/p$. Using Remark \ref{rem:eigenvalues and eigenfunction expansion} and Corollary \ref{cor:fj} we deduce the following bound for $\Vert H^k_s\Vert_1$:
\begin{equation}
 \Vert H^k_s \Vert_1 \, \leq\, C+C\sum_{l_1,l_2=1}^\infty \frac{l_1+l_2}{l_1l_2}e^{(T_k-s)\frac{l_1^2+l_2^2}{C}}
+C\sum_{l=1}^\infty e^{(T_k-s)\frac{l^2}{C}}\, ,
\end{equation}
where the first term of the right hand side corresponds to the case $l_1=0,l_2=0$ in \eqref{bound vert Hk}, the second term corresponds to $l_1>0,l_2>0$ and the third term to
$l_1=0,l_2>0$ or $l_2=0,l_1>0$. Applying \eqref{bound integral term gammak} and Proposition \ref{prop:closeness traj manifold}, we get:
\begin{multline}
|\gamma_k|\, \leq \, C\frac{T^{1/2}}{N^{3/2}}N^{2\gz}\int_{T_{k-1}}^{T_k} \dd s \Vert H^k_s\Vert_1  
\leq\,C\frac{T^{1/2}}{N^{3/2}}N^{2\gz}\left( T + \sum_{h,l=1}^\infty \frac{h+l}{hl(h^2+l^2)}
+\sum_{l=1}^\infty \frac{1}{l^2}\right)\\
\leq\, C\frac{T^{3/2}}{N^{3/2}}N^{2\gz}\, ,
\end{multline}
and thus for $N$ large enough
\begin{equation}
\label{bound final gammak}
\sum_{k=1}^n |\gamma_k|\, \leq\,   \sqrt{\frac{T}{N}}N^{3\zeta}\, .
\end{equation}
On the other hand, applying Doob Inequality, \eqref{eq:def Aki} and Proposition \ref{prop:closeness traj manifold}, it comes
\begin{multline}
\label{doob ineq Ak3}
 \bbP\left[ \sup_{1\leq m\leq n} |M_m|\geq \sqrt{\frac{T}{N}}N^{3\zeta}\right]\, \leq\,\frac{N}{T N^{6\zeta}}\bbE\left[\langle M\rangle_n\right]\, \leq\,\frac{N^{1-6\zeta}}{T}\sum_{m=1}^n\bbE\left[A_{k,3}^2\right]\, 
\leq\, \frac{T^3}{N^{1+2\zeta}}\, .
\end{multline}
Since $T N^{-1-2\zeta}\rightarrow 0$ (see Remark \ref{rem:choice T zeta}), the combination of \eqref{bound final gammak} and \eqref{doob ineq Ak3} leads to
\begin{equation}
 \bbP\left[\left|\sum_{i=1}^n A_{k,3}\right|\geq \sqrt{\frac{T}{N}}N^{3\zeta}\right]\rightarrow 0\, ,
\end{equation}
and the proof is complete.


\end{proof}

\section{Approach to $M$}
\label{sec:approachM}

The long time behavior of the solutions to \eqref{eq:K} is rather well understood, so, in particular, 
we know that if $p_0$ is not on the set attracted to the unstable solution $\frac 1{2\pi}$, then it converges to
one probability density in $M$ (cf. Proposition~\ref{th:PDE}): this is directly extracted from
\cite{cf:BGP,cf:GPP}). This takes care of the first stage of the evolution, that is the approach to a small
neighborhood of $M$ also for the empirical measure. Things however change
when 
the empirical measure is
at a distance from $M$ which vanishes as $N \to \infty$, because the noise starts playing a role
and the difference between the empirical measure and the solution to the Fokker-Planck PDE \eqref{eq:K}
is no longer negligible. But we do need to get to distances of about $N^{-1/2}$ and this is done
by exploiting the {\sl approximate} contracting properties of the dynamics when the empirical measure is close 
to $q_\psi$. We talk about {\sl approximate} contracting properties because the noise plays against getting to $M$ and limits
the contraction effect of the linearized operator. Nevertheless, the proof mimics the deterministic proof
of nonlinear stability, to which the control of the noise is added. In principle the argument is straightforward:
one exploits the spectral gap of the linearized evolution. In practice, one has to set up an iterative procedure
similar to the one developed in Section~\ref{sec:apriori}, because the center of synchronization may change 
somewhat 
over long times. This procedure is however substantially easier than the one presented in  Section~\ref{sec:apriori},
mostly because here the control required on the noise is for substantially shorter times ($\log N$ versus $N$!),
so we will not go through the arguments in full detail again.

\medskip

\begin{proposition}
\label{th:approachM}
Choose   $p_0\in\cM_1 \setminus   U$ such that \eqref{eq:init-main} is satisfied.
Then there exists  $\psi_0$ (non random!), that depends on $K$ and $p_0(\cdot)$, $C$, that depends
only on $K$,  and a random variable $ \Psi_N$ such that 
\begin{equation}
\label{eq:approch M}
\lim_{N \to \infty}
\bbP \left( 
\left\Vert \mu_{N, \tilde \gep_N N} - q_{\Psi_N} \right \Vert_{-1}\, \le\,  \frac{N^{2\zeta}}{\sqrt{N}} \right) \, =\, 1\, ,
\end{equation}
where $\tilde \gep_N := \lfloor C \log N \rfloor /N$, and $\lim_N\Psi_N= \psi_0$ in probability. Moreover for $\gep$ and $\gep_N$ 
as in Theorem~\ref{th:main} we have
\begin{equation}
\label{eq:approch M2}
\lim_{N \to \infty}
\bbP \left( \sup_{t \in [\gep_N N,\tilde \gep_N N ]}
\left\Vert \mu_{N, t} - q_{\psi_0} \right \Vert_{-1}\, \le\,  \gep \right) \, =\, 1\, ,
\end{equation}
\end{proposition}

\begin{proof}
 
The proof is divided in two parts. First we prove, using the convergence of $\mu_t=\mu_{N,t}$ to the deterministic solution $p_t$, that for a given $h>0$ (arbitrarily small), there exists $t_0$ such that for $\gep$ small enough, $\bbP(\text{dist}(\mu_{N,t_0},M)\leq h)\rightarrow 1$ when $N\rightarrow \infty$. Then we show that after a time of order $\log N$, the empirical measure $\mu_{t}$ moves to a distance $N^{\zeta-1/2}$ from $M$, without a macroscopic change of the phase.

\medskip

The first part of the proof relies on the following result:
\medskip

\begin{proposition}
\label{th:PDE}
If $p_0\in \cM_1\setminus U$ then there exists $\psi\in \bbS$ such that
$\lim_{t \to \infty} p_t = q_\psi$ in $C^k(\bbS; \bbR)$ (for every $k$). 
\end{proposition}

\medskip

Proposition~\ref{th:PDE}  is essentially  taken from
\cite{cf:GPP}, in the sense that it follows by piecing together some results taken from 
\cite{cf:GPP}. We give below a proof that of course relies on \cite{cf:GPP}.
We point out that the very same result can be proven also by adapting entropy production 
arguments, like in \cite{cf:ABM}. 

\medskip

Proposition~\ref{th:PDE} guarantees that  the deterministic solution $p_t$ converges to a element $q_{\psi_0}$ of $M$.
Therefore for $t\ge t_0$, we have that  $p_t$ is no farther than  $h/2$ from $q_{\psi_0}$ (this is a statement that can be made for example in $C^k$, but here we just need
it in $H_{-1}$). Actually, it is not difficult to see that one can choose $t_0 = -\frac{2}{\gl_1}\log h$, for $h$ sufficiently small
($\gl_1$ is the spectral gat of $L_{q_{\psi_0}}$), but this is of little relevance here. 
 Applying the It\^o formula 
\begin{equation}
\label{eq:diff mu p}
 \mu_t-p_t\, =\,e^{t\frac{\Delta}{2}}( \mu_0-p_0)-\int_0^t e^{(t-s)\frac{\Delta}{2}}[\mu_s J*\mu_s-p_sJ*p_s]\dd s + z_t\, ,
\end{equation}
where
\begin{equation}
 z_t\, =\, \frac{1}{N}\sum_{j=1}^N \partial_{\theta'}\cH(\theta,\phi^{j,N}_s)\dd W^j_s\, ,
\end{equation}
$e^{t\frac \Delta 2}$ is the semi-group of the Laplacian and $\cH$ is the kernel of $e^{s\frac{\Delta}{2}}$ in $\bbL^2$. Define $W_N=\{ w,\, \Vert \mu_0-p_0\Vert_{-1} \leq \gep\}$. Using the classical estimate $\Vert e^{t \Delta/2}u\Vert_{-1}\leq \frac{C}{\sqrt{t}}\Vert u\Vert_{-2}$ and similar argument as in Section \ref{sec:apriori}, we deduce that there exist events $\tilde W_N\subset W_N$ such that $\bbP(\tilde W_N)\rightarrow 1$ and that for all outcomes in $\tilde W_N$ we have
\begin{equation}
 \sup_{0\leq t\leq t_0}\Vert z_t\Vert_{-1}\leq \sqrt{\frac{t_0}{N}}N^\zeta\, .
\end{equation}
From now, we restrict ourselves to $\tilde W_N$.
From \eqref{eq:diff mu p} we get for all $t\in[0,t_0]$
\begin{equation}
 \Vert \mu_t-p_t \Vert_{-1}\, \leq\, \gep+C\int_0^t\frac{1}{\sqrt{t-s}}\Vert \mu_s-p_s \Vert_{-1} \dd s + \sqrt{\frac{t_0}{N}}N^\zeta\, .
\end{equation}
The Gronwall-Henry inequality (see \cite{cf:SellYou}) implies that there exists $\gamma>0$ (independent of $\gep$ and $N$) such that
\begin{equation}
\sup_{t \le t_0}
 \Vert \mu_{t_0}-p_{t_0} \Vert_{-1}\, \leq \, \left( \gep + \sqrt{\frac{t_0}{N}}N^\zeta \right)e^{\gamma t_0}\, .
\end{equation}
So for $\gep=h/4$ and $N$ large enough, $\Vert \mu_{t_0}-q_{\psi_{0}} \Vert_{-1} \leq h$ on the event $\tilde W_N$.

\medskip

To show that we enter a neighborhood of size {\sl slightly larger} than $N^{-1/2}$, it will be $N^{2 \zeta-1/2}$,
we set up an  iterative scheme. It is very similar to the one given in Section \ref{sec:sheme}, but with times $t_i$ bounded with respect to $N$. This times are chosen such that after each iteration, the distance between the empirical measure and $M$ is at least divided by $2$. We define $h_0:=h$ and for $m \geq 1$
\begin{equation}
 t_m\, :=\,  t_{m-1}+\frac{1}{\gl_1}|\log \ga|\, ,
\end{equation}
\begin{equation}
 h_{m}\, :=  \frac 12 h_{m-1}\, ,
\end{equation}
until the index $m_f$ defined by
\begin{equation}
 m_f\, :=\, \inf\left\{m\geq 1,\, h_m\leq N^{2\gz-1/2}\right\}\, .
\end{equation}
The constant $\ga$ above does not depens on $N$ and will be chosen below.
It is now easy to check that $m_f$ is of order $\log N$.
Then we define $\tilde \tau_0\, :=\, t_0$, and for $1\leq m\leq m_f+1$
\begin{equation}
 \tilde \psi_{m-1}\, :=\, \proj (\mu_{t_{m-1}})\, ,
\end{equation}
\begin{equation}
 \tilde\nu^{m}_{t_{m-1}}\, :=\, \mu_{t_{m-1}}-q_{\tilde\psi_{m-1}}
\end{equation}
if $\text{dist}(\mu_{t_{m-1}},M)\leq \gs$ (see Lemma~\ref{lem:def proj on M}).  We consider for $1\leq m\leq m_f$ the stopping times
\begin{multline}
 \tilde\tau_m\, :=\, \tilde\tau_{m-1} \ind_{\{\tilde\tau_{m-1}<t_{m-1}\}}\\+\inf\{s\in[t_{m-1},t_m],\, \Vert \mu_s-q_{\tilde\psi_{m-1}}\Vert_{-1}\geq  \gs\}\ind_{\{\tilde\tau_{m-1}\geq t_{m-1}\}}\, .
\end{multline}
and the process solution of
\begin{multline}
\label{eq:def tilde nu}
 \tilde\nu^{m}_{t}\, =\, 
 \ind_{\{\tilde\tau_m<t_{m-1}\}}\tilde\nu^m_{\tilde\tau_{m}} +  \ind_{\{\tilde\tau_m\geq t_{m-1}\}} \times \\
\left(e^{-(t\wedge \tilde\tau_m-t_{m-1}) L_{\tilde\psi_{m-1}}} \tilde\nu^{m}_{t_{m-1}}-\int_{t_{m-1}}^{t\wedge \tilde\tau_m} e^{-(t\wedge \tilde\tau_m -s)L_{\tilde\psi_{m-1}}}\partial_\theta[\tilde\nu^{m}_s J*\tilde\nu^{m}_s]\dd s +\tilde Z^{m}_{t\wedge \tilde\tau_m}\right)\, ,
\end{multline}
where 
\begin{equation}
 \tilde Z^m_{t}\, =\, \frac1N \sum_{j=1}^N \int_{t_{m-1}}^t \partial_{\theta'}\cG^{\tilde\psi_{m-1}}_{t-s}\left(\theta,\varphi^{j,N}_s\right)\dd W^j_s\, .
\end{equation}
With the same arguments as given in Lemma \ref{lem:bound Zk}, we can prove (recall that $m_f$ is of order $\log N$) that the probability of the event
\begin{equation}
 \gO_N\, :=\, \left\{ \sup_{1\leq m\leq m_f} \sup_{t_{m-1}\leq t\leq t_m} \left\Vert\tilde Z^m_t\right\Vert_{-1} \leq \sqrt{\frac{t_m-t_{m-1}}{N}}N^{\zeta}\right\}
\end{equation}
tends to $1$ when $N\rightarrow\infty$. From now, we assume that $\gO_N$ is verified. We insist on the fact that the generic constants $C$ appearing in the following do not depend on $N$, and if not mentioned do not depend on $\ga$. From Lemma \ref{th:add1} and \eqref{eq:def tilde nu} we get thet for all $1\leq m\leq m_f$,
\begin{equation}
\label{bound:tilde nu}
 \Vert \tilde\nu^m_t\Vert_{-1}\, \leq \, C h_{m-1}+C(t+\sqrt{t})\sup_{s\in [t_{m-1},t]}\Vert \tilde\nu^m_s\Vert_{-1}+\sqrt{\frac{t_m-t_{m-1}}{N}}N^{\zeta}\, .
\end{equation}
We now prove that for $1\leq m\leq m_f-1$, $\Vert \nu^{m}_{t_{m-1}}\Vert_{-1}\leq h_{m-1}$ implies $\Vert\nu^{m+1}_{t_m}\Vert_{-1}\leq h_{m}$, and that $\Vert \tilde\nu^{m_f}_{t_{m_f-1}}\Vert_{-1}\leq h_{m_f-1}$ implies $\Vert\tilde\nu^{m_f+1}_{t_{m_f}}\Vert_{-1}\leq N^{2\gz-1/2}$. Define
\begin{equation}
 s^*_m\, :=\, \sup\{s\in [t_{m-1},t_m],\, \Vert \tilde\nu^m_s\Vert_{-1}\leq h_{m-1}^{3/4}\}\, . 
\end{equation}
Then for $s<s^*_m$, if $\Vert \nu^{m}_{t_{m-1}}\Vert_{-1}\leq h_{m-1}$ ,we get using \eqref{bound:tilde nu}
\begin{equation}
\label{bound:tilde nu 2}
\Vert\tilde \nu^m_s\Vert_{-1}\, \leq\, Ch_{m-1}+C(s+\sqrt{s})h_{m-1}^{3/2}+\sqrt{\frac{t_{m}-t_{m-1}}{N}}N^{\zeta}\, .
\end{equation}
Since $N^{2\gz-1/2}\leq h_{m-1}$, we deduce that $s^*_m=t_m$ if $h_0$ is small enough. Then using \eqref{eq:def tilde nu} we get
\begin{equation}
\Vert \tilde\nu^m_{t_m}\Vert_{-1}\, \leq\, C\ga h_{m-1} +C h_{m-1}^{3/2} + \sqrt{\frac{t_m-t_{m-1}}{N}}N^\zeta\, .
\end{equation}
Since $ h_{m-1}^{3/2} \leq \ga h_{m-1}$ for $h_0$ small enough, it leads us to (recall that $h_{m-1}=2h_m$)
\begin{equation}
\label{bound:tilde nu 3}
 \Vert \tilde\nu^m_{t_m}\Vert_{-1}\, \leq\, 4C\ga h_{m}+ \sqrt{\frac{t_m-t_{m-1}}{N}}N^\zeta\, .
\end{equation}
If $m<m_f$, $\sqrt{\frac{t_m-t_{m-1}}{N}}N^\zeta\leq C\ga h_{m}$ and thus $\Vert \tilde\nu^m_{t_m}\Vert_{-1}\leq 5C\ga h_{m}$. If $m=m_f$, $h_m\leq N^{2\gz-1/2}$ and thus $\Vert \tilde\nu^m_{t_m}\Vert_{-1}\leq 5C\ga N^{2\gz-1/2}$. 
We now have a good control on $\mu_{t_m}=q_{\tilde\psi_{m-1}}+\tilde\nu^m_{t_m}$, and project it with respect to $\tilde \psi_{m}$ (writing $\mu_{t_m}=q_{\tilde\psi_{m}}+\tilde\nu^{m+1}_{t_m}$) to get a bound for $\Vert \tilde\nu^{m+1}_{t_m}\Vert_{-1}$. We use the same decomposition as the proof of Proposition \ref{prop:closeness traj manifold}:
\begin{multline}
\label{eq:decomp tilde nu}
 \tilde \nu^{m+1}_{t_m}\, =\, q_{\tilde\psi_{m-1}}+\tilde\nu^m_{t_m}-q_{\tilde\psi_{m}} 
                          =\, P^\perp_{\tilde\psi_{m}}[q_{\tilde\psi_{m-1}}+\tilde\nu^m_{t_m}-q_{\tilde\psi_{m}}] \\
                          =\, \left(P^\perp_{\tilde\psi_{m}}-P^\perp_{\tilde\psi_{m-1}}\right)[q_{\tilde\psi_{m-1}}+\tilde\nu^m_{t_m}-q_{\tilde\psi_{m}}] +P^\perp_{\tilde\psi_{m-1}}[q_{\tilde\psi_{m-1}}-q_{\tilde\psi_{m}}]+P^\perp_{\tilde\psi_{m-1}}\tilde\nu^m_{t_m}\, .
\end{multline}
Since the projection $\proj$ is smouth, we get the bound
\begin{equation}
\label{bound:diff tilde psi}
 |\tilde\psi_m-\tilde\psi_{m-1}|\, =\, |\proj(\mu_{t_m})-\proj(\mu_{t_{m_1}})|\, \leq\,  C\Vert \mu_{t_m}-\mu_{t_{m_1}}\Vert_{-1}\, \leq\, C\Vert \tilde\nu^m_{t_m}-\tilde\nu^m_{t_{m-1}}\Vert_{-1}\, .
\end{equation}
But \eqref{bound:tilde nu 3} implies in particular that 
\begin{equation}
\Vert \tilde\nu^m_{t_m}\Vert_{-1}\leq C(1+4\ga)h_{m-1} \, ,
\end{equation}
which implies, using also \eqref{bound:diff tilde psi},
\begin{equation}
\label{bound:diff tilde psi 2}
|\tilde\psi_m-\tilde\psi_{m-1}|\, \leq\, 2C(1+4\ga)h_{m-1}\, .
\end{equation}
Using similar arguments as in the proof of Proposition \ref{prop:closeness traj manifold} (using in particular the smouthness of the projection $P^\perp_\psi$), we see that the two first terms of the right hand side in \eqref{eq:decomp tilde nu} are of order $h_{m-1}^2$. More precisely, there exists a constant $C'[\ga]$ depending in $\ga$ (increasing in $\ga$) such that  
\begin{equation}
 \Vert \tilde \nu^{m+1}_{t_m}\Vert_{-1}\, \leq C'[\ga]h_{m-1}^2+C\Vert \tilde\nu^m_{t_m}\Vert_{-1}\, .
\end{equation}
So, since $\Vert \tilde\nu^m_{t_m}\Vert_{-1}\leq 5C\ga h_{m}$ for $m<m_f$ and $\Vert \tilde\nu^{m_f}_{t_{m_f}}\Vert_{-1}\leq 5C\ga N^{2\gz-1/2}$, if $h_0$ and $\ga$ are small enough 
we get $\Vert \tilde\nu^{m+1}_{t_m}\Vert_{-1}\leq h_{m}$ for $m<m_f$ and $\Vert \tilde\nu^{m_f+1}_{t_{m_f}}\Vert_{-1}\leq N^{2\gz-1/2}$.

We have therefore shown that after a time of order $\log N$, the empirical measure comes at distance $N^{2\zeta-1/2}$ from from $q_{\tilde\psi_{m_f}}$. This angle $\tilde\psi_{m_f}$ corresponds to the angle $\Psi_N$ in the Proposition~\ref{th:approachM}. So it remains to prove that $\tilde\psi_{m_f}$ converges to $\psi_0$ in probability as $N$ goes to infinity. 
We decompose
\begin{equation}
 |\tilde\psi_{m_f}-\psi_0|\, \leq\, |\tilde\psi_0-\psi_0|+\sum_{m=1}^{m_f}|\tilde\psi_{m}-\tilde\psi_{m-1}|\, .
\end{equation}
We restrict our study on the event $\gO_N\bigcap\tilde W_N$, whose probability tends to $1$. Since $\Vert \mu_{t_0}-q_{\psi_0}\Vert_{-1}\leq h$ and the projection $\proj$ is smouth, we get
\begin{equation}
 |\tilde\psi_0-\psi_0|\leq C h
\end{equation}
and \eqref{bound:diff tilde psi 2} implies (recall $h_0=h$ and $h_{m-1}=2h_m$)
\begin{equation}
 |\tilde\psi_{m}-\tilde\psi_{m-1}|\leq C 2^{1-m} h\, .
\end{equation}
Consequently for $C$ large enough $\bbP[|\tilde\psi_{m_f}-\psi_0|>Ch]\rightarrow_{N\rightarrow\infty} 0$, which completes the proof
of \eqref{eq:approch M}. The bound \eqref{eq:approch M2} is much rougher and it follows directly from
the argument we have used for establishing \eqref{eq:approch M}. This completes the proof of Proposition~\ref{th:approachM}
\end{proof}

\medskip

\noindent
{\it Proof of Proposition~\ref{th:PDE}}
The crucial issues are the gradient flow structure of \eqref{eq:K} and its dissipativity  properties.
The gradient structure of \eqref{eq:K} \cite{cf:BGP}
implies that the functional 
\begin{equation}
\cF (p):= \frac 12 \int_\bbS p(\theta) \log p(\theta) \dd \theta
-\frac K2 \int_ \bbS\int_\bbS p(\theta) \cos( \theta- \theta ') p(\theta ') \dd \theta \dd \theta'  \, ,
\end{equation}
is non increasing along the time evolution.
The dissipativity properties proven in \cite[Theorem~2.1]{cf:GPP} show that for every 
$k \in \bbN$ and $a>0$ we can find $\tilde t$ such that $\Vert p_t \Vert_{C^k}< a$ for every $t \ge \tilde t$.
Therefore for any $k$ there exists $\{t_n\}_{n=1,2, \ldots}$ such that $t_{n+1}-t_n>1$ and $\lim_n p_{t_n}$
exists in $C^k$ and we call it $p_\infty$. An immediate consequence is that $\lim_n \cF (p_{t_n}) 
=\cF (p_{\infty})$. But we can go beyond by introducing the semigroup $S_t$ associated to
\eqref{eq:K}, by setting $S_{t'}p_t=p_{t+t'}$. \cite[Theorem~2.2]{cf:GPP} implies the continuity 
of this semigroup in $C^k$, so that, since for $t \in [0,1]$ we have $t_n \le t_n + t < t_{n+1}$,
we obtain $\cF (S_t p_\infty)=  \cF (p_\infty)$.  Therefore $\partial_t\cF (S_t p_\infty)=0$,
but the condition $\partial_t\cF (p_t)=0$, for a solution of \eqref{eq:K}, directly implies that
$\partial^2 _\theta p_t= 2 \partial_\theta(p_t J*p_t)$, which is the stationarity condition for \eqref{eq:K}. Therefore $p_t$
is either $q_\psi$, for some $\psi$, or it coincides with $\frac 1{2\pi}$ (see \eqref{eq:q}-\eqref{eq:M}).

Let us point out that if $p_{t_n}$ converges to $\frac 1{2\pi}$ then $\{p_t\}_{t>0}$ itself converges to
$\frac 1{2\pi}$. This is just because $\cF\left(\frac 1{2\pi}\right) > \cF( q_\psi)$, so that 
if $\lim_n p_{t'_n}= q_\psi$ and $\lim_n p_{t_n}=\frac 1{2\pi}$
then it suffices to choose $n$ such that $\cF(p_{t'_n})< \cF\left(\frac 1{2\pi}\right) $
and $m$ such that $t_m>t'_n$ to get $\cF(p_{t'_n})\ge \cF(p_{t_m})\ge \cF\left(\frac 1{2\pi}\right)$, which is impossible.

So we have seen that either $\lim_{t \to \infty} p_t
= \frac 1{2\pi}$ or all limit points are in $M$. The stronger result we need 
is the convergence also when the limit point is not $\frac 1{2\pi}$. This result is provided 
by the nonlinear stability result \cite[Therem~4.6]{cf:GPP} which says  that if
$p_0$ is in a neighborhood of $M$ (the result is proven for a $\bbL ^2$ neighborhood,
which is much more than what we need here), then there exists $\psi$ such that
$\lim_{t \to \infty} p_t= q_\psi$ in $C^k$. 

To complete the proof we need to characterize the portion of $\cM_1$ which is attracted 
by $\frac 1{2\pi}$, that is we need to identify the stable manifold of the unstable point with the set $U$
in \eqref{eq:Umanif}.
But this is the content of \cite[Proposition~4.4]{cf:GPP}.
\qed

\section{Proof of Theorem~\ref{th:main}}
\label{sec:proofmain}

The proof of Theorem~\ref{th:main} relies 
on the results  of the previous sections
and on  a convergence argument of the process in the tangent space that we give here. 

\medskip

\noindent
{Proof of Theorem~\ref{th:main}.}
First of all Proposition~\ref{th:approachM} takes care of the evolution up to time $N \tilde \gep_N= C\log N$
and provides an estimate on the closeness of the empirical measure to the manifold $M$ that allows
to apply directly  
Proposition~\ref{prop:closeness traj manifold} and then Proposition~\ref{th:dyntan}.
Note that the iterative scheme that we have set up in
Section~\ref{sec:sheme} has been presented without asking $\psi_0$ not to be random 
or not to depend on $N$. In fact we start the iterative scheme at time $N \tilde\gep_N$ and from the random phase $\Psi_N$
of Proposition~\ref{th:approachM}
that converges in probability to the (non random) value $\psi_0$. 
Of course there is here an abuse of notation in the use of $\psi_0$, but notice actually that, by the rotation invariance of the system,
we can actually consider without loss of generality that the empirical measure $\mu_{N, C \log N}$ has precisely 
the phase $\psi_0$. Moreover we make a time shift of $N\tilde\gep_N$, so that the phase is $\psi_0$ at time $T_0=0$. The result in Theorem~\ref{th:main} is given for times
starting from $N\gep_N$ and not $N\tilde\gep_N$, but as stated in Proposition~\ref{th:approachM}, the empirical measure stays close to $q_{\psi_0}$ in the time interval
$[N\gep_N,N\tilde\gep_N]$. Therefore we have the finite sequence of times $T_0, T_1, \ldots, T_n$, with the corresponding phases $\psi_0, \psi_1, \cdot,
\ldots, \psi_n$ and we define $\psi_t$ for every $t\in [0, T_n]$ by linear interpolation. We assume $T_n> \tau_f N$.
 
 We then note that, in view of
\eqref{ineq:bound sup nuk}, the control on the phases, see Proposition~\ref{th:dyntan},   on the times $T_1, T_2, \ldots$ of our iteration scheme
suffices not only to control the distance between the empirical measure $\mu_{N, t}$ and  $q_{\psi_{t}}$, in the $H_{-1}$ norm,
for $t= T_k$, but for every $t\in [0, T_n]$. We are now ready to identify the process $W_{N, \cdot}$ of Theorem~\ref{th:main}:
\begin{equation}
\label{eq:WNmain}
W_{N, \tau}\, :=\, \frac{\psi_{\tau N} -\psi_0}{D_K}\, ,
\end{equation}
where we recall that $\tau\in [0, T_n/N]$. We are therefore left with showing that
$W_{N, \cdot}$ converges to standard Brownian motion. Note that it would be equally possible and maybe
more natural to define $W_{N, \tau}$, for $\tau \ge  \gep_N$ as in the right-hand side of \eqref{eq:WNmain},
but with $\tau$ replaced by $\tau- \gep_N$, and $W_{N, \tau}=0$ for $\tau \in [0, \gep_N]$. In view of the statement 
we want to prove this detail is irrelevant.

In proving the convergence to Brownian motion 
we apply Proposition~\ref{th:dyntan} and replace the process $\psi_\cdot$ with the
cadlag process  $\psi_0+ M_{N , \cdot} \in D([0, T_n /N]; \bbR)$ defined by 
\begin{equation}
M_{N, \tau} := \sum_{k \in \bbN:\, T_k \le  N \tau}\gD M _{N, k}\, ,
\end{equation}
and 
\begin{equation}
\gD M _{N, k} \, :=\, 
 \frac{( Z^{k}_{T_k}, q'_{\psi_{k-1}})_{-1,1/q_{\psi_{k-1}}}}{( q',q')_{-1,1/q}}\, .
\end{equation}
It is straightforward to see that $M_{N ,\cdot}$ is a martingale with respect to the filtration 
$\tilde \cF _\tau:= \cF_{\lfloor \tau T\rfloor / T}$, where $\cF_\cdot$ is the natural filtration
of $\{W^j_{N \cdot}\}_{j=1, \ldots, N}$: the martingale is actually in $L^p$, for every $p$, as  
 the moment estimates is Section~\ref{sec:apriori} show. 
 We can now apply  the Martingale Invariance Principle 
in the form given by \cite[Corollary~3.24, Ch.~VIII]{cf:JS} to $M_{N , \cdot}$ for continuous time
martingales: the hypotheses to verify in the case of piecewise constant cadlag martingales boil
down to  the variance convergence condition that for every $\tau \in [0, \tau_f]$
\begin{equation}
\label{eq:JSc1}
\lim_{N \to \infty} \sum_{k \in \bbN:\, T_k \le \tau N} \bbE \left[ \left(\gD M_{N, T_k}\right)^2 \Big \vert\,  \cF _{T_{k-1}}
\right]\, =\, \tau D_K^2\, ,
\end{equation}
in probability, 
and the Lindeberg condition that for every $\gep>0$ in probability we have
\begin{equation}
\label{eq:JSc2}
\lim_{N \to \infty} \sum_{k \in \bbN:\, T_k \le \tau N} \bbE \left[\left(\gD M_{N, T_k}\right)^2; \,  \gD M_{N, T_k}^2> \gep  \, \Big \vert\,  \cF _{T_{k-1}}
\right]\, =\, 0\,  .
\end{equation}
For what concerns \eqref{eq:JSc1} we have
\begin{equation}
\bbE \left[ \left(\gD M_{N, T_k}\right)^2 \Big \vert\,  \cF _{T_{k-1}}
\right]\, =\, \frac 1{N\Vert q' \Vert _{-1,1/q}^2} \int_{T_{k-1}}^{T_k}  \int_\bbS\left(f'_{\psi_{k-1},0}(\theta)\right)^2 
\mu_{N, s}(\dd \theta) \dd s\, .
\end{equation}
Now take the sum over $k$ and 
use the uniform estimate \eqref{ineq:bound sup nuk}
 of Proposition~\ref{prop:closeness traj manifold} to replace the empirical measure with $q_{\psi_{T_{k-1}}}(\theta)\dd \theta$.
 Since a direct computation shows that $ \int_\bbS (f'(\theta)_{\psi, 0})^2 q_\psi(\theta) \dd \theta =1$, \eqref{eq:JSc1} follows.

For what concerns  \eqref{eq:JSc2} we remark that, by the Markov inequality, it suffices to show that
\begin{equation}
\label{eq:JSc2.1}
\lim_{N \to \infty}  \sum_{k \in \bbN:\, T_k \le \tau N} \bbE \left[\left(\gD M_{N, T_k}\right)^4  \, \Big \vert\,  \cF _{T_{k-1}}
\right]\, =\, 0\,  .
\end{equation}
Actually one can show that there exists a non random constant $C$ such that almost surely
\begin{equation}
\label{eq:JSc2.2}
\bbE \left[\left(\gD M_{N, T_k}\right)^4  \, \Big \vert\,  \cF _{T_{k-1}}
\right]\, \le \, C \left( \frac T N \right)^2\, .
\end{equation}
 This 
 is an immediate consequence of \eqref{eq:estcp1}, but of course, since we are projecting
on $q'$ and since we are just considering the fourth moment, a similar  estimate can be easily obtained
explicitly by proceeding like for \eqref{eq:JSc1} and by using the fact that $\Vert f'_{\psi, 0}\Vert_\infty=
\Vert f'_{0}\Vert_\infty< \infty$. 
Of course \eqref{eq:JSc2.1} follows from \eqref{eq:JSc2.2}.

Therefore $M_{N, \cdot}\in D([0, \tau_f]; \bbR)$ converges in law to $W_\cdot / \Vert q' \Vert_{-1,1/q}$, where $W_\cdot$
is a standard Brownian motion. 
This is almost the result we want (recall that $D_K=1/ \Vert q' \Vert_{-1,1/q}$), since we $M_{N, \cdot}/D_K$
differs from $W_{N, \cdot}$ just for the fact that they interpolate in a different way between the times $T_k$
(where the coincide) and that in the case of $W_{N, \cdot}$ the convergence is in $C^0([0, \tau_f]; \bbR)$.
But \eqref{eq:JSc2.2} guarantees that the sum of the fourth power of the jumps of $M_{N, \cdot}$ adds up to 
$O(T^2/N)=o(1)$ in probability, so the supremum of the jumps is $o(1)$,  and therefore the convergence for  $M_{N, \cdot}\in D([0, \tau_f]; \bbR)$ 
implies the convergence of $W_{N, \cdot}\in C^0([0, \tau_f]; \bbR)$. The proof of
Theorem~\ref{th:main} is therefore complete. 
\qed

\appendix

\section{The evolution  in $H_{-1}$}
\label{sec:A_H}
In what follows we fix $q$ in the invariant manifold $M$ (see \eqref{eq:M}). Unlike the rest of the paper
here we do not identify $q$ with $q_\psi$ and then with $\psi$, so in particular we write 
$L_q$ (and not $L_\psi$), $\cG^q_t(\cdot)$   (and not $\cG^\psi_t(\cdot)$ like in \eqref{eq:Z1}), and so on.
We work with the signed measure
\begin{equation}
\nu_{N,t}(\dd \theta)\, :=\,   \mu_{N, t}(\dd \theta) - q(\theta) \dd \theta \, ,
\end{equation}
which can be seen as an element of $H_{-1}$. This is simply because it is the difference of
two probability measures. 
In fact, if $\mu\in \cM_1$, $\theta \mapsto \mu ([0, \theta])$ is a primitive  of $\mu$ and, by Remark~\ref{rem:norecenter}, $\Vert \mu - \nu \Vert_{-1}^2 \le \int_\bbS (\mu ([0, \theta])-\nu ([0, \theta]))^2 \dd \theta \le 2 \pi$. Therefore  $\Vert \mu- \nu\Vert_{-1}\le \sqrt{2 \pi}$: of course this quick argument needs to be cleaned up by first {\sl smoothing} the measures. 
That is, we introduce an approximate identity $\phi_n \in C^\infty$ ($\phi_n \ge 0$, $\phi_n (\theta)= 0$ for $\theta \in [1/n , 2\pi -1/n]$, $\int_\bbS \phi_n  =1$ and $\lim_n \int _\bbS F\phi_n =F(0)$ for every $F\in C^0$). We then introduce the probability density  $\theta \mapsto \mu_n( \theta):= \int_\bbS \phi_n (\theta- \theta') \mu( \dd \theta')$ and verify that
 \begin{equation} 
 \Vert \mu_n -\mu_m\Vert_{-1}^2 \, \le\, \frac 4 {\min(n,m)}\, ,  \end{equation} 
so that $\lim_n\mu_n$ exists in $H_{-1}$ (of course the limit exists also weakly and it is $\mu$). 

\medskip

We aim at proving:
\begin{proposition}
\label{th:H-1evol}
If $\{\gp_t^{j,N}\}_{t \ge 0, \, j=1, \ldots, N}$ solves \eqref{eq:evol} then $\nu_{N, \cdot}\in C^0([0, \infty); H_{-1})$ and we have
\begin{equation}
\label{eq:mild}
\nu_{N,t} \, =\, \exp(tL_q) \nu_{N,0}- \int_0^t \exp( (t-s) L_q)\partial \left( (J* \nu_{N,s}) \nu_{N,s} \right) \dd s +  
Z_{N,t}\, ,
\end{equation}
where $Z_{N,t}$ is the limit in $H_{-1}$ as $\tau \nearrow t$ of $Z_{N, t, \tau}$, where
\begin{equation}
\label{eq:GtauA}
Z_{N,t, \tau }(\theta)\, :=\, 
\frac 1N \sum_{j=1}^N \int_0^\tau
\partial_{\theta'}\cG_{t-s}^q (\theta,\gp^{j,N}_s)
\dd W^j_s\, 
\end{equation}
Moreover all the terms appearing in the right-hand side of \eqref{eq:mild}, as functions of time, are in $C^0([0, \infty); H_{-1})$.
\end{proposition}
\medskip

\noindent
{\it Proof.}
For $(t, \theta) \mapsto F_t(\theta)$ in $C^{1,2} (\bbR^+\times\bbS; \bbR)$, from
\eqref{eq:evol} we directly obtain 
\begin{multline}
\label{eq:weakform}
\int_\bbS F_t( \theta) \nu_{N,t}(\dd \theta)\, =\, 
\int_\bbS F_0(\theta) \nu_{N,0}(\dd \theta)
+ \int_0^t \int_\bbS \left( L_q^* F_s\right) (\theta) \nu_{N,s}(\dd \theta) \dd s \\
+ \int_0^t \int_\bbS \partial_s F_s ( \theta)  \nu_{N,s}(\dd \theta) \dd s
+  \int_0^t \int_\bbS \partial_\theta F_s ( \theta)  
(J*\nu_{N,s}) (\theta)
\nu_{N,s}(\dd \theta) \dd s + Z^F_{N,t}\, ,
\end{multline}
where
\begin{equation}
\label{eq:weakformZ}
Z^F_{N,t}\, =\, 
\frac 1N \int_0^t \sum_{j=1}^N \partial_{\theta} F_s(\theta)\Big\vert_{\theta= \gp_s^{j,N}}
\dd W^j_s\, ,
\end{equation}
and $L_q^*$ is the adjoint in $\bbL^2_0$ of $L_q$, that is
\begin{equation}
L_q^* v\, =\, \frac 12 v'' + (J*q) v'- J* (qv')- \int_\bbS \, (J*q) v'\, ,
\end{equation}
 for $v \in C^2(\bbS; \bbR)$
such that $\int_ \bbS v =0$

We  sum up here some useful  properties of $L_q^*$:
\medskip
\begin{enumerate}
\item In \cite{cf:BGP} it is shown that the $\bbL_0^2$-norm is equivalent to the Dirichlet form norm of $L_q$:
the squared Dirichlet form norm of $u$ is
$\Vert u\Vert ^2_{-1,1/q} + \left( u, (-L_q) u \right)_{-1,1/q}$. On the other hand it is straightforward to see
that the properties of $L_q$ in $H_{-1,1/q}$, notably the fact 
 that it is self-adjoint and that it has compact resolvent, still hold true in the space of the Dirichlet form. 
 So  $L_q$ has compact resolvent in $\bbL^2_0$, which directly implies that $L_q^*$ has compact resolvent
and the very same spectrum (see e.g. \cite[VI.5]{cf:RS1}).
\item Recall that we denote by $\{e_j\}_{j=0,1, \ldots}$ a complete set of eigenvectors of $L_q$ which is orthonormal in $H_{-1,1/q}$ and observe that there is a unique solution $f_j$ to
\begin{equation}
\label{eq:link e and f}
{\mathtt A}_q f_j (\theta)\, :=\, 
- \partial_\theta \left( q(\theta) \partial_\theta f_j (\theta) \right) \, =\, e_j (\theta)\, ,
\end{equation}
such that $\int_\bbS f_j =0$. 
More generally,
${\mathtt A}_q$ is a bijection from  $\{u\in C^\infty: \, \int_\bbS u=0\}$ to itself: in fact, $v ={\mathtt A}_q u$ is equivalent to $u'= -\cV/q$ in our standard notations, which determines $u$ since $\int_\bbS u=0$. 
In particular
$f'_j= - \cE_j/ q$ and  $f_j \in C^\infty$, since $e_j$ is $C^\infty$, and one obtains
\begin{equation}
\label{eq:A.ort}
\left( f_i , e_j\right)_2\, =\int_\bbS f_i e_j \, =\, -\int_\bbS f_i' \cE_j \, =\, \int_\bbS \frac{\cE_i \cE_j}{q}\, =\, \gd_{i,j}\, .
\end{equation}
 By using the fact that $q(\cdot)$ is even, one verifies directly also  that if $e_j$ is even (respectively, odd)  -- recall from Section~\ref{sec:linear}
that $e_j$ is either even or odd -- the $f_j$ is even  (respectively, odd) too.
\item By observing also that   $L_q{\mathtt A}_q  = {\mathtt A}_q L_q^*$ one verifies that
$\{ f_j \}_{j=0,1, \ldots}$ is a complete set of eigenfunctions for $L_q^*$ and, of course,
$L_q^* f_j = -\gl_j f_j$.
\end{enumerate}

\medskip

Therefore for every $t>0$ and $s\le t$ we can define $F_s(\theta)=(\exp((t-s) L_q^* ) F) (\theta)$ for $F\in \bbL^2_0$
and standard parabolic regularity \cite{cf:Friedman} results imply that $F_s(\cdot)$ is $C^\infty$ for
$s<t$ (in our case this can be proven directly by using the Fourier transform, like in \cite{cf:GPP}, but for what follows we choose $F\in C^2$  and the regularity result
is even more straightforward). By plugging this choice into \eqref{eq:weakform}
we obtain
\begin{multline}
\label{eq:fromweakform}
\int_\bbS F( \theta) \nu_{N,t}(\dd \theta)\, =\, 
\int_\bbS (\exp(t L_q^* ) F) (\theta) \nu_{N,0}(\dd \theta)
\\  
+  \int_0^t \int_\bbS \partial_\theta (\exp((t-s) L_q^* ) F) (\theta)  
(J*\nu_{N,s}) (\theta)
\nu_{N,s}(\dd \theta) \dd s + Z^F_{N,t}\, .
\end{multline}

\medskip
 
At this point we step to  looking at  $\nu_{N,t}$  as an element
of $H_{-1}$ and we reconsider \eqref{eq:fromweakform} with this novel viewpoint.

First of all $\int_\bbS F( \theta) \nu_{N,t}(\dd \theta)= \langle F, \nu_{N,t}\rangle _{1,-1}$,
where $\langle \, \cdot\, ,  \, \cdot\,\rangle _{1,-1}$ is the duality between $H_1$
and $H_{-1}$ (cf. Sec.~\ref{sec:linear}). For the first term in the right-hand side we 
observe that, for $v\in H_{-1}$ we have $\langle \exp(t L_q^* ) F, v\rangle_{1,-1}=
\langle F, \exp(t L_q) v\rangle_{1,-1}$: this is because this relation holds 
when $v \in \bbL^2_0$ (in this case the duality can be replaced by the $\bbL^2$ scalar product)
and because one can choose a sequence  $\{v_n\}_{n=1,2, \ldots}$, $v_n\in  \bbL^2_0$ 
such that $v_n \to v$ 
 in $H_{-1}$ (one can choose $v_n=\phi_n *v$) 
 so that
 \begin{equation}
 \label{eq:stepA.1}
 \langle \exp(t L_q^* ) F, v\rangle_{1,-1}\, =\, \lim_n ( F,  \exp(t L_q)v_n)_2
 \, =\, \langle F,  \exp(t L_q)v\rangle_{1,-1}\, ,
 \end{equation}
 where we have used the continuity properties of the duality and of the semigroup operator.
 
 For the second term in the right-hand side of \eqref{eq:fromweakform} 
we write
 \begin{multline}
  \int_0^t \int_\bbS \partial_\theta (\exp((t-s) L_q^* ) F) (\theta)  
(J*\nu_{N,s}) (\theta)
\nu_{N,s}(\dd \theta) \dd s\, =\\
  \int_0^t 
\langle (J*\nu_{N,s})\partial \exp((t-s)L^*_q)F, \nu_{N,s} \rangle_{1,-1} \dd s\, ,
 \end{multline} 
We now introduce $v_{n,s} := \phi_n*\nu_{N,s}$
so that for every $s\in [0,t)$
\begin{equation}
\label{eq:A2.1}
\begin{split}
\langle (J*\nu_{N,s})\partial \exp((t-s) L^*_q)F, \nu_{N,s} \rangle_{1,-1}
\, &=\,- \lim_n \left( F, \exp((t-s) L_q) \partial((J*\nu_{N,s}) v_{n,s})\right)_2\\ 
&=\, -
 \langle F, \exp((t-s) L_q) \partial ((J*\nu_{N,s})\nu_{N,s})\rangle_{1,-1} ,
\end{split}
\end{equation}
where in the last step we have used the fact that $\exp((t-s) L_q)$ is a continuous 
operator from
$H_{-2}$ to $H_{-1}$ (Lemma~\ref{th:add1}). Notice moreover that we have
\begin{equation}
 \vert \left( F, \exp((t-s) L_q) \partial((J*\nu_{N,s}) v_{n,s})\right)_2\vert 
\le \, \Vert \phi_n \Vert_1
\Vert \partial ((J*\nu_{N,s})\partial \exp((t-s)L_q^*)F)\Vert_2 \Vert \nu_{N,s}\Vert_{-1}\, ,
\end{equation}
and, since   $J(\cdot)=-K \sin(\cdot)$,
 one sees that this expression is bounded by a constant times $\Vert F''\Vert_2$, uniformly in 
 $n$ and $s \le t$. Such a bound tells us that one can exchange limit and integration in 
\begin{equation}
\int_0^t \lim_n
  \Big( F, \exp((t-s) L_q) \partial((J*\nu_{N,s}) v_{n,s})\Big)_2 \dd s 
 \, ,
\end{equation}
and then, for fixed $n$ one can of course exchange integral in $\dd s$
and integral in $\dd \theta$. At this point we appeal again to Lemma~\ref{th:add1} 
that guarantees that  $\int_0^t \exp((t-s) L_q)\partial ((J*\nu_{N,s}) v_{n,s}) \dd s $ converges, in
$H_{-1}$, to  $\int_0^t \exp((t-s) L_q)\partial ((J*\nu_{N,s}) \nu_{N,s}) \dd s $:
note in fact that $\Vert  \partial((J*\nu_{N,s}) v)\Vert _{-2} \le c_J \Vert v \Vert_{-1}$  
 so that (by Lemma~\ref{th:add1} )
\begin{multline}
\left \Vert \int_0^t
\exp((t-s) L_q) \partial((J*\nu_{N,s}) (v_{n,s}-v_{n',s}) \dd s \right \Vert_{-1}\, \le \\
c_J C \int_0^t  \left(1+\frac{1}{\sqrt{t-s}}\right)\Vert v_{n,s}- v_{n',s}\Vert_{-1} \dd s\, ,
\end{multline}
and the right-hand side vanishes for $\min(n,n') \to \infty$.
Therefore we obtain
\begin{multline}
\int_0^t \int_\bbS \partial_\theta (\exp((t-s) L_q^* ) F) (\theta)  
(J*\nu_{N,s}) (\theta)
\nu_{N,s}(\dd \theta) \dd s\, =\\
  \langle F, \int_0^t \exp((t-s) L_q) \partial((J*\nu_{N,s}) \nu_{N,s}) \dd s\rangle_{1,-1} \, .
\end{multline}

We are left with the last term in \eqref{eq:fromweakform}. 
It is now useful to use the kernel of the $L_q$-semigroup in $\bbL^2_0$ 
\begin{equation}
\label{eq:Gkernel}
\cG^q_s(\theta, \theta')\, :=\,  \sum_{l=0}^\infty \exp(-s\gl_l) e_l(\theta) f_l (\theta')\, ,
\end{equation}
so that
\begin{equation}
\label{eq:Luv}
\left( u, \exp(s L_q )v \right)_2\, =\, \left( \exp(s L_q^*)  u,  v \right)_2\, =\,
\int_{\bbS} \int_{\bbS}u(\theta) \cG^q_s(\theta, \theta') v (\theta')\dd \theta \dd \theta' \, .
\end{equation}
Note also that, for $s>0$, $\cG_s^q$ is $C^\infty$ in both variables, by the standard parabolic regularity results
we have mentioned above.
So, for every $\tau<t$, $\theta \mapsto Z_{N,t, \tau }(\theta)$ (recall \eqref{eq:GtauA})
is well defined and smooth in $\theta$. But  Lemma~\ref{th:add2}  
tells us that $\lim_{\tau \nearrow t} Z_{N,t, \tau })$ exists in $H_{-1}$.
If we call the limit $Z_{N,t })$ we directly see that (recall \eqref{eq:weakformZ})
\begin{equation}
Z^F_{N,t}\, =\, \langle F , Z_{N,t }\rangle _{1,-1}\, .
\end{equation}
 Therefore we have shown that
  \eqref{eq:fromweakform} implies the validity 
  of \eqref{eq:mild}  if we take the duality with respect to an arbitrary $F\in C^2$.
  But we have also shown that every term in \eqref{eq:mild} is in $H_{-1}$, therefore
 the equation extends to $F \in H_1$ and
 \eqref{eq:mild} is proven. 
 
 The continuity claimed in the statement follows by the continuity of the three terms
 in the right-hand side of \eqref{eq:mild}. The continuity of the first term is immediate from
 the properties of the semigroup. The continuity of the second term follows from a direct estimate by applying 
 both bounds in Lemma~\ref{th:add1}. Finally the continuity of the third term is claimed in Lemma~\ref{th:add2}.
 The proof of Proposition~\ref{th:H-1evol} is therefore complete.  
 \qed
 
\medskip

\begin{lemma}
\label{th:add1}
For $\tau>0$ the operator $\exp(\tau L_q)$ extends to a bounded operator
from $H_{-2}$ to $H_{-1}$ and 
there exists $C>0$ such that for every $\tau>0$
\begin{equation}
\label{eq: bound semi group H-2}
\left\Vert \exp(\tau L_q) u \right \Vert_{-1} \, \le \, C\left(1+\frac{1}{\sqrt{\tau}}\right) \Vert u \Vert_{-2}\, ,
\end{equation}
and such that for every $\gep\in (0,1/2)$ we have
\begin{equation}
\left\Vert \exp((\tau+\gd) L_q) u - \exp(\tau L_q) u \right \Vert_{-1} \, \le \, C \gd^\gep \left( 1+ \frac1{\tau^{\gep +1/2}}\right)
 \Vert u \Vert_{-2}\, ,
\end{equation}
for every $\tau>0$ and $\gd\ge0$.
\end{lemma}

\medskip

\begin{proof}
We introduce the interpolation spaces associated to $L_q$ that is the (Hilbert) spaces
\begin{equation}
 V^m\, :=\, \left\{u=\sum_{k=0}^\infty u_k e_k, \quad \sum_{k=0}^\infty (1+\gl_k)^m u_k^2<\infty\right\}\, , 
\end{equation}
associated with the norms
\begin{equation}
\label{eq:def norm V}
 \Vert u\Vert^2_{V^m}\, :=\, \Vert (1-L_q)^{m/2} u\Vert^2_{-1,1/q}\, = \, \sum_{k=0}^\infty (1+\gl_k)^m u_k^2\, .
\end{equation}
It is proven in \cite[Remark A.1]{cf:GPPP}  that the norms $\Vert .\Vert_{V^n}$ and $\Vert .\Vert_{n-1}$ are equivalent. This equivalence can also be deduced  from Remark \ref{rem:eigenvalues and eigenfunction expansion}. In particular 
$\Vert .\Vert_{V^{-1}}$ and $\Vert .\Vert_{-2}$ are equivalent, so we will prove \eqref{th:add1} with $\Vert .\Vert_{-1,1/q}$ and $\Vert .\Vert_{V^{-1}}$.
For all $u=\sum_{k=0}^\infty u_k e_k$, we extend $e^{\tau L_q}u$ as
\begin{equation}
 e^{\tau L_q}u\, =\, \sum_{k=0}^\infty e^{-\gl_k \tau} u_k e_k\, ,
\end{equation}
and  we deduce
\begin{equation}
\label{eq:formule semi groupe H-2}
\Vert  e^{\tau L_q}u \Vert^2_{-1,1/q}\, =\, \sum_{k=0}^\infty (1+\gl_k)e^{-2\gl_k \tau} \frac{u_k^2}{1+\gl_k}\, .
\end{equation}
But if we define $f(y):=(1+y)e^{-2 y \tau}$, it is easy to see that for all $y\geq 0$, there exist $C$ such that $f(y)\leq C^{2}\left(1+\frac{1}{\sqrt{\tau}}\right)^{2}$, which with \eqref{eq:def norm V}.
and \eqref{eq:formule semi groupe H-2} gives the first inequality.

For the second inequality we make a similar spectral decomposition and we obtain
\begin{equation}
\label{eq:spdecp2}
\Vert  e^{(\tau+\gd) L_q}u-  e^{\tau L_q}u \Vert^2_{-1,1/q}\, 
=\, \sum_{k=0}^\infty (1+\gl_k)e^{-2\gl_k \tau}
\left(1- \exp( -\gd \gl_k)\right)^2
 \frac{u_k^2}{1+\gl_k}\, .
\end{equation}
We then use $(1-\exp(-x)) \le x^\gep$ for $x\ge 0$ and $(1+x) x^{2\gep}\exp(-x\tau)
\le C^2 (1+\tau^{-\gep-1/2})^2$, for a suitable $C$ which can be chosen independent
of $\gep \in (0/1/2)$. 
\end{proof}

\medskip

\begin{lemma}
\label{lem:H-1 H-2}
For all $u,v\in H_{-1}$, there exists $C>0$ such that
\begin{equation}
 \Vert \partial_\theta(uJ*v)\Vert_{H_{-2}}\, \leq \, C \Vert u\Vert_{H_{-1}}\Vert v\Vert_{H_{-1}}\, .
\end{equation}

\end{lemma}

\medskip

\begin{proof} For this proof it is practical to write the $H_{-1}$-norms by using the Fourier
coefficients. In fact if $u\in H_{-s}$, here $s=1$ or $s=2$, we can define 
$u_n= \langle u, b_n\rangle$, where $b_n(\theta)= \exp(i n  \theta)/2\pi$ (note that $u_0=0$)
and  $\theta \mapsto \sum_{n \in \bbZ: \vert n \vert \le N} u_n \exp(in \theta)$ converges as
$N \to \infty$ in $H_s$ to $u$.
Moreover we have
\begin{equation}
 \Vert u\Vert_{-s}\, :=\, \left(\frac{1}{2\pi} \sum_{m\in\bbZ} \frac{u_m^{2}}{m^{2s}}\right)^{1/2}\, .
\end{equation}
Since $J(\theta)=-K\sin(\theta)$, a direct calculation gives
\begin{equation}
 \partial_\theta(uJ*v)\, =\, K\pi \left[(m-1)v_{-1}\sum_{m\in\bbZ} e^{i(m-1)\theta}u_m -(m+1)v_1\sum_{m\in\bbZ} e^{i(m+1)\theta}u_m\right]\, ,
\end{equation}
from which we extract
\begin{multline}
 \Vert \partial_\theta(uJ*v)\Vert^2_{-2}\, =\, \frac{K^2\pi}{2} \sum_{m\in\bbZ,m\neq 0} m^{-4}\left| m(v_{-1}u_{m+1}-v_1 u_{m-1})\right|^2 \\
                                           \leq\, K^2\pi \max(|v_{-1}|^2,|v_1|^2)\sum_{m\in\bbZ,m\neq 0} m^{-2}(u_{m-1}^2+u_{m+1}^2) \\
                                           \leq\, 4K^2\pi \Vert v\Vert_{-1}^2 \Vert u\Vert_{-1}^2\, .
\end{multline}

\end{proof}

\medskip

\begin{lemma}
\label{th:add2}
The almost sure  limit of
$ Z_{N, t, \tau}$ as $\tau \nearrow t$ exists in $H_{-1}$ and, if we call the limit point
$Z_{N,t}$, we can choose a continuous version of $Z_{N, \cdot}$, that is $Z_{N, \cdot}\in C^0 ([0, \infty); H_{-1})$. 
\end{lemma}

\medskip

\noindent
{\it Proof.} 
The claim follows from the same estimates as the one that we have obtained for the proof of Lemma~\ref{lem:bound Zk},
which are however substantially more precise than what we need here: recall that now $N$ is fixed, while 
in Section~\ref{sec:apriori}  one of the crucial points is to follow the $N$ dependence of the results.
Therefore we will not go through the arguments in detail, but we just point out that
that one goes from $Z_{N, t, \tau}$  to $Z^k_{t, t'}$, see  \eqref{eq:M.1} by making obvious changes. 
So, in particular, proceeding like 
for 
\eqref{bound:first term noise non perp}
we easily gets 
\begin{equation}
\label{eq:forKo}
\bbE\left[ \left \Vert Z_{N, t,\tau}- Z_{N, t,\tau'} \right\Vert_{-1}^{2m} \right] \, \le \, C h_1^m(\vert \tau-\tau'\vert)\, ,
\end{equation}
where $C$ depends on $N$ and $m$ and $ 0 \le \tau, \tau' < t$.
An estimate like \eqref{eq:forKo} implies almost sure H\"older continuity of $Z_{N, t,\cdot}$, by a direct application 
of Kolmogorov continuity Lemma \cite{cf:SV} or by using the 
Garsia-Rodemich-Rumsey Lemma (Lemma~\ref{th:GRR}). That is, there exists 
a (positive) random variable $X$ and a positive constant $c>0$ such that
\begin{equation}
\left \Vert Z_{N, t,\tau}- Z_{N, t,\tau'} \right\Vert_{-1}\, \le \, X \, \vert \tau -\tau'\vert ^c\, ,
\end{equation}
for every $ 0 \le \tau, \tau' < t$. Therefore the almost sure limit of $Z_{N, t,\tau}$, as $\tau\nearrow t$, exists.

The continuity of the limit follows in the same way, this time using also 
\eqref{bound:second term noise}. Actually,  in the proof of Lemma~\ref{lem:bound Zk}
we use Lemma~\ref{th:add2} only to define $Z^k_t$ as  almost sure limit in $H_{-1}$:
the proof of continuity is strictly contained in the argument that starts from \eqref{eq:estcp1}
and goes till the end of that proof (but, once again, that proof is substantially more informative
and involved, since it follows the $N$-dependence).
\qed

\subsection{Second order estimates of the projection}
As anticipated in \S~\ref{sec:heur} our approach requires a control up to and including the second
order for the projection map $p(\cdot)$
(recall \S~\ref{sec:Manif} for the definition). The expansion is with respect to the $H_{-1}$ distance from
the manifold $M$.
\medskip

\begin{lemma}
\label{lem:second order projection}
 For all $q=q_\psi \in M$ and $h\in H_{-1}$ with $\Vert h\Vert_{-1}< \gs$, we have
\begin{equation}
 p(q+h)\, =\, \psi -\frac{( h, q')_{-1, 1/q}}{( q', q')_{-1,1/q}}
 \left(1-\frac{1}{2\pi I_0^2(2Kr)}
 \frac{( h, (\log q)'')_{-1,1/q}}{( q',q')_{-1,1/q}}\right)
 +O(\Vert h\Vert_{-1}^3)\, .
\end{equation}
\end{lemma}
\medskip

\begin{proof}
For $h$ as in the statement we have that
\begin{equation}
\label{eq:pA.1-1}
 ( q_\psi+h-q_{\psi+\gep},q'_{\psi+\gep})_{-1,1/q_{\psi+\gep}}\, =\, 0\, ,
\end{equation}
for $\gep:= p(q_\psi+h)-\psi$. Since $p(\cdot)$ is smooth, we have $\gep=O( \Vert h\Vert_{-1})$.
By expanding 
$q_{\psi+\gep}$ with respect to $\gep$ 
we see that \eqref{eq:pA.1-1} implies
\begin{equation}
\label{eq:pA.1-2}
 \left( h+\gep q'_\psi -\frac{ \gep ^2}{2}q''_\psi,q'_{\psi+\gep }\right)_{-1,1/q_{\psi+\gep}}\, =\, O(\gep ^3)\, .
\end{equation}
Let us rewrite \eqref{eq:pA.1-2} more explicitly (recall Remark~\ref{rem:compute}) as
\begin{equation}
 \int_\bbS \left(\cH (\theta) +\gep  q_\psi(\theta)-\frac{\gep ^2}{2} q'_\psi(\theta)
 \right)\left(1-\frac{1}{2\pi I_0^2(2Kr)}\frac{1}{q_{\psi+\gep }(\theta)}\right)\dd \theta \, =\, O(\gep ^3)\, ,
\end{equation}
where $\cH$ is the primitive of $h$ such that $\int_\bbS \frac{\cH}{q_\psi}=0$. 
At this point 
we expand also  $q_{\psi+\gep }$ with respect to $\gep $ and, using $\gep=O(\Vert h \Vert_{-1})$, the parity of $q_\psi( \cdot+ \psi)$ and Remark~\ref{rem:compute}, we get to
\begin{equation}
 ( h,q'_\psi)_{-1,1/q_\psi}+\gep (q'_\psi,q'_\psi)_{-1,1/q_\psi} +\gep \frac{1}{2\pi I^2_0(2Kr)}( h,(\log q)'')_{-1,1/q_\psi} \, =\, 
 O\left(\Vert h \Vert_{-1} ^3\right)\, .
\end{equation}
Now it suffices to solve this equation for $\gep$ and perform one last Taylor expansion. 
\end{proof}

\section{spectral estimates}
\label{sec:appB}

The aim of this section is to find approximations of  the eigenvalues and eigenfunctions of the operators $L_\psi$
for large eigenvalues. In such a regime we expect the Laplacian to dominate and the spectrum of $L_\psi$
 should get close to the one of the Laplacian (as long as we deal with large eigenvalues). These are standard estimates, developed for example in
  \cite{cf:Nai} that we follow, but we could not find in the literature the result for the  non-local operators we consider. Without loss of generality, we can focus on  $L_0$. We have 
\begin{equation}
\label{eq:def Lq}
 L_0 u \, =\, \frac12 u'' - (uJ*q_0+q_0J*u)'\, =\, \frac12 u'' -(J*q_0) u'- (J*q_0')u-q_0'J*u - q_0J'*u\, .
\end{equation}
We make a change of variable to get rid of the coefficient of order $1$: if we define
\begin{equation}
\label{eq:change variable Lq}
 u\, = \, =\, \sqrt{q_0}y\, ,
\end{equation}
and we observe  that $\sqrt{q}= e^{\tilde J*q_0}$, with $\tilde J(\theta):=K\cos(\theta)$,
 then  we get
\begin{equation}
 u'\, =\, \sqrt{q_0}y'+(J*q_0)\sqrt{q_0}y\, ,
\end{equation}
\begin{equation}
 u''\, =\, \sqrt{q_0}y''+ 2(J*q_0)\sqrt{q_0}y'+(J*q_0')\sqrt{q_0}y+(J*q_0)^2\sqrt{q_0}y\, ,
\end{equation}
and these two last equations together with \eqref{eq:def Lq} give
\begin{multline}
L_{q_0}\sqrt{q_0}y, =\,  \frac12  [\sqrt{q_0}y''+ 2(J*q_0)\sqrt{q_0}y'+(J*q_0')\sqrt{q_0}y+(J*q_0)^2\sqrt{q_0}y]\\ -(J*q_0)[ \sqrt{q_0}y'+(J*q_0)\sqrt{q_0}y]
-(J*q_0')\sqrt{q_0}y-q_0'J*(\sqrt{q_0}y) - q_0J'*(\sqrt{q_0}y)
\end{multline}
which leads, after simplification, to the new operator
\begin{equation}
 \tilde L y\, :=\, \frac12 y''-m(y)\, , 
\end{equation}
where we have set
\begin{equation}
 m(y)\, :=\, \frac12 ((J*q_0)^2 +J*q_0')y+\frac{q_0'}{\sqrt{q_0}}J*(\sqrt{q_0}y)+\sqrt{q_0}J'*(\sqrt{q_0}y)\, .
\end{equation}
Of course $m(y)$ is a function and when we want to make explicit the $\theta$-dependence we use
$m_\theta(y)$.
Since the operator $L_0$ is negative, we are interested in couples $(\rho,y)$ solution of
\begin{equation}
 \tilde L y\, =\,-\rho^2 y \, , 
\end{equation}
where $\rho$ is a positive real number. The method of variation of the parameters shows that such solutions exist (for all $\rho>0$ if we do not restrict the study to the $2\pi$-periodic eigenfunctions of $\tilde l$) and are of the form
\begin{equation}
\label{eq:variation parameter}
 y(\theta)\, =\, c_1 e^{\sqrt{2}\rho i\theta}+c_2 e^{-\sqrt{2}\rho i \theta} - \frac{1}{\sqrt{2}\rho} \int_0^\theta G(\theta,\theta',\rho)m_{\theta'}(y) \dd\theta'
\end{equation}
where
\begin{equation}
 G(\theta,\theta',\rho)\, =\, i e^{\sqrt{2}\rho i(\theta-\theta')} - i e^{-\sqrt{2}\rho i (\theta-\theta')}\, .
\end{equation}
We define $y_1$ the solution such that $c_1=1$, $c_2=0$, and $y_2$ the solution such that $c_1=0$, $c_2=1$. In what
follows we start by getting a first estimate of the eigenfunctions $y_1$ and $y_2$ with respect to $\rho\longrightarrow \infty$. This estimate  implies a first estimate of the eigenvalue $-\gl=-\rho^2$, and this  leads to a new approximation of the eigenfunctions, and thus a new approximation of $-\gl$. This procedure can be repeated recursively, but for us two steps will suffice.

\medskip

\begin{lemma}
\label{lem:first order rho}
For each $\gl>0$, there exist $y_1$ and $y_2$ independent (non necessarily periodic) eigenfunctions of $\tilde L$ associated to $-\gl$ such that (recall that $\rho$=$\sqrt{\gl}$):
 \begin{equation}
  \label{eq:first order y1 rho}
 y_1(\theta,\rho)\, =\, e^{\sqrt{2}\rho i\theta} + O\left(\frac{1}{\rho}\right)\, ,
 \end{equation}
\begin{equation}
 \label{eq:first order yprime1 rho}
 y_2(\theta,\rho)\, =\, e^{-\sqrt{2}\rho i\theta} + O\left(\frac{1}{\rho}\right)\, ,
\end{equation}
\begin{equation}
 y_1'(\theta,\rho)\, =\, \sqrt{2}\rho i e^{\sqrt{2}\rho i\theta} + O(1)\, ,
\end{equation}
\begin{equation}
 y_2'(\theta,\rho)\, =\, -\sqrt{2}\rho i e^{-\sqrt{2}\rho i\theta} + O(1)\, ,
\end{equation}
where $\theta\in [0, 2\pi]$ and $O(\cdot)$ is as $\rho$ tends to infinity
(and we stress that here and below the $O(\cdot)$ term does not depend on $\theta$ or,
equivalently, it is uniform in $\theta \in [0, 2\pi]$).
\end{lemma}
\medskip

\begin{proof}
We prove the result for $y_1$. The proof for $y_2$ is similar. We define
\begin{equation}
\label{eq:def A0}
 A_0(\theta,\theta',v)\, =\, - \frac{1}{\sqrt{2}\rho} G(\theta,\theta',\rho)m_{\theta'}(v)\ind_{\theta'<\theta}
\end{equation}
so that for $\theta\in [0,2\pi]$,
\begin{equation}
\label{eq:y0 with A0}
 y_1(\theta)\, =\,  e^{\sqrt{2}\rho i\theta} + \int_0^{2\pi} A_0(\theta,\theta',y_1) \dd\theta'\, .
\end{equation}
The expression for $y_1$; cf. \eqref{eq:variation parameter}, can be iterated arbitrarily many times and it leads 
to a series expression for $y_1$, at least for $\rho$ sufficiently large. To see this set   $f_0(\theta):=e^{\sqrt{2}\rho i\theta}$ and observe that
\begin{multline}
\label{eq:serie y1}
 y_1(\theta_0)\, =\, f_0(\theta_0)+\sum_{j=1}^{m}\int_0^{2\pi}\cdots\int_0^{2\pi} A_0(\theta_0,\theta_1,A_0(\theta_1,\theta_2,\cdots A_0(\theta_{i-1},\theta_{i},f_0)\cdots))\dd\theta_1\cdots\dd\theta_m \\
+\int_0^{2\pi}\int_0^{2\pi}\cdots\int_0^{2\pi} A_0(\theta_0,\theta_1,A_0(\theta_1,\theta_2,\cdots A_0(\theta_{m},\theta_{m+1},y_1)\cdots))\dd\theta_1\cdots\dd\theta_{m+1} \, .
\end{multline}
One directly verifies that there exists $C=C(K)$ such that for  $\theta,\, \theta'\in [0,2\pi]$, 
\begin{equation}
\label{bound A0}
 |A_0(\theta,\theta',v)|\, \leq\, \frac{C}{\rho}\Vert v \Vert\, , 
\end{equation}
where $\Vert v \Vert:=\sup_{\theta\in[0,2\pi]}|v(\theta)|$.
From \eqref{eq:serie y1} and using $\Vert f_0(\cdot) \vert \equiv 1$ we see that
\begin{equation}
\Vert y_1\Vert \le \, 1+ \sum_{j=1}^m \left( \frac{2\pi C}{\rho}\right)^m 
+\left( \frac{2\pi C}{\rho}\right)^{m +1} \Vert y_1\Vert\, ,
\end{equation}
so for $\rho > 2\pi C$ we see that $\Vert y_1\Vert< \infty$ and we have a series expression for $y_1$, from which
 we directly obtain \eqref{eq:first order y1 rho}.

To deal with $y_1'$ we take the derivative of both sides of
 \eqref{eq:variation parameter} with $c_1=1$ and  $c_2=0$, so that
\begin{equation}
 y'_1(\theta)\, =\, \sqrt{2} \rho i e^{\sqrt{2}\rho i \theta} - \frac{1}{\sqrt{2}\rho}\int_0^\theta \partial_\theta G(\theta,\theta',\rho)m_{\theta'}(y_1)\dd\theta'\, .
\end{equation}
We define the new kernel 
\begin{equation}
 A_1(\theta,\theta',v)\, :=\, -\frac{1}{\sqrt{2}\rho^2}\partial_\theta G(\theta,\theta',\rho)m_{\theta'}(v)\ind_{\theta'<\theta}\, ,
\end{equation}
so we can write
\begin{equation}
 \frac{1}{\rho} y'_1(\theta) \, =\, \sqrt{2} i e^{\sqrt{2}\rho i \theta} + \int_0^{2\pi} A_1(\theta,\theta',y_1) \dd\theta'\, .
\end{equation}
Also $A_1$ verifies 
\begin{equation}
\label{bound A1}
 |A_1(\theta,\theta',v)|\, \leq\, \frac{C}{\rho}\sup_{\theta\in[0,2\pi]}|v(\theta)|\, , 
\end{equation}
for a suitable $C=C(K)$ 
and the same argument as above gives 
\begin{equation}
  \frac{1}{\rho} y'_1(\theta) \, =\, \sqrt{2} i e^{\sqrt{2}\rho i \theta} + O\left(\frac1\rho\right)\, , 
\end{equation}
which is equivalent to \eqref{eq:first order yprime1 rho}.
\end{proof}

\medskip

\begin{lemma}
\label{lem:gl square root}
There exists $l_0\in\bbN$ such that for all $p\in \bbN$ the eigenvalues of $L_0$ satisfy
\begin{equation}
 \gl_{l_0+2p}\, =\, \frac{p^2}{2}+O(\sqrt{p})\, ,
\end{equation}
\begin{equation}
 \gl_{l_0+2p+1}\, =\, \frac{p^2}{2}+O(\sqrt{p})\, .
\end{equation}
\end{lemma}

\medskip

\begin{rem}
\label{rem:eigenvalues and eigenfunction expansion}
An immediate consequence of Lemma~\ref{lem:gl square root}
and of the basic properties of $L_0$ is that there exist $C>1$ such that for $j=0,1, \ldots$.
\begin{equation}
\frac{j^2}C \, \le \, \gl_{j} \, \le \, C j^2\, .
\end{equation}
\end{rem}

\medskip

\begin{proof}
Let $y_1$ and $y_2$ the eigenfunctions of $\tilde L$ given by Lemma \ref{lem:first order rho} associated to the eigenvalue $-\gl=-\rho^2$. As a linear combination of $y_1$ and $y_2$ is $2\pi$-periodic, the following determinant is equal to zero:
\begin{equation}
\label{eq:determinant exact}
 \left|
\begin{array}{cc}
 y_1(2\pi)-y_1(0) & y_2(2\pi)-y_2(0) \\
 y'_1(2\pi)-y'_1(0) &y'_2(2\pi)-y'_2(0) 
\end{array}
\right| \, =\, 0\, .
\end{equation}
Lemma \ref{lem:first order rho} implies
\begin{equation}
\label{eq:det periodicity eigenvector}
 \left|
\begin{array}{cc}
 e^{2\sqrt{2}\pi\rho i}-1+O\left(\frac1\rho\right) & e^{-2\sqrt{2}\pi\rho i}-1+O\left(\frac1\rho\right) \\
 \sqrt{2}\rho i (e^{2\sqrt{2}\pi\rho i}-1)+O(1) &-\sqrt{2}\rho i (e^{-2\sqrt{2}\pi\rho i}-1) +O(1)
\end{array}
\right| \, =\, 0\, ,
\end{equation}
and thus we get
\begin{equation}
 |e^{2\sqrt{2}\pi\rho i}-1|^2=O\left(\frac1\rho\right)\, .
\end{equation}
We deduce that there exits $k\in \bbN$ such that 
\begin{equation}
\label{eq:approx rho square root}
 \rho=\frac{k}{\sqrt{2}}+O\left(\frac{1}{\sqrt{k}}\right)\, .
\end{equation}
Reciprocally, all $\rho$ satisfying \eqref{eq:approx rho square root} satisfies \eqref{eq:det periodicity eigenvector}, so the Lemma follows.
\end{proof}

\medskip

\begin{proposition}
\label{prop:eigenvalues and eigenfunction expansion}
There exists $l_0\in \bbN$ such that for all $p\in \bbN$ the eigenvalues of $L_0$ satisfy
\begin{equation}
 \gl_{l_0+2p}\, =\, \frac{p^2}{2} -\frac{K^2r^2}{8}+O\left(\frac{1}{p}\right)\, ,
\end{equation}
\begin{equation}
\gl_{l_0+2p+1}\, =\, \frac{p^2}{2} -\frac{K^2r^2}{8}+O\left(\frac{1}{p}\right)\, ,
\end{equation}
and any eigenfunction of $L_0$ associated to $\gl_{l_0+2p}$ or $\gl_{l_0+2p+1}$ is, up to a correction of order $1/p^2$, a linear combination
 of the two functions $q_0^{1/2}v_{1,l_0+p}$ and $q_0^{1/2}v_{2,l_0+p}$, where 
\begin{equation}
\label{eq:2eqAB}
\begin{split}
v_{1,l_0+p}(\theta)\, &=\,\cos(p\theta)-\frac{\sin(p\theta)}{p}\left[\frac{Kr}{2}\sin(\theta)+\frac{K^2r^2}{8}\sin(2\theta)\right]\, ,
\\
  v_{2,l_0+p}(\theta)\, &=\, \sin(p\theta)+\frac{\cos(p\theta)}{p} \left[\frac{Kr}{2}\sin(\theta)+\frac{K^2r^2}{8}\sin(2\theta)\right]\, .
  \end{split}
\end{equation}
\end{proposition}

From Proposition~\ref{prop:eigenvalues and eigenfunction expansion} one can directly extract some important conclusions:
let us give them before the proof of the proposition. 
\medskip

\begin{cor}
\label{cor: unitary eigenfunction expansion}
There exists $l_0 \in \bbN$ such that for all $p\in \bbN$ and $\psi\in \bbS$, the unitary (in $H_{-1,1/q_{\psi}}$) eigenfunctions $e_{\psi,l_0+2p}$ and $e_{\psi,l_0+2p+1}$ of $L_\psi$ are up to a correction of order $1/p$ a bounded (with respect to $p$) linear combination of $\theta\mapsto pq_\psi^{1/2}(\theta) v_{1,l_0+p}(\theta-\psi)$ and $\theta\mapsto pq_\psi^{1/2}(\theta) v_{2,l_0+p}(\theta-\psi)$ (see Proposition~\ref{prop:eigenvalues and eigenfunction expansion} for the definition of $v_{1,l}$ and $v_{2,l}$).
\end{cor}

\medskip
\noindent
{\it Proof.} We set $\psi=0$ without loss of generality.
Proposition~\ref{prop:eigenvalues and eigenfunction expansion} tells us that the normalized eigenfunctions of $L_0$
can be written either as
\begin{equation}
\label{eq:eigenAB}
c_p \left(
\cos(p\theta)-\frac{\sin(p\theta)}{p}\left[\frac{Kr}{2}\sin(\theta)+\frac{K^2r^2}{8}\sin(2\theta)\right] + r_p(\theta) \right)
\end{equation}
where $r_p(\theta)= O(1/p^2)$ and $c_p$ is the normalizing constant,
or with the analogous expression coming from the second line in \eqref{eq:2eqAB} (but we will deal only
with \eqref{eq:eigenAB} because the other case is treated analogously). To estimate $c_p$
let us observe that the first two addends in \eqref{eq:eigenAB} are in $H^{-1}$ (since they are smooth,
it suffices to remark that their integral from $0$ to $2\pi$ is zero), so $r_p\in H_{-1}$, since
the eigenfunction is: of course $r_p$ is smooth, since the eigenfunction is. Now we claim that
the $H_{-1,1/q}$ norm of $\cos (p\cdot)$, that is the first addendum, is proportional to $1/p$, apart for a correction 
that is beyond all orders in $1/p$, while the norm of the two other terms is $O(1/p^2)$.
In fact if we set $u(\theta):= \cos(p\theta)$, then $\cU(\theta)= \sin (p\theta) /p$
so
\begin{equation}
\Vert u \Vert _{-1,1/q}\, =\, \frac 1p \sqrt{\int_\bbS \frac{1-\cos(2p\theta)}{2 q(\theta)}}\dd \theta\, .
\end{equation} 
If we use the standard estimate 
\begin{equation}
\label{eq:neglectable local terms 2}
 I_k(x)\,=\, \frac{1}{2\pi}\int_0^{2\pi}\cos(k\theta)e^{x\cos(\theta)}\dd\theta\, =\, \sum_{m=0}^\infty\frac{1}{m!\gG(m+k+1)}\left(\frac{x}{2}\right)^{2m+k}\, \leq\, \frac{C_x}{(k!)^{1/2}}\, .
\end{equation}
we readily see that
\begin{equation}
\int_\bbS \frac{\cos(2p\theta)}{ q(\theta)}\dd \theta\, = \, 
O\left( \frac 1 {\sqrt{(2p)!}}\right)\, .
\end{equation}
On the other hand
\begin{equation}
\int_\bbS \frac{1}{q(\theta)}\dd \theta\, =\, \left( 2\pi I_0(2Kr) \right)^2\,,
\end{equation}
so $\Vert u \Vert _{-1,1/q}$ is equal to $c(K) /p$, $c(K):=\sqrt{2}\pi I_0(2Kr)$, up to a correction
that decays faster than any power of $1/p$. 

For the second addendum it suffices to observe that it can be rewritten as
a linear combination of terms of  $\cos(p' \theta)$, with $\vert p-p'\vert=1$ and $2$.
But then the computation is very similar to the one that we have done for the first addendum
(or, easier, one can explicitly compute the $H_-1$ norm, without weight). Therefore this term
is $O(1/p^2)$.

For the third addendum we recall that $\vert r_p(\theta) \vert \le C /p^2$, so that if we set
$\cR (\theta):= \int_0^\theta r_p(\theta')\dd \theta'$, we have $\Vert \cR (\theta) \vert \le C \theta/p^2$.
Of course $\cR$ is not necessarily centered, but, by using Remark~\ref{rem:norecenter},
we see that $\Vert r_p \Vert_{-1}\le 2C^2 \pi^2 /p^2$. 

By collecting the estimates of the three addends we see that
\begin{equation}
c_p \, =\, c(K) p \left( 1+ O(1/p)\right)\, ,
\end{equation} 
and this completes the proof of Corollary~\ref{cor: unitary eigenfunction expansion}.
\qed

\medskip

By putting  Corollary \ref{cor: unitary eigenfunction expansion} and \eqref{eq:link e and f} together we obtain
\medskip

\begin{cor}
\label{cor:fj}
With $\{f_j\}_{j=0,1, \ldots}$ defined as in Appendix~\ref{sec:A_H}, we have
$\sup_j \Vert f_j' \Vert_\infty < \infty$ and $\sup_j \Vert f_j'' \Vert_\infty /j <\infty$. 
\end{cor}
\medskip

\noindent
{\it Proof.} From \eqref{eq:link e and f}, see also the discussion right after that, we see that
$f'_j= -\cE_j/q$ and $f''_j = -e_j/q + \cE_j  q' /q^2$. Where $e_j$ is the $j^{\mathrm{th}}$
(normalized) eigenvector. Taking into account the normalization, see proof of Corollary~\ref{cor: unitary eigenfunction expansion}, the claim is readily proven. 
\qed

\medskip

\noindent{\it Proof of Proposition~\ref{prop:eigenvalues and eigenfunction expansion}.}
Injecting \eqref{eq:first order y1 rho} in the integral term of \eqref{eq:y0 with A0} leads to 
\begin{multline}
\label{eq:second order y1 rho}
 y_1(\theta)\, =\, e^{\sqrt{2}\rho i\theta}-\frac{1}{\sqrt{2}\rho}\bigg[i e^{\sqrt{2}\rho i\theta}\int_0^\theta e^{-\sqrt{2}\rho i\theta'}m_{\theta'}(e^{\sqrt{2}\rho i\cdot})\dd\theta'\\
-i e^{-\sqrt{2}\rho i\theta}\int_0^\theta e^{\sqrt{2}\rho i\theta'}m_{\theta'}(e^{\sqrt{2}\rho i\cdot})\dd\theta'\bigg]
+O\left(\frac{1}{\rho^2}\right) \, .
\end{multline}
Similarly, we obtain
\begin{multline}
\label{eq:second order y1 rho2}
 y_1'(\theta)\, =\, \sqrt{2}\rho ie^{\sqrt{2}\rho i\theta}+\bigg[ e^{\sqrt{2}\rho i\theta}\int_0^\theta e^{-\sqrt{2}\rho i\theta'}m_{\theta'}(e^{\sqrt{2}\rho i\cdot})\dd\theta'\\+e^{-\sqrt{2}\rho i\theta}\int_0^\theta e^{\sqrt{2}\rho i\theta'}m_{\theta'}(e^{\sqrt{2}\rho i\cdot})\dd\theta'\bigg]
+O\left(\frac{1}{\rho}\right)\, ,
\end{multline}
and similar expressions for $y_2$ and $y'_2$, which actually are just the complex conjugate of $y_1$
and $y_1'$.
We define
\begin{equation}
 H_1\, =\, \int_0^{2\pi} e^{-\sqrt{2}\rho i\theta'}m_{\theta'}(e^{\sqrt{2}\rho i\cdot})\dd\theta'\, ,
 \ \ \ 
 H_2\, =\, \int_0^{2\pi} e^{\sqrt{2}\rho i\theta'}m_{\theta'}(e^{\sqrt{2}\rho i\cdot})\dd\theta' \ \ \text{ and } \ \
 \gO\, =\, e^{2\sqrt{2}\pi\rho i}\, .
\end{equation}
With 
the higher estimates \eqref{eq:second order y1 rho} and \eqref{eq:second order y1 rho2}, 
we see that \eqref{eq:determinant exact} becomes
\begin{equation}
\label{eq:determinant order 2}
\left|
\begin{array}{cc}
 \gO-1 -\frac{1}{\sqrt{2}\rho}\left[i \gO H_1-i \bar{\gO} H_2\right] +O\left(\frac{1}{\rho^2}\right) &  
 \bar{\gO}-1 -\frac{1}{\sqrt{2}\rho}\left[-i \bar{\gO}\bar{H_1}+i \gO \bar{H_2}\right] +O\left(\frac{1}{\rho^2}\right) \\
\gO-1-\frac{1}{\sqrt{2}\rho}\left[i \gO H_1+i \bar{\gO} H_2\right] +O\left(\frac{1}{\rho^2}\right) &  
 -\bar{\gO}+1 -\frac{1}{\sqrt{2}\rho}\left[i \bar{\gO}\bar{H_1}+i \gO \bar{H_2}\right] +O\left(\frac{1}{\rho^2}\right)                                                    
\end{array}
\right|\, =\, 0\, , 
\end{equation}
which implies
\begin{equation}
\label{eq:expansion determinant im}
|\gO-1|^2-\frac{\sqrt{2}}{\rho}\Im( (\gO-1)H_1)\, =\, O\left(\frac{1}{\rho^2}\right)\, .
\end{equation}
We now use the expansion of $\rho$ given by \eqref{eq:approx rho square root}. In particular, the $O(1/\rho^2)$ above becomes a $O(1/k^2)$. The second term of the left hand side above is of order $1/k^2$. In fact, we get the first order of $H_1$ :
\begin{equation}
 H_1\, =\, \int_0^{2\pi} e^{-k i\theta'}m_{\theta'}(e^{k i\cdot})\dd\theta'+O\left(\frac{1}{\sqrt{k}}\right)\, ,
\end{equation}
where the non local terms in the integral are negligible, since we have 
\begin{equation}
\label{eq:neglectable local terms}
 J*(\sqrt{q_0}e^{ki\cdot})(\theta)\, =\, \frac{iK}{2(2\pi I_0(2Kr))^{1/2}}(e^{i\theta}I_{k-1}(Kr)-e^{-i\theta}I_{k+1}(Kr))
\end{equation}
and we can apply \eqref{eq:neglectable local terms 2}.
A similar bound apply for $J'*(\sqrt{q_0}e^{ki\cdot})$. So it remains the (real !!) first order (remark that $J*q_0(\cdot)
=-Kr\sin(\cdot)$):
\begin{equation}
\label{eq:expansion H1}
 H_1 =\, \int_0^{2\pi} \frac12 ((J*q_0)^2 +J*q_0')(\theta')\dd\theta'+O\left(\frac{1}{\sqrt{k}}\right)\, =\, \frac{\pi K^2r^2}{2}+O\left(\frac{1}{\sqrt{k}}\right)\, .
\end{equation}
But since (using \eqref{eq:approx rho square root})
\begin{equation}
\label{eq:expansion expo}
 \gO-1\, =\, 2\pi i (\sqrt{2} \rho-k)+O\left(\frac{1}{k}\right)\, ,
\end{equation}
where the first term of the right hand side is of order $1/\sqrt{k}$, we have improved the result of Lemma \ref{lem:gl square root}, since using \eqref{eq:expansion determinant im}, \eqref{eq:expansion H1}, \eqref{eq:expansion expo} and \eqref{eq:approx rho square root} we obtain
\begin{equation}
 |e^{2\sqrt{2}\pi\rho i}-1|^2-\frac{2\pi^2K^2r^2}{k}(\sqrt{2}\rho-k)=O\left(\frac{1}{k^2}\right)
\end{equation}
which implies
\begin{equation}
\label{eq:expansion rho 1 over k}
 \sqrt{2}\rho\, =\, k +O\left(\frac1k\right)\, . 
\end{equation}
Taking \eqref{eq:expansion rho 1 over k} into account,  \eqref{eq:determinant order 2}  yields
\begin{equation}
\label{eq:det order 3}
 |\gO-1|^2- \frac{2}{k}\Im( (\gO-1)H_1)+\frac{1}{k^2}(|H_1|^2-|H_2|^2)\, =\, O\left(\frac{1}{k^3}\right)\, .
\end{equation}
The non local terms in $H_2$ are negligible as for $H_1$ (see above) and a direct calculation shows that the local terms are of order $1/k$, so from \eqref{eq:det order 3}, \eqref{eq:expansion H1} and \eqref{eq:expansion expo} we get
\begin{equation}
(\sqrt{2}\rho-k)^2 -\frac{K^2r^2}{2k}(\sqrt{2}\rho-k)+\frac{K^4r^4}{16k^2}\, =\, \left(\sqrt{2}\rho-k-\frac{K^2r^2}{4k}\right)^2\, =\, O\left(\frac{1}{k^3}\right)\, ,
\end{equation}
which implies
\begin{equation}
\label{eq:second order rho plusmoins}
 \sqrt{2}\rho\, =\, k+\frac{K^2r^2}{4}\frac1k +O\left(\frac{1}{k^{3/2}}\right) \, .
\end{equation}
We now go further in the expansion to prove that the $O(1/k^{3/2})$ in \eqref{eq:second order rho plusmoins} is in fact a $O(1/k^2)$.
Using \eqref{eq:serie y1}, we get the the second order expansion of $y_1$ (recall \eqref{eq:def A0} and $f_0=e^{\sqrt{2}\rho i\cdot}$)
\begin{equation}
\label{eq:expansion y1 rho cube}
 y_1(2\pi)\, =\,\gO+\int_0^{2\pi}A_0(2\pi,\theta_1,f_0)\dd \theta_1 +\int_0^{2\pi}\int_0^{2\pi} A_0(2\pi,\theta_1,A_0(\theta_1,\theta_2,f_0))\dd\theta_1\dd\theta_2+O\left(\frac{1}{\rho^3}\right)\, .
\end{equation}
From \eqref{eq:expansion rho 1 over k}, we deduce
\begin{multline}
 \int_0^{2\pi}A_0(\theta_1,\theta_2,f_0)\dd\theta_2\, =\, -\frac{1}{\sqrt{2}\rho}\Bigg[ie^{\sqrt{2}\rho i\theta_1}\int_0^{\theta_1}e^{-\sqrt{2}\rho i\theta_2}m_{\theta_2}(e^{\sqrt{2}\rho i \cdot})\dd\theta_2\\
-ie^{-\sqrt{2}\rho i\theta_1}\int_0^{\theta_1}e^{\sqrt{2}\rho i\theta_2}m_{\theta_2}(e^{\sqrt{2}\rho i \cdot})\dd\theta_2 \Bigg]\\
=\, -\frac{i}{k}\Bigg[e^{ki\theta_1}\int_0^{\theta_1}e^{-ki\theta_2}m_{\theta_2}(e^{ki\cdot})\dd\theta_2-e^{-ki\theta_1}\int_0^{\theta_1}e^{ki\theta_2}m_{\theta_2}(e^{ki\cdot})\dd\theta_2\Bigg]+O\left(\frac{1}{k^2}\right)\, ,
\end{multline}
and since the non local terms are negligible  (see \eqref{eq:neglectable local terms}), we get
\begin{equation}
\label{eq:expansion A0}
  \int_0^{2\pi}A_0(\theta_1,\theta_2,f_0)\dd\theta_2\, =\,\frac{iKre^{ki\theta_1}}{2k}\left(\sin \theta_1 +\frac{Kr}{4}\sin(2\theta_1)-\frac{Kr}{2}\theta_1\right)+O\left(\frac{1}{k^2}\right)\, .
\end{equation}
We deduce the following expansion for the third term of the right hand side of \eqref{eq:expansion y1 rho cube}:
\begin{multline}
\label{eq:expansion A0A0}
 \int_0^{2\pi}\int_0^{2\pi}A_0(2\pi,\theta_1,A_0(\theta_1,\theta_2,f_0))\dd\theta_1\dd\theta_2\, 
 =
 \\ \frac{Kr}{2k^2}\Bigg(\gO\int_0^{2\pi} e^{-ki\theta_1}m_{\theta_1}\Bigg[e^{ki\cdot}\Big(\sin\cdot+\frac{Kr}{4}\sin(2\cdot)-\frac{Kr}{2}\cdot\Big)\Bigg]\dd\theta_1 
 \\
     -\bar{\gO}\int_0^{2\pi} e^{ki\theta_1}m_{\theta_1}\Bigg[e^{ki\cdot}\Big(\sin\cdot+\frac{Kr}{4}\sin(2\cdot)-\frac{Kr}{2}\cdot\Big)\Bigg]\dd\theta_1 \Bigg)+O\left(\frac{1}{k^3}\right)\, .
\end{multline}
Using similar arguments as before, we get to
\begin{equation}
 \int_0^{2\pi} e^{-ki\theta_1}m_{\theta_1}\left(e^{ki\cdot}\Big(\sin(\cdot)+\frac{Kr}{4}\sin(2\cdot)\Big)\right)\dd\theta_1\, =\, O\left(\frac{1}{k^3}\right)\, ,
\end{equation}
\begin{equation}
 \int_0^{2\pi} e^{ki\theta_1}m_{\theta_1}\left(
 e^{ki\cdot}\Big(\sin(\cdot)+\frac{Kr}{4}\sin(2\cdot)\Big)\right)\dd\theta_1\, =\, O\left(\frac{1}{k^3}\right)\, .
\end{equation}
Moreover, the non local terms of $m_{\theta_1}(e^{ki\cdot}\cdot)$ are of order $1/k$. In fact, these non local terms are finite sums of the form
\begin{equation}
\label{eq:non local + linear}
 \int_0^{2\pi}e^{mi\theta}\sqrt{q_0}(\theta)\theta\dd\theta\, ,
\end{equation}
where $|m|$ is included in $[k-1,k+1]$, and it is easy to see that since the Fourier coefficients of $\sqrt{q_0}$ decay very quickly (see \eqref{eq:neglectable local terms 2}), \eqref{eq:non local + linear} is of order $1/k$. So \eqref{eq:expansion A0A0} becomes
\begin{equation}
 \int_0^{2\pi}\int_0^{2\pi}A_0(2\pi,\theta_1,A_0(\theta_1,\theta_2,f_0))\dd\theta_1\dd\theta_2\, =\, -\frac{K^4r^4\pi^2}{8k^2}\gO +\left(\frac{1}{k^3}\right)\, ,
\end{equation}
and we deduce from \eqref{eq:expansion y1 rho cube}
\begin{equation}
 y_1(2\pi)-y_1(0)\, =\, \gO-1-\frac{i}{k}(\gO H_1 -\bar{\gO}H_2)-\frac{K^4r^4\pi^2}{8k^2}\gO +O\left(\frac{1}{k^3}\right)\, .
\end{equation}
Similarly, we obtain
\begin{equation}
 \frac{y'_1(2\pi)-y'_1(0)}{\sqrt{2}\rho i}\, =\, \gO-1-\frac{i}{k}(\gO H_1+i\bar{\gO}H_2)-\frac{K^2r^2\pi^2}{8k^2}\gO+ O\left(\frac{1}{k^3}\right)\, .
\end{equation}
Using these new estimates, \eqref{eq:determinant order 2} becomes
\begin{equation}
\label{eq:determinant order 3}
\left|
\begin{array}{cc}
 \gO-1 -\frac{1}{\sqrt{2}\rho}\left[i \gO H_1-i \bar{\gO} H_2\right]  &  
 \bar{\gO}-1 -\frac{1}{\sqrt{2}\rho}\left[-i \bar{\gO}\bar{H_1}+i \gO \bar{H_2}\right]  \\
 -\frac{K^4r^4\pi^2}{8k^2}\gO +O\left(\frac{1}{k^3}\right) & -\frac{K^4r^4\pi^2}{8k^2}\bar{\gO} +O\left(\frac{1}{k^3}\right) \\
\gO-1-\frac{1}{\sqrt{2}\rho}\left[i \gO H_1+i \bar{\gO} H_2\right] &  
 -\bar{\gO}+1 -\frac{1}{\sqrt{2}\rho}\left[i \bar{\gO}\bar{H_1}+i \gO \bar{H_2}\right] \\
  -\frac{K^4r^4\pi^2}{8k^2}\gO +O\left(\frac{1}{k^3}\right) & +\frac{K^4r^4\pi^2}{8k^2}\bar{\gO}+O\left(\frac{1}{k^3}\right)                                                 
\end{array}
\right|\, =\, 0\, ,
\end{equation}
which leads to
\begin{multline}
\label{eq:det order 4}
 |\gO-1|^2- \frac{\sqrt{2}}{\rho}\Im( (\gO-1)H_1)+\frac{1}{k^2}(|H_1|^2-|H_2|^2)+\frac{K^4r^4\pi^2}{8k^2}\left(4-2\gO-2\bar{\gO}\right)\\
+\frac{K^4r^4\pi^2}{2k^3}\Im(H_1) \, =O\, \left(\frac{1}{k^4}\right)\, .
\end{multline}
The last term of \eqref{eq:det order 4} is of order $1/k^4$ since using \eqref{eq:expansion rho 1 over k} we get
\begin{equation}
\label{eq:expansion gO order 2}
 \gO\, =\, 1+i 2\pi (\sqrt{2}\rho-k)+O\left(\frac{1}{k^2}\right)\, .
\end{equation}
Moreover  using \eqref{eq:expansion rho 1 over k} we have
\begin{multline}
\label{eq:expansion H1 k2}
 H_1\, =\, \int_0^{2\pi}e^{-ki\theta} m_{\theta}(e^{ki\cdot})\dd\theta +i(\sqrt{2}\rho-k)\Bigg(-\int_0^{2\pi}e^{ki\theta}\theta m_\theta(e^{ki\cdot})\dd\theta \\
+\int_0^{2\pi}e^{ki\theta}m_\theta(e^{ki\cdot}\cdot)\dd\theta\Bigg)+O\left(\frac{1}{k^2}\right)\, .
\end{multline}
As before, the non local terms of $m_\theta(e^{ki\cdot}\cdot)$ are of order $1/k$, so the last two integrals in \eqref{eq:expansion H1 k2} are equal up to a correction of order $1/k$, and thus (recall \eqref{eq:expansion H1} for the first order term), using \eqref{eq:expansion rho 1 over k},
\begin{equation}
 H_1\, =\, \frac{\pi K^2r^2}{2}+O\left(\frac{1}{k^2}\right)\, .
\end{equation}
We deduce that the first term of the second row of \eqref{eq:det order 4} is of order $1/k^4$, and that, using \eqref{eq:expansion gO order 2},
\begin{equation}
\frac{\sqrt{2}}{\rho}\Im((\gO-1)H_1)\, =\, \frac{2\pi^2K^2r^2}{k}(\sqrt{2}\rho-k)+O\left(\frac{1}{k^4}\right)\, ,
\end{equation}
and
\begin{equation}
 \frac{1}{k^2}|H_1|^2\, =\, \frac{\pi^2 K^4 r^4}{4k^2}+O\left(\frac{1}{k^4}\right)\, .
\end{equation}
Since $|H_2|$ is of order $1/k$ and that \eqref{eq:expansion gO order 2} implies
\begin{equation}
 |\gO-1|^2\, =\, 4\pi^2(\sqrt{2}-\rho)^2+O\left(\frac{1}{k^4}\right)\, ,
\end{equation}
\eqref{eq:det order 4} becomes
\begin{equation}
(\sqrt{2}\rho-k)^2 -\frac{K^2r^2}{2k}(\sqrt{2}\rho-k)+\frac{K^4r^4}{16k^2}\, =\, O\left(\frac{1}{k^4}\right)\, ,
\end{equation}
and we deduce
\begin{equation}
\label{eq:second order rho k2}
 \sqrt{2}\rho\, =\, k+\frac{K^2r^2}{4}\frac{1}{k}+O\left(\frac{1}{k^2}\right)\, .
\end{equation}
Now we are able to get a second expansion of the eigenvectors: using \eqref{eq:second order y1 rho}, \eqref{eq:expansion A0} and \eqref{eq:second order rho k2}, we get the following expansion for $y_1$
\begin{equation}
 y_1(\theta)\, =\, e^{ki\theta}\left(1+\frac{Kri}{2k}\sin(\theta)+\frac{K^2r^2i}{8k}\sin(2\theta)\right)+O\left(\frac{1}{k^2}\right)\,
\end{equation}
and $y_2$ is the complex conjugate. So if we define $w_{1}$ and $w_{2}$ the real and imaginary parts, we get
\begin{align}
w_{1}(\theta)\, =\,& \cos(k\theta)-\frac{\sin(k\theta)}{k}\left(\frac{Kr}{2}\sin\theta +\frac{K^2r^2}{8}\sin(2\theta)\right)+O\left(\frac{1}{k^2}\right)\, , \\
w_{2}(\theta)\, =\,& \sin(k\theta)+\frac{\cos(k\theta)}{k}\left(\frac{Kr}{2}\sin\theta +\frac{K^2r^2}{8}\sin(2\theta)\right)+O\left(\frac{1}{k^2}\right) \, .
\end{align}
Therefore the proof of Proposition \ref{prop:eigenvalues and eigenfunction expansion} is complete.
\qed

\section*{Acknowledgments}
G. G. acknowledges the support of ANR, grants SHEPI and ManDy, and 
the support of the Petronio Fellowship Fund at the Institute for Advanced Study
(Princeton, NJ) where part of this research has been conducted. 
C. P. acknowledges the support of FSMP.

\end{document}